\documentclass[openright,a4paper,american,11pt, reqno]{amsart}

\usepackage[latin1]{inputenc}

\usepackage{amsmath}
\usepackage{amssymb}
\usepackage{hyperref}
\usepackage{babel}
\usepackage{tikz}
\usepackage{pgf,tikz,pgfplots}
\usepackage{graphicx}
\usepackage{tkz-graph}

\usetikzlibrary{arrows}

\usetikzlibrary{calc}
\usepackage{amsfonts}
\input xy
\xyoption{all}
\newcommand{\ZZ}{{\mathbb Z}}
\newcommand{\Z}{{\mathbb Z}}

\newcommand{\CC}{{\mathbb C}}

\newcommand{\R}{{\mathbb R}}

\newcommand{\TT}{{\mathbb T}}
\newcommand{\T}{{\mathbb T}}
\newcommand{\TP}{{\mathbb{T}P}}
\newcommand{\GL}{{\text{GL}}}
\newcommand{\CP}{{\mathbb{C}P}}

\newcommand{\A}{{\mathcal A}}

\newcommand{\C}{{\mathcal C}}

\newcommand{\F}{{\mathcal F}}
\newcommand{\X}{{\mathcal X}}

\newcommand{\M}{{\mathcal M}}

\renewcommand{\P}{{\mathcal P}}
\renewcommand{\M}{{\mathcal M}}
\renewcommand{\L}{{\mathcal L}}
\newcommand{\Ln}{{\mathcal L}}
\newcommand{\V}{{\mathcal V}}

\newcommand{\Log}{\text{Log}}
\newcommand{\Ed}{\text{Edge}}
\renewcommand{\div}{\text{div}}
\newcommand{\val}{\text{val}}
\newcommand{\trp}{\text{Trop}}

\newtheorem{introthm}{Theorem}
\newtheorem{introproposition}[introthm]{Proposition}
\newtheorem{introconj}[introthm]{Conjecture}

\newtheorem{thm}{Theorem}[section]

\newtheorem{defn}[thm]{Definition}
\newtheorem{definition}[thm]{Definition}
\newtheorem{prop}[thm]{Proposition}
\newtheorem{proposition}[thm]{Proposition}
\newtheorem{lemma}[thm]{Lemma}
\newtheorem{cor}[thm]{Corollary}
\newtheorem{corollary}[thm]{Corollary}
\newtheorem{conj}[thm]{Conjecture}
          {\theoremstyle{definition}
\newtheorem{rem}[thm]{Remark}}
          {\theoremstyle{definition}

\newtheorem{example}[thm]{Example}
\newtheorem{construction}[thm]{Construction}}

 \numberwithin{equation}{section}

\tikzset{%
  add/.style args={#1 and #2}{to path={%
 ($(\tikztostart)!-#1!(\tikztotarget)$)--($(\tikztotarget)!-#2!(\tikztostart)$)%
  \tikztonodes}}
} 

\newcommand{\comment}[1]{}

\begin{document}
\title{Tropical surfaces}
\author[Kristin Shaw]{Kristin Shaw}
\address{Kristin Shaw, 
Technische Universit\"at
Berlin, MA 6-2, 10623 Berlin, Germany.}
\email{shaw@math.tu-berlin.de} 
\address{}
\email{}

\thanks{This research is supported by an Alexander von Humboldt postdoctoral fellowship.}


\begin{abstract}
We present tools and definitions to study abstract tropical 
manifolds in dimension $2$, which we call simply tropical surfaces.  This includes 
explicit descriptions of intersection numbers of $1$-cycles, normal bundles to some curves and 
tropical Chern cycles and numbers. 
We provide a new method for constructing tropical surfaces, called the tropical sum, similar to the 
fiber sum of usual manifolds.

We prove a tropical adjunction formula for curves in compact tropical surfaces satisfying a local condition, a partial 
Castelnuovo-Enriques criterion for contracting $-1$-curves,  and also invariance of $(p, q)$-homology and Chow groups under tropical modification. 
Finally we prove a tropical version of Noether's formula for compact surfaces 
constructed from tropical toric surfaces by way of summations and tropical modifications. 
\end{abstract}
\maketitle

\tableofcontents

\section{Introduction}

Tropical geometry is a recent field of mathematics which can be loosely described as geometry over the tropical semi-field. 
This semi-field is also known as the max-plus semi-field and is  $\T := \R \cup \{-\infty\}$ equipped with the operations of maximum  and addition.

 Tropical spaces are polyhedral in nature and carry an integral affine structure. Sometimes these spaces arise via limits, degenerations, or skeletons of algebraic varieties  defined over a field. 
However, this is not always the case. A popular example  comes from the theory of matroids. The 
tropicalization of a linear space defined over a field  depends only on an associated (perhaps, valuated) matroid \cite{SturmPoly}. This led to the notion of  the Bergman fan of a matroid, which gives a representation of any matroid as a tropical linear space. 
In classical manifold theory, manifolds  are locally modeled on affine spaces, therefore we take matroidal fans as  the local models of tropical manifolds. 
Interestingly, whether a   matroid is representable over a field or not, the associated tropical linear space has many properties analogous to those of smooth spaces in classical geometry.  This eventually lead to properties of  tropical manifolds akin to properties of non-singular algebraic varieties. 

\bigskip
Abstract tropical manifolds of dimension $1$, tropical curves,  are simply metric graphs. 
A Riemann-Roch theorem for divisors on graphs given  by Baker and Norine \cite{BakerNorine1} and extended to  tropical curves \cite{GatKer}, \cite{MikZha:Jac}, sparked a study of linear series on graphs with many applications to the study of linear series on classical algebraic curves. These applications come usually by way of specialization theorems, which allow one to deduce inequalities between ranks of divisors and their tropicalizations \cite{Baker1}. 
For example, there is  a proof of Brill-Noether theorem  \cite{Pay3}, the study of gonality of curves \cite{AminiKool}  and even applications to the arithmetic geometry of curves 
\cite{KatRabZur}. On the combinatorial side, the theory of divisors on tropical curves is related to  chip-firing games, see \cite{Baker:ChipFire}.

Before this, the enumerative geometry of  tropical curves in $\R^2$ had powerful applications to real and complex enumerative geometry by way of Mikhalkin's correspondence theorem \cite{Mik1}.  Abstract tropical curves have also had applications to enumerative geometry via Hurzwitz numbers \cite{CJM}, \cite{Br13}, \cite{MarkwigRau:RealHurwitz}, as well as via the study of the intersection theory of their moduli spaces  \cite{MikPsi}, \cite{GKM}, \cite{MarkwigRau:TropicalDescendent}. 

The goal of this paper is to provide necessary definitions and tools to study higher dimensional  tropical manifolds, focusing on dimension $2$.
The definition of abstract tropical varieties from \cite{MIk3} has been adapted to consider tropical spaces locally modeled on matroidal fans.  These spaces will be called tropical manifolds.
Rather than adopt the term ``non-singular tropical variety", the name ``tropical manifold" is more fitting since there are examples of tropical manifolds which are models of  non-algebraic complex manifolds \cite{RuggShaw}. 

Before summarizing the main results, we would like to point out the work of  Cartwright on tropical complexes \cite{Cart:comp}, \cite{Cartwright:surfaces}. A tropical complex is a simplicial complex equipped with some extra numerical data.  One should think of the simplicial complex as the dual complex of a degeneration and the additional data encoding intersection numbers. 
Tropical complexes of dimension $1$, do not require this extra numerical information and can be considered simply as discrete graphs appearing in the work of Baker and Norine. 

Cartwright  proves specialization lemmas for ranks of divisors on tropical complexes. For tropical complexes of dimension $2$, there is an analogue of the Hodge Index Theorem on the tropical Neron-Severi group, as well as a tropical version of Noether's formula.
For tropical surfaces studied here,  there is also an intersection product on the $(1, 1)$-tropical homology group of a compact tropical surface, see \cite{MikZhaEig}, \cite{ShawTh}. The intersection product on the $(1, 1)$-tropical homology group does not necessarily satisfy the analogue of the Hodge Index Theorem \cite{ShawHomo}.  However, this does not imply that the tropical Hodge Index Theorem fails when the intersection pairing is restricted to the  Neron-Severi group of a tropical surface.

\bigskip 
We now summarize the outline and results of this paper. We begin by describing the local models of tropical manifolds; 
fan tropical linear spaces in $\R^N$ and in tropical affine space. A fan tropical linear space in $\R^N$ is given by a matroid and a choice of $\Z^N$-basis. Here we are only interested in the support of the fan, also called the coarse structure in \cite{Ard}. In general, it is possible to obtain the same tropical fan plane from non-isomorphic matroids, if we also choose different bases of $\Z^N$. However, in restricting to dimension $2$ we prove that this is not possible. 

\begin{introproposition}
The underlying matroid of a fan tropical plane in $\R^N$ is unique. 
\end{introproposition}

This however does not imply that the coordinates with respect to which a fan tropical linear space is equal to a Bergman fan of a matroid is unique, just that it is not possible to change the underlying matroid by changing the $\Z^N$-basis, see Example \ref{ex:delPezzo5}. 
There are known counter-examples to the above statement for fan  tropical linear spaces of dimension greater than $2$.

The rest of Section \ref{sec:prel} defines tropical modifications of fan tropical linear spaces, $1$-cycles in fan tropical planes, and also intersections of $1$-cycles in fan tropical planes using  the definitions from \cite{Br17}.  
Tropical intersection theory is, for the most part, defined locally and these definitions will be used when passing to general tropical surfaces. 

Section \ref{sec:SURF} switches from local to global considerations. 
Section \ref{sec:defsuf} gives the general  definition of tropical manifolds and Section \ref{sec:bdyArr} describes their boundaries. 
In Section \ref{sec:Cherncycles}, the definition of   Chern cycles of tropical manifolds are extended beyond the first Chern cycles.
 For tropical surfaces
we give a description of  $c_2(X)$ in terms of combinatorial data of 
 tropical fan planes.  We also show that the canonical cycle of a tropical manifold is balanced. 

\begin{introproposition}
The canonical cycle $K_X$ of a  tropical manifold $X$  is a balanced tropical cycle. 
\end{introproposition}

Section \ref{sec:int1cyc} describes how to pass from local intersection theory from Section \ref{sec:fancycles}
 to intersections of $1$-cycles of curves on surfaces. Also we describe the 
normal bundles to boundary curves, in order to define their self-intersections. This will also be used to define the tropical sum in Section \ref{sec:construction}.
At the boundary of a tropical manifold, intersection products are no longer defined on the cycle level, and must be defined up to equivalence. Here we consider only numerical equivalence for simplicity. 

There is also a tropical intersection theory based on Cartier divisors, see  \cite{MIk3}, \cite{AlRa1}, and the $2$ theories agree. 
Moreover, we show that on any tropical manifold there is an equivalence of Weil and Cartier divisors. This is a property of non-singular spaces in classical algebraic geometry. 

\begin{introthm}\label{thm:eqvCartWeilIntro}
Every codimension $1$ cycle  on a tropical manifold is a Cartier divisor. 
\end{introthm}

Finishing off Section \ref{sec:SURF}, we summarize the definitions of tropical rational equivalence from \cite{MIk3}, and also $(p, q)$-holomology from \cite{IKMZ}. Via a cycle map a $1$-cycle $A$  in a tropical surface produces a $(1, 1)$ homology cycle $[A]$.  
For consistency,  we show that the intersection of $1$-cycles in a tropical surface is compatible with their intersection as $(1, 1)$-cycles. It  follows that  rational equivalence respects the numerical  intersection pairing.

\begin{introproposition}
The intersection of $1$-cycles $A, B$ in a compact tropical surface $X$ is numerically equivalent to their intersection as $(1, 1)$-cycles.
\end{introproposition}

It is also the case that the intersection product on tropical cycles descends to rational equivalence, however we do not require this here so we do not include the proof.  Such statements have been proved in the non-compact setting, using a bounded version of rational equivalence  \cite{AllHamRau:RatEq}.

In Section \ref{sec:construction} we describe $2$ operations  to construct  tropical surfaces. The first is already familiar in  tropical  geometry and is known as tropical modification. This operation was introduced by Mikhalkin in \cite{MIk3} and  produces from a tropical variety $X$ and a tropical Cartier divisor, a new tropical variety  $\tilde{X}$.  Between the $2$ varieties there is a tropical morphism $\delta:\tilde{X} \to X$. 
We restrict to so-called locally degree $1$ modifications $\delta: \tilde{X} \to X$ 
which are along  codimension $1$-cycles satisfying an extra condition. This additional condition ensures that the  modification produces a tropical manifold.  The next proposition then follows  from Theorem \ref{thm:eqvCartWeilIntro}. 

\begin{introproposition}
Let $X$ be a tropical manifold, and $D \subset X$ a locally degree $1$, then  there is a tropical manifold $\tilde{X}$ such that $\delta: \tilde{X} \to X$ is the tropical modification along $D$. 
\end{introproposition}

In dimension $2$, we use the above proposition to prove a tropical version of the adjunction formula.  This formula relates the $1^{st}$ Betti number of the tropical curve $C$ in a surface $X$ to its self-intersection and intersection with the canonical cycle. 
Recall that the genus of  a tropical curve $C$ is $b_1(C)$. 

\begin{introthm}[Tropical Adjunction Formula]
Let $C $ be a locally degree $1$ tropical curve in a compact tropical surface $X$, then 
$$b_1(C) = \frac{K_X \cdot C + C^2}{2} + 1,$$
where $b_1(C)$ is the $1^{st}$ Betti number of $C$. 
\end{introthm}

We also show that   tropical manifolds $\tilde{X}$ and $X$ related by  a locally degree $1$ modification have the same tropical Chow groups and also $(p, q)$-homology groups. 

\begin{introthm}
Let $\delta: \tilde{X} \to X$ be a modification of tropical manifolds then the tropical Chow groups are isomorphic for all $k$
$$CH_{k}(\tilde{X}) \cong CH_k(\tilde{X}).$$ 

\end{introthm}

\begin{introthm}
Let $\delta: \tilde{X} \to X$ be a modification of tropical manifolds then their tropical $(p, q)$-homology groups are isomorphic for all $p, q$, 
 $$H_{p, q}(\tilde{X}) \cong H_{p, q}(X).$$
\end{introthm}

The second operation described in Section \ref{sec:construction} is analogous to the fibre sum of manifolds.   Given tropical surfaces $X_1$, $X_2$ each containing isomorphic boundary curves $C_1$, $C_2$ respectively, under suitable conditions we may glue the $2$ surfaces together to produce a new tropical surface $X_1 \# X_2$. In general, this sum relies on a choice of identification of neighborhoods of the curves $C_1$ in $X_1$  and $C_2$ in $X_2$. 

For example, together tropical  sums and modifications can be used to contract locally degree $1$ tropical rational curves of self-intersection $-1$  (Example \ref{ex:-1}). 

\begin{introthm}[Partial Castelnuovo-Enriques Criterion]

Let $E$ be a locally degree $1$ tropical curve in  a tropical surface $X$, such that $b_1(E) = 0$ and $E^{2} = -1$ then there exists a tropical modification $\delta: \tilde{X} \to X$,  another tropical surface $Y$, and a tropical morphism 
$$\pi: \tilde{X} \to Y$$ 
which  sends $\tilde{E}$ to a point $y$ and is   an isomorphism on $\tilde{X} \backslash \tilde{E} \to Y \backslash y$. 
\end{introthm}

Section \ref{sec:Noether} treats a  tropical analogue of Noether's formula for classical projective surfaces. Classically, this formula relates the holomorphic Euler characteristic $\sum_{q = 0}^2 (-1)^q h^{0, q}(\mathcal{X})$ with the Todd class of the surface $\frac{1}{12}(K^2_{\mathcal{X}} + c_2(\mathcal{X}))$. This formula is a special case of the Riemann-Roch formula for the trivial line bundle of a surface. 

To translate this tropically, firstly following \cite{IKMZ}, the ranks of the  $(p, q)$-homology  groups from \cite{IKMZ} play the role of Hodge numbers of tropical surfaces. 
When $p = 0$ the tropical $(p, q)$-homology groups correspond to the usual singular homology groups.  
Therefore, in the tropical version of Noether's formula, the  holomorphic Euler characteristic is replaced with the topological Euler characteristic of the tropical surface. The top tropical  Chern class $c_2(X)$ and also the square of the canonical class $K^2_X$ are defined in Section \ref{sec:Cherncycles}. 

\begin{introconj}[Tropical Noether's Formula]
A compact tropical surface $X$ satisfies 
$$\chi(X) = \frac{K_X^2 + c_2(X)}{12},$$
where $\chi(X)$ is the topological Euler characteristic of $X$. 
\end{introconj}

Proving a sequence of lemmas which relate $K_X^2$ and $c_2(X)$ for sums of surfaces and modificatons we prove Noether's formula in special cases. 

\begin{introthm}
Tropical Noether's formula holds for any compact tropical surface obtained  by way of successive modifications and summations of tropical toric surfaces. 
\end{introthm}

\noindent \textbf{Acknowledgment.}  Some of the contents of the paper are contained in the thesis \cite{ShawTh}. I am especially grateful to Grigory Mikhalkin for sharing his insight. Also to Beno\^it Bertrand, Erwan Brugall\'e, Dustin Cartwright,  Ilia Itenberg, Yael Karshon, Johannes Rau, and Bernd Sturmfels  for helpful discussions.

\section{Preliminaries}\label{sec:prel} 

\subsection{The tropical semi-field}

The tropical semi-field consists of the tropical numbers $\T = [-\infty, \infty)$ equipped with the operations; tropical multiplication, which is usual addition, and tropical addition, which is the maximum;
$$ ``xy" = x+y \qquad \text{and} \qquad ``x+y " = \text{max}\{x, y\} .$$ 
The additive and multiplicative identities are respectively, $-\infty$ and $0$. Tropical division is usual subtraction, whereas additive inverses do not exist due to the idempotency of addition; $``x + x" = x$. 

Tropical polynomial functions are convex integer affine functions $f: \T^N \to \T$ given by finite tropical sums  of the form	
$$f(x) = `` \sum_{I \in  (\Z_{\geq 0})^N} a_I x^I" = \max_{I \in  (\Z_{\geq 0})^N} 
\{ a_I + \langle I, x\rangle \}. $$
Notice that $2$ different tropical  polynomial expressions may produce the same function $\T^N \to \T$. 
Tropical rational functions are the difference of $2$ tropical polynomial functions $``g/h" = g-h$. Notice that this difference of functions is always 
defined on $\R^N$. 
The extension of such a function to points in $\T^N$ with coodinates equal to $-\infty$ may have indeterminacies.

\subsection{Standard tropical affine and projective spaces}\label{subsec:affproj}
Standard tropical affine space is $\T^N$ and the tropical torus is $(\T^*)^N = \R^N$. We equip $\T$ with the topology whose basis of open sets are the open intervals in $\R$ and $[-\infty, a)$. 

The space $\T^N$ has a stratification coming from the 
 order of sedentarity of points introduced by Losev and Mikhalkin.  

\begin{definition} \cite{MIk3}\label{def:sed}
The sedentarity of a point $x \in \T^N$ is 
$$S(x) = \{ i \ | \ x_i = -\infty \} \subseteq \{1, \dots , N\},$$ the order of sedentarity of $x$ is denoted by $s(x) = |S(x)|$. 
\end{definition}

For $I \subset \{1, \dots , N\}$, let  $\R^N_I$ denote the points of $\T^N$ of sedentarity $I$. Then $\R^N_I \cong \R^{n - |I|}$ and together these sets define a stratification of $\T^N$. We denote the closure in $\T^N$ of the stratum $\R^N_I$  by $\T^N_I$. 

Tropical projective space is defined analogously to classical projective space \cite{MIk3}  as 
$$\TP^N = \frac{\T^{N+1} \backslash (-\infty, \dots, -\infty)}{(x_0, \dots , x_N) \sim ``\lambda (x_0, \dots , x_N)"}.$$
This representation provides tropical homogeneous coordinates, denoted by $[x_0: \dots : x_N]$. 
Using these coordinates, the notion of sedentarity  extends to $\TP^N$. We use the same notation to denote the strata of $\TP^N$, namely  $\R^N_I$, where now $I \subsetneq \{0, \dots , N\}$. 
Denote the closure of $\R^N_I$ in $\TP^N$ by $\TP^N_I$, then we have $\TP^N_I \cong \TP^{n-|I|}$. 
The stratification of $\TP^N$ induced by the sedentarity  is the same as the face structure of the $N$-dimensional simplex. 

Tropical projective space  $\TP^N$ can also be glued from $N+1$ copies of tropical affine space as in the classical situation. 
Let $U_i$ be the open  set of points in $\TP^N$ with homogeneous coordinate 
$x_i \neq -\infty$ and the map $\phi_i :U_i \to \T^N$ be given by  $$\phi_i ([x_0: \dots: x_N]) = (x_0 -x_i, \dots, \hat{x}_i, \dots, x_N-x_0) \in \T^N.$$
On the overlaps of the charts the coordinate changes
$$\Phi_{ij}: \phi_i(U_i \cap U_j)  \to \phi_j(U_i \cap U_j),$$
are given explicitly by, 
$$(x_0, \dots, {x}_{i-1}, x_{i+1} \dots, x_N) \mapsto (x_0-x_j, \dots, x_{i-1}- x_j, -x_j, x_{i +1} -x_j , \dots, x_N - x_j).$$ 
Notice that the functions on the right hand side can be expressed using tropical operations as 
$``\frac{x_i}{x_j}"$.
In the image, the $j^{th}$ coordinate is removed (notice,  $x_j - x_j = 0$), so we have a map from $\T^N$ to $\T^N$.    
Let  $e_1,  \dots, e_N$  denote the standard basis vectors of $\R^N$, then $-e_1, \dots , -e_N, e_0 := \sum_{i = 1}^N e_i$
form distinguished directions in $\R^N$ with respect to this  compactification to $\TP^N$.  By this we mean that the closure in $\TP^N$ of  a half ray in $\R^N$ intersects the interior of a boundary hyperplane $\TP^N_{\{i\}}$ if and only if it is in one of these $N+1$ directions.

Given any other $\Z^N$-basis $\Delta = \{ u_1, \dots , u_N \}$, the tropical torus $\R^N$ can be compactified to projective space using $\Delta$ so that the distinguished directions are $u_0, \dots, u_N$, where  
$u_0 := - \sum_{i = 1}^N u_i$. 
Denote this compactification by  $\TP^N_{\Delta}$. The notation $\TP^N$  is reserved to denote $\TP^N_{\{-e_1, \dots , -e_N\}}$.

\subsection{Tropical fan planes}\label{sec:fanplanes}

The Bergman fan  \cite{Berg} of an algebraic variety $\V \subset (\CC^*)^N$ is defined as 
$$ \lim_{t \to \infty} \Log_t (\V) \subset \R^N.$$ 
In other words, the Bergman fan is the  limit  as $t \to \infty$ of amoebas of $\V$,  from  \cite{GKZ}. 
For $\V \subset (\CC^*)^N$ of dimension $k$, 
the set  $\lim_{t \to \infty} \Log_t (\V)$ is also $k$ dimensional and  can be equipped with a polyhedral fan structure. Moreover, positive integer weights can be assigned to the top dimensional faces so that the weighted polyhedral complex  satisfies to the balancing condition well known in tropical geometry, see \cite{St7}, \cite{BIMS}. The resulting weighted fan is denoted $\trp(\V) \subset \R^N$.

A $k$-dimensional   variety $\L \subset (\CC^*)^N$ is a  linear space if it is defined by a linear ideal in some monomials $z^{u_1}, \dots, z^{u_N}$ where the collection $u_1, \dots , u_N \subset \Z^N$ forms a $\Z^N$-basis.  A linear space $\L \subset (\CC^*)^N$ defines a \emph{matroid}  $M$ of rank $k+1$ on $N+1$ elements. The reader is directed to the standard reference on matroid theory \cite{Oxley} or also the more algebro-geometric  \cite{Katz:Matroid}.  
A  geometric way of obtaining the matroid is to first compactify  to $\bar{\L} =  \CP^k \subset \CP^N$ by way of  the coordinates $z^{u_1}, \dots, z^{u_N}$, so that $\bar{\L}$ is of degree $1$. 
Intersecting $\bar{\L}$ with any of the  $N+1$ coordinate hyperplanes  in $\CP^N$ defines a hyperplane arrangement $\mathcal{H}_0, \dots \mathcal{H}_N$ on $ \CP^k = \overline{\mathcal{L}}$. Define the rank function of the matroid on the ground set $\{0, \dots, N\}$ by 
 $$\text{rank}( I ) = \text{codim}_{\CP^k}( \cap_{i \in I } \mathcal{H}_i).$$
For a linear space $\L \subset (\CC^*)^N$, the tropicalization $\trp(\L)$ is a fan equipped with weight $1$ on all of its faces. The fan is determined by the underlying matroid of $\L$ and the $\Z^N$-basis $u_1, \dots , u_N$. This leads to the notion of Bergman fan for any matroid
 in \cite{SturmPoly}. 
  In general, the Bergman fan of a loopless matroid $M$ of rank $k+1$ on $N+1$ elements with respect to a $\Z^N$-basis $\Delta = \{u_1, \dots, u_N\}$ is a  $k$-dimensional tropical variety  which we denote, $\trp_{\Delta}(M) \subset \R^N$ where $\Delta$ is a $\Z^N$-basis. The weights on all top dimensional faces of $\trp_{\Delta}(M)$ are always $1$. 
  
    A description  of the Bergman fan of a matroid $M$ can be given in terms of its lattice of flats $\Lambda_M$ \cite{Ard}. A flat of a matroid is a subset of the ground set that is closed under the rank function. We review this construction only  in the case of rank $3$ matroids which  yield $2$ dimensional  fans. 

\begin{construction}\label{cons:matfan}\cite{Ard}
Let $M$ be a loopless matroid of rank $3$ on $N+1$ elements  labeled $0, \dots , N$.   
Let $\Delta = \{ u_1, \dots, u_N \}$ denote a  $\Z^N$-basis of $\R^N$  and set $u_0 = - \sum_{i = 1}^N$.
To construct $\trp_{\Delta}(M)$, first for every flat $F_I$  in the lattice of flats  $\Lambda_M$ fix the direction $u_I = \sum_{i \in I} u_i$. 
Then for every flag of flats 
$$\F = \{ \emptyset \subsetneq  F_{I_1} \subsetneq   \dots \subsetneq F_{I_l} \subsetneq \{0, \dots , N\} \}$$
 contained in the lattice of flats $\Lambda_M$ there is  a cone in $\trp_{\Delta}(M)$ 
spanned by the vectors $u_{I_1}, \dots u_{I_l}$. 
For a matroid of rank $3$ the maximum length of a chain of flats is $4$, including the flats $\emptyset$ and $\{0, \dots , N\}$ for which the corresponding direction $u_I$ is $0$.   Thus the maximal cones of $\trp_{\Delta}(M)$  are $2$ dimensional. This collection of cones is what is known as  the  fine fan structure \cite{Ard}.  Here we are only concerned with the support of this fan $\R^N$ which we denote $\trp_{\Delta}(M)$. This is known as the  coarse structure. 
 
To obtain  just the cones necessary for the support $\trp_{\Delta}(M)$  we can  delete 
 any rays of the above fan 
 corresponding to flats $F_I$ with $|I| = 2$, and also rays corresponding to single element flats $F_{\{i\}}$ which are contained in exactly $2$ flats of size $2$.  These $2$  types of flats produce exactly the  rays which subdivide a cone of dimension $2$.
\end{construction}

\begin{lemma}\label{lem:missingline}
Let $M$ be a rank $3$ simple matroid (loopless matroid and without double points) on $\{0, \dots , N\}$  and $\Delta$ a $\Z^N$-basis. Suppose that the Bergman fan $\trp_{\Delta}(M) \subset \R^N$  does not contain the ray in  the coarse structure in the direction  $u_i \in \Delta$ for some $i \in \{0, \dots , N\}$. Then $M$ is a parallel connection of $2$ rank $2$ matroids. 
\end{lemma}

\begin{proof}
Since we assume that $M$ has no loops or double points the rank $1$ flats of $M$ are precisely the elements of the ground set $\{0, \dots , N\}$. 
If there is no ray in direction $u_i$ in the coarse fan structure of $\trp_{\Delta}(M)$ then the element $i$ is contained in precisely $2$ flats of $M$ of  rank $2$, call them  $K$ and $L$. Moreover, by the axioms of flats 
the sets $K$ and $L$ partition  the ground set $\{0, \dots , N\} \backslash i$.  
Any other rank $2$ flat $J$ is $\{k,l\}$ where $k \in K$ and $l \in L$, since the intersection
$K \cap J$ is also a flat it  must be a  single element, and similarly for $L \cap J$. 
Moreover, $i \not \in J$ if $J \neq K, L$. 
Along with $\emptyset$ and $\{0, \dots , N\}$ this determines all of the flats of the rank $3$ matroid $M$. 
It is easy to check that these are precisely the flats of the parallel connection of 
$U_{2, k+1}$ and $U_{2, l+1}$ where $|K | = k+1$ and $|J| = j+1$. 
\end{proof}

\begin{corollary}\label{cor:missinglines}
Let $M$ be a rank $3$ simple matroid (loopless matroid and without double points) on $\{0, \dots , N\}$  and $\Delta$ a $\Z^N$-basis. Suppose that the Bergman fan $\trp_{\Delta}(M) \subset \R^N$  does not contain the ray in  the coarse structure in the direction  $u_i \in \Delta$ for some $i \in \{0, \dots , N\}$. Then 
$\trp_{\Delta}(M)$ is one of the following,
\begin{enumerate}
\item $\trp_{\Delta}(M)  = \R^2$, so that $M = U_{3, 3}$;
\item  $\trp_{\Delta}(M)  = \R \times L$, where $L$ is a fan tropical line, the underlying matroid is  $M = U_{2, N} \oplus \{ e\}$; 
\item $\trp_{\Delta}(M) $  is the cone over a complete bipartite graph $K_{k+1, l+1}$ and $M$ is the parallel connection of the uniform matroids $U_{2, k+1}$ and $U_{2, l+1}$, where $k+l = N$.
\end{enumerate}
\end{corollary}

\begin{proof}
Using $k, l$ from Lemma \ref{lem:missingline}. When both $k, l = 1$ then we are in Case $(1)$. If exactly one of $k, l$ is $1$ we are in Case $(2)$, otherwise we are in Case $(3)$. 
This completes the proof of the lemma. 
\end{proof}

Notice that the parallel connection of the uniform matroids $U_{2, k+1}$ and $U_{2, l+1}$  is realizable over characteristic $0$, see Figure \ref{fig:linearrangements}.  Case $(2)$ of Corollary \ref{cor:missinglines} 
consists of a $k$ lines through a point and $1$ line not through the point.  The line arrangement   in case $(3)$,  consists of $k+1$ lines containing a point 
 $p$ and $l+1$ lines containing a distinct point $q$, and a single line containing only the points $p$ and $q$, see Figure \ref{fig:linearrangements}. The line containing only $2$
 points of the arrangement corresponds to the direction $u_i$ which is not a ray in the coarse fan structure of $\trp_{\Delta}(M)$.

Given a Bergman fan of a matroid $L = \trp_{\Delta}(M)$, using $\Delta$ we may take the  closure of $L$ in $\T^N_{\Delta}$ or even compactify it in $\TP^N_{\Delta}$. 

\begin{definition}
A  $k$ dimensional fan tropical linear space  $L \subset \R^N$ is a tropical cycle equipped with weight $1$ on all of its facets such that  $L = \trp_{\Delta}(M)$ for 
some $\Z^N$-basis $\Delta$ and rank $k+1$ matroid $M$. 

A  $k$ dimensional  fan tropical linear space $L \subset \T_{\Delta}^N$,  (respectively $L \subset \TP_{\Delta}^N$),  of sedentarity $\emptyset$ is a tropical cycle  equipped with weight $1$ on all of its facets such that it is the closure of $L^o = \trp_{\Delta}(M) \subset \R^N$  in $\T_{\Delta}^N$ 
(respectively $\TP^N$) where $\Delta = \{-e_1, \dots , -e_N\}$ and $M$ is a rank $k+1$ matroid $M$. 
 
\end{definition}

A fan tropical linear space is non-degenerate if it is not contained in a subspace of $\R^N$.
Equivalently a fan tropical linear space is non-degenerate if it is the Bergman fan of a matroid without double points.

Set $\Delta = \{-e_1, \dots , -e_N\}$, and let $\bar{L}$ denote the compactification of $L =\trp_{\Delta}(M) \subset \R^N$ in  $\TP^N$. 
Then $\bar{L}$ defines an arrangement of $k-1$-dimensional fan tropical linear spaces located at the boundary of $\bar{L}$ which we call $\A_M$, namely
 $H_i = \bar{L} \cap \TP^N_{\{i\}}$ for all $i = 0 , \dots N$. 
If $L = \trp_{\Delta}(M)$ then  the rank function of the matroid $M$ coincides with the one defined by
$$\text{rank}(I ) =  k  - \text{dim}( \cap_{i \in I }(\bar{L} \cap \TP^N_{\{i\}})).$$
Because of this geometric 
realization of a matroid $M$ as a tropical hyperplane arrangement $\A_M$,  flats of rank $k$ of a rank $k+1$ matroid  will be referred to as points, and flats of rank $k-1$ as lines, and so on, even when the matroid is  not representable  over a field.

\vspace{0.5cm}
Tropical manifolds introduced in Section \ref{sec:defsuf} are locally modeled on fan tropical linear spaces and 
coordinate changes are integer affine maps. Because of this we are interested in integer linear maps which act as automorphisms of fan tropical linear 
spaces. 

\begin{definition}
An automorphism of a 
fan tropical $k$-plane 
$L \subset \R^N$ is a map $\phi \in \GL_N(\Z)$ which preserves $L$ as a set in $\R^N$. 

An automorphism is trivial if it corresponds to a permutation of the elements of the basis $\Delta$. 
\end{definition}

If a fan tropical plane has a trivial automorphism it implies that the underlying matroid $M$ has automorphisms itself; these are permutations of the ground set preserving the matroid.  For example, the whole symmetric group $S_n$ acts on the uniform matroid $U_{k, n}$. 
A non-trivial automorphism of a fan tropical linear space  $L \subset \R^N$ is equivalent to the existence of another choice of $\Z^N$-basis $\Delta^{\prime}$  for which $L = \trp_{\Delta^{\prime}}(M')$ for a matroid  $M'$. Here $M'$ and $M$ may or may not be distinct. 
Some examples of tropical planes in $\R^N$ having non-trivial automorphisms were previously presented in {\cite[Example 2.23] {Br17}}\label{ex:2comp}.

\begin{example}\label{ex:delPezzo5}
Consider the  arrangement of $6$ lines  known as the braid arrangement, drawn on the left of Figure \ref{fig:linearrangements}.
Up to automorphism of $\CP^2$ there is only one such arrangement.  We may suppose that the lines can be given by the linear forms  
$z_i = 0$ for $i = 0, 1, 2$ and $z_i - z_j =0$ for $0 \leq i < j \leq 2$. The tropicalization of a linear embedding of the complement  $\CP^2 \backslash \A \to (\CC^*)^5$ is a  fan tropical plane  $P \subset \R^5$ which 
 is the cone over the Petersen graph \cite{Ard}. 

The complement $\CP^2 \backslash \A$ can be identified with the moduli space $\M_{0,5}$ of $5$-marked rational curves up to automorphism. 
Similarly, the fan $P$ is  isomorphic to the moduli space of $5$-marked complex rational tropical curves, $M_{0,5}$, see \cite{Ard}, \cite{Mik7}. There is also an action of  $S_5$  
 on $P \cong M_{0,5}$ induced by permuting the markings of the $5$-marked rational tropical curves. This gives the entire group of non-trivial automorphisms of $P$ which is the entire symmetric group $S_5 $, see \cite{ReShSt}.  

\end{example}

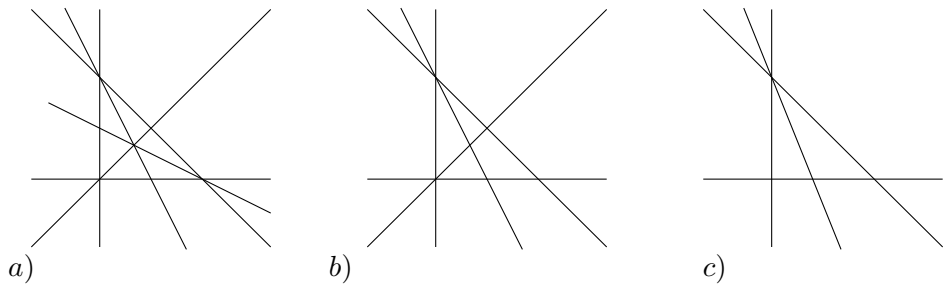
\begin{figure}\label{fig:linearrangements}
\begin{tikzpicture}[scale=0.9]
    \draw (0,-1) coordinate(c_1) -- coordinate(c_2) (0, 2.5); 
     \draw (-1,0)coordinate(d_1) -- coordinate(d_2) (2.5, 0); 
    \draw (-1,-1) coordinate (a_1) -- (2.5,2.5) coordinate (a_2); 
        \draw (-1,2.5) coordinate (b_1) -- (2.5, -1) coordinate (b_2); 
    \coordinate (ab) at (intersection of a_1--a_2 and b_1--b_2);
        \coordinate (cb) at (intersection of c_1--c_2 and b_1--b_2);
        \draw [add=.68 and .69] (cb) coordinate(e_1) to  (0.75,0) coordinate(e_2);
        \coordinate (eb) at (intersection of e_1--e_2 and a_1--a_2);
        \coordinate (db) at (intersection of d_1--d_2 and b_1--b_2);
          \draw [add=1 and 1.25] (db) to (eb);
\end{tikzpicture}
\hspace{1cm}
\begin{tikzpicture}[scale=0.9]
    \draw (0,-1) coordinate(c_1) -- coordinate(c_2) (0, 2.5); 
     \draw (-1,0)coordinate(d_1) -- coordinate(d_2) (2.5, 0); 
    \draw (-1,-1) coordinate (a_1) -- (2.5,2.5) coordinate (a_2); 
        \draw (-1,2.5) coordinate (b_1) -- (2.5, -1) coordinate (b_2); 
    \coordinate (ab) at (intersection of a_1--a_2 and b_1--b_2);
        \coordinate (cb) at (intersection of c_1--c_2 and b_1--b_2);
        \draw [add=.68 and .69] (cb) coordinate(e_1) to  (0.75,0) coordinate(e_2);
        \coordinate (eb) at (intersection of e_1--e_2 and a_1--a_2);
        \coordinate (db) at (intersection of d_1--d_2 and b_1--b_2);
        
\end{tikzpicture}
\hspace{1cm}
\begin{tikzpicture}[scale=0.9]
     \draw (0,-1) coordinate(c_1) -- coordinate(c_2) (0, 2.5); 
     \draw (-1,0)coordinate(d_1) -- coordinate(d_2) (2.5, 0); 
     \draw (-1,2.5) coordinate (b_1) -- (2.5, -1) coordinate (b_2); 
        \coordinate (ab) at (intersection of a_1--a_2 and b_1--b_2);
        \coordinate (cb) at (intersection of c_1--c_2 and b_1--b_2);
	\draw [add=.68 and .69] (cb) coordinate(d_1) to  (0.6,0) coordinate(d_2);
\end{tikzpicture}
\put(-230, -10){$b)$}
\put(-350, -10){$a)$}
\put(-90, -10){$c)$}

\caption{The braid arrangement from Example \ref{ex:delPezzo5} and examples of  arrangements from case $(2)$ and $(3)$ of Corollary \ref{cor:missinglines}.}
\label{fig:braidarrangement}

\end{figure}

In the example above, the automorphism group of the fan is non-trivial, but it turns out the underlying matroid $M$ is still determined  by the fan $P \subset \R^N$.
If a fan tropical linear space $L \subset \R^N$ of dimension $k$ is equal to $\trp_{\Delta}(M)$, %
it may be the case that  $L = \trp_{\Delta^{\prime}}(M^{\prime})$ for a distinct matroid $M^{\prime}$. 
The next proposition shows that, when restricting to rank $3$ matroids on the same number of elements, this is not possible.   In this case, given a $2$ dimensional matroidal fan in $\R^N$ the underlying rank $3$ matroid is unique, despite the possible existence of non-trivial automorphisms of the fan.

\begin{thm}\label{thm:MatEquiv}
Let $M_1, M_2$ be  rank $3$  simple (containing no loops or double points) matroids  on $N+1$ elements  and suppose that  $$\trp_{\Delta_1}(M_1)=  \trp_{\Delta_2}(M_2) \subset \R^N$$for some   $\Z^N$-bases, $\Delta_1$ and $\Delta_2$, then $M_1$, $M_2$ are isomorphic matroids.
\end{thm}

\begin{proof}
Denote the ground sets $E_1 = \{0, \dots , N\}$ and $E_2 = \{0', \dots , N'\}$. We will show that there is a bijection $f: E_1 \to E_2$ inducing an isomorphism of matroids. For this we just need to see that $f$ induces a bijection on the rank $2$ flats of $M_1$ and $M_2$, in other words of the points of the tropical line arrangements. 

Let $P = \trp_{\Delta_i}(M_i)$ and let the graph $G$ be the link of $P$. Notice that $G$ has no $2$ valent vertices. The identification of $P$ as the fan of the 
matroid $M_1$ gives a partial labelling of the vertices of $G$ by the elements of the ground set $E_1$.  Any unlabeled vertices 
correspond to points of the tropical line arrangement (flats of rank $2$). 
Similarly, there is another partial labelling of $G$ by the ground set of $M_2$.  Figure \ref{fig:graphsG} shows the graphs $G$ for the line arrangements from Figure \ref{fig:linearrangements}. 

Assume that all elements of $E_i$ label vertices of $G$, otherwise we are in the situation of Lemma \ref{lem:missingline} and the topology of $G$ is unique to the matroids listed in Corollary \ref{cor:missinglines}.  
In the case when  all elements of $E_i$ label the vertices of $G$, it is important to 
remark that  $2$ vertices of $G$ corresponding to points of $M_i$ cannot be adjacent in the graph $G$. This is not the case for the graphs $G$ in Figure \ref{fig:graphsG} $b)$ and $c)$ since not every line corresponds to a vertex of $G$.

Let $\tilde{V}$ denote the collection of vertices of $G$ which are labeled by exactly one of  $E_1$ or $E_2$ but 
not both. In other words a vertex is in  $\tilde{V}$ if it corresponds to a line in $M_1$ and a point in $M_2$ or vice versa. 
Suppose, without loss of generality, that a vertex in $V \backslash \tilde{V}$ is labeled by $m$ in $E_1$ and $m'$ in $E_2$.

We claim that either $|\tilde{V}| = 6$ or that the subgraph $\tilde{G}$ on the vertices of $\tilde{V}$ corresponds to the adjacency graph of a finite projective plane.  
If  $\tilde{V}  = 2$,  then we are in Case $(2)$ of Corollary \ref{cor:missinglines}, which was already excluded, so $\tilde{V} >2$. 
Suppose $v_1, v_2 \in \tilde{V}$ are labeled by elements $E_1$. Then  $v_1, v_2$ must either be connected by a unique edge of $G$ or they are both adjacent to a unique vertex $v$ of $G$ corresponding to a point of the tropical line arrangement of  $M_1$. 
We claim that only the second case is possible and that moreover,  $v \in \tilde{V}$ as well. 

Notice that $v_1, v_2$ are points in $M_2$, therefore they cannot be connected by an edge, and there must be a 
unique vertex  $v$  incident to $v_1, v_2$ in $G$. This vertex must be labeled by a line in $M_2$. The argument above says that $v$ is labeled by a point of $M_1$, so $v \in \tilde{V}$.

 On the other hand $2$ vertices $v_1, v_2 \in \tilde{V}$ corresponding to points in $M_1$ are labeled by elements of  $E_2$ and thus they must both be incident to a unique vertex which is labeled by  $E_1$.
 We can conclude that the subgraph of $G$ with vertices of $\tilde{V}$ is the incidence graph of a finite projective plane, or $|\tilde{V}| = 6$ and $\tilde{G}$ is a hexagon.

\begin{figure}
\begin{center}
\begin{tikzpicture}[scale=0.3, vertex_style/.style={circle,fill=black,draw, scale=0.5}, vertex_stylea/.style={circle, fill=white ,draw, scale=0.5} = {circle, color = white}, 
edge_style/.style={ black, scale=0.5}]
 
\useasboundingbox (-5.05,-4.4) rectangle (5.1,5.25);
 
\begin{scope}[rotate=90]
 \foreach \x/\y in {0/1,72/2,288/5}{
 \node[vertex_style] (\y) at (canvas polar cs: radius=2.5cm,angle=\x){};
 }
 
 \foreach \x/\y in {144/3,216/4}{
 \node[vertex_stylea] (\y) at (canvas polar cs: radius=2.5cm,angle=\x){};
 }
 
 \foreach \x/\y in {72/7,288/10}{
 \node[vertex_stylea] (\y) at (canvas polar cs: radius=5cm,angle=\x){};
 }
 
  \foreach \x/\y in {0/6,144/8, 216/9}{
 \node[vertex_style] (\y) at (canvas polar cs: radius=5cm,angle=\x){};
 }
\end{scope}
 
\foreach \x/\y in {1/6,2/7,3/8,4/9,5/10}{
 \draw[edge_style] (\x) -- (\y);
}
 
\foreach \x/\y in {1/3,2/4,3/5,4/1,5/2}{
 \draw[edge_style] (\x) -- (\y);
}
 
\foreach \x/\y in {6/7,7/8,8/9,9/10,10/6}{
 \draw[edge_style] (\x) -- (\y);
}
\end{tikzpicture}
\begin{tikzpicture}[
vertex_style/.style={circle,fill=black,draw, scale=0.5}, vertex_stylea/.style={circle, fill=white ,draw, scale=0.5} = {circle, color = white}, edge_style/.style={ black}]

\begin{axis}[height=5cm,width=7cm,hide x axis,hide y axis, ymin=-0.3,ymax=5,xmin=-0.3,xmax=4.3]
\node[vertex_style] (02) at (axis cs:1,4){};
\node[vertex_style] (00) at (axis cs:1,0){};
\node[vertex_style] (12) at (axis cs:2,4){};
\node[vertex_style] (10)  at (axis cs:2,0){};
\node[vertex_stylea] (22) at (axis cs:3,4) {};
\node[vertex_stylea] (20) at (axis cs:3,0){}; 
\draw[edge_style] (02) -- (00); 
\draw[edge_style] (02) -- (10);
\draw[edge_style] (02) -- (20);
\draw[edge_style] (12) -- (00);
\draw[edge_style] (12) -- (10);
\draw[edge_style] (12) -- (20);
\draw[edge_style] (22) -- (00);
\draw[edge_style] (22) -- (10);
\draw[edge_style] (22) -- (20);

 \end{axis}

\end{tikzpicture}
\begin{tikzpicture}[
vertex_style/.style={circle,fill=black,draw, scale=0.5}, vertex_stylea/.style={circle, fill=white ,draw, scale=0.5}, edge_style/.style={ultra thick, black}]

\begin{axis}[height=5cm,width=3.5cm,hide x axis,hide y axis, ymin=-0.3,ymax=5,xmin=-1.5,xmax=5.5]

\node[vertex_stylea] (0) at (axis cs:2,4){};
\node[vertex_style] (1) at (axis cs:2,0){};
\draw (0) to [bend right = 60] (1); 
\draw (0) to [bend left = 60]  (1);
\draw (0) to (1);
\draw (0) [bend right = 10] (1);
(0) edge [bend left=10] node  {} (1)
\end{axis}
\end{tikzpicture}
\put(-310, 0){$a)$}
\put(-190, 0){$b)$}
\put(-60, 0){$c)$}
\end{center}
\caption{The graph $G$ from the proof of Theorem \ref{thm:MatEquiv} for the arrangements from Figure \ref{fig:braidarrangement}. The vertices in black correspond to lines of the arrangement and vertices in white correspond to points.}
\label{fig:graphsG}

\end{figure}
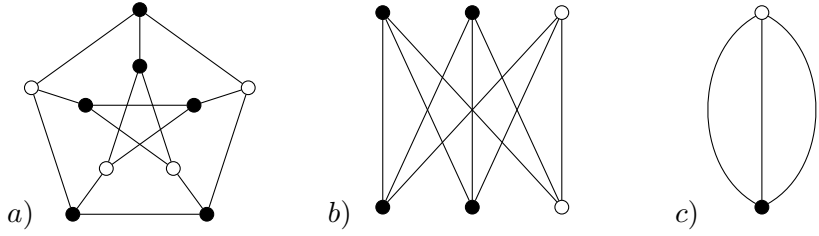

  If it is a finite projective plane, the integer linear map $\phi$ sending  $\Delta_1$ to $\Delta_2$ cannot be in $\text{GL}_N(\Z)$ since it is a composition of a permutation matrix with a linear map given by the incidence matrix of a finite projective plane.  
 The latter has determinant $> 1$ and so $\Delta_1$ and $\Delta_2$ cannot both be $\Z^N$ bases. 
 Therefore, $|\tilde{V}| = 6$  and the subgraph of $G$ on the vertices of $\tilde{V}$ is a hexagon.
 
Denote the points of $M_1$ corresponding to the vertices of $\tilde{V}$ labeled by flats  $I, J, K$, and the vertices of $\tilde{V}$ corresponding to elements of $E_1$ by $i, j, k$. 
 We claim that   $I \cup J \cup K = E_1$. Suppose $m \in E_1$, $m \neq i, j, k$ and $m \not \in I \cup J \cup K$.   Then in the graph $G$  the vertex $v$, labeled by $m$ must be connected to each of the vertices $i, j, k$ either by a single edge or by a path consisting of $2$ edges with an intermediate vertex which is a labeled by a point of the line arrangement of $M_1$. The second case is not possible, since in the labelling of $G$ given  by $M_2$, there would be  $2$ vertices labeled by points adjacent in $G$. If $v$ is adjacent to the vertices labeled by $i, j, k$, we also obtain a contradiction, since $v$ is also labeled by a line $m' \in E_2$, and  $m'$ and $i'$ would be contained in $2$ flats of $M_2$, (similarly for $j'$ and $k'$).   This contradicts the covering axiom for flats of matroids.

 Suppose the vertices of $\tilde{G}$  labeled by elements of $E_1$ are labeled  $i, j, k$. The three other vertices of $\tilde{G}$ are labeled by elements $i', j', k'$ of $E_2$, so that $i$ and $i'$ are opposite vertices of the hexagon and similarly for $j, j'$ and $k, k'$. 
Let  $f: E_1 \to E_2$ be given by  $i \mapsto i'$, then $f$ induces a bijection on the flats of $M_1$ to flats of $M_2$. 
So the matroids $M_1$ and $M_2$ are isomorphic. \end{proof}

From the above proof, a matroidal fan has non-trivial automorphisms in $\text{GL}_N(\Z)$ if and only if there exists a certain configuration in the underlying tropical line arrangement. 
Call a saturated triangle  of an arrangement $\A_M$ a collection of three lines $L_i, L_j, L_k$, and three points $p_I, p_J, p_K$ for $I, J, K \subset \{0, \dots , N\}$  such that, 
$$L_k \in I \cap J,  L_j \in I \cap K, L_i \in J \cap K \quad \text{ and } \quad  I \cup J \cup K = \{0, \dots, N\}.$$

\begin{cor}
A fan tropical plane in $P \subset \R^N$ has non-trivial automorphisms if and only if the underlying tropical line arrangement contains  a saturated triangle. 
\end{cor}

When a matroid is representable over a field of characteristic $0$,  an automorphism of its Bergman fan yields a birational automorphism of $\mathbb{P}^2$ regular on the complement of the line arrangement. In this case the above proposition and corollary can be proved using the structure of the Cremona group in dimension $2$.

\begin{rem}
It is known that Theorem \ref{thm:MatEquiv} does not hold for matroidal fans of higher dimensions. The first counter-example is a $3$-dimensional matroidal  fan in $\R^5$. Topologically the fan is the cone over  the graph $K_{3, 3}$ direct product with $\R$. It can be given by two matroids of rank three on $6$ elements, namely $U_{2, 3} \oplus U_{2, 3} $ and $M   \oplus \{e\}$ where $U_{2, 3}$ is the uniform matroid of rank $2$ on $3$ elements and $M$ is the matroid of the arrangement of $5$ lines drawn in on the left of  Figure \ref{fig:braidarrangement}. There are also examples of pairs of connected matroids exhibiting this property. 
\end{rem}

\subsection{One dimensional fan cycles}\label{sec:fancycles}

We restrict to describing $1$ dimensional tropical cycles, since we are for the most part  interested in cycles in tropical surfaces.  Definitions of general tropical cycles in $\R^N$ can be found in \cite{MIk3}, \cite{St7}. 

\begin{definition}
A fan $1$-cycle $C \subset \R^N$ is a $1$-dimensional rational fan equipped with integer weights on its edges such that 
$$\sum_{e \in \text{Edge}(C)} w_e v_e = 0,$$
where $\text{Edge}(C)$ is the set of edges of the fan, $w_e$ is the integer weight assigned to $e \in \text{Edge}(C)$ and $v_e$ is the primitive integer vector in the direction of $e$. 
\end{definition}

A fan $1$-cycle is called a fan tropical curve if all of the weights $w_e$ are positive integers. Cycles with positive weights are also known as effective.

For $\Delta = \{u_1, \dots, u_N\}$  a choice of $\Z^N$-basis  we recall the definition of the degree of $C$ relative to $\Delta$ from \cite{Br17}. Again let $u_0 = \sum_{i = 1}^N u_N$.   Firstly, the primitive integer vector $v_e$ of an edge $e$ of $C$ can be uniquely expressed as a positive linear combination
$$v_e = \sum_{i = 0}^N r_e(i)u_i, $$
where $r_e(i) \geq 0$ and $r_e(i) > 0 $ for at most $N$ of the $u_i$.

\begin{definition}\label{def:degreeDel}
Let $\Delta = \{u_1, \dots , u_N\}$ be a $\Z^N$-basis of $\R^N$, and  $C \subset \R^N$  a fan
tropical cycle.
Then the degree of $C$ with respect to $\Delta$ is 
$$\deg_{\Delta}(C) = \sum_{e \in \Ed(C)} w_er_e(i),$$
for any choice of  $i \in \{0, \dots, N\}$.  
\end{definition}

The fact that the above definition does not depend on the choice of $i$  follows from the balancing condition. 
The above definition of $\deg_{\Delta}(C)$ is equivalent to the multiplicity of the tropical stable intersection of $C$ with $H_{\Delta} \subset \R^N$, where $H_{\Delta}$ is the standard tropical hyperplane with respect to the basis $\Delta$ {\cite[Lemma 3.5]{Br17}}. In other words, $H_{\Delta} = \trp_{\Delta}(U_{N, N+1})$.

The degree of a fan $1$-cycle is dependent on the  choice of $\Delta$. Even when the fan $1$-cycle is contained in a fan tropical plane $P$, 
and we consider only $\Z^N$-bases $\Delta$ for which $P = \trp_{\Delta}(M)$  {\cite[Examples 3.3 and  3.4]{Br17}}. Define the degree of a $1$-cycle in a tropical plane to be the minimal of these degrees. 

\begin{definition}\label{def:degreeLoc}
The relative degree of a fan tropical $1$-cycle $C$ in a fan tropical plane $P \subset \R^N$ is 
$$\deg_P(C) = \min_{\Delta \in \mathcal{B}} \{ \deg_{\Delta}(C) \},$$
where $\mathcal{B}$ be the collection of $\Z^N$-bases $\Delta$ such that $P = \trp_{\Delta}(M)$. 
\end{definition}

It is also possible to define the degree of a fan  $1$-cycle in $\R^N$, using the minimum of $\Delta$ degrees, 
$$\deg(C) = \min\{ \deg_{\Delta}(C) \ | \ \Delta \mbox{ is a } \Z^N\mbox{-basis  of } \R^N \}.$$
Clearly, $\deg(C) \leq \deg_P(C) \leq \deg_{\Delta}(C)$.

\begin{definition}\label{def:localdegreeone}
A fan tropical  curve $C \subset P \subset \R^N$ for which $\deg_P(C) = 1$  is said to have degree $1$ in $P$. 
\end{definition}

\vspace{0.5cm}
Intersection numbers of tropical cycles in fan tropical planes have been defined in different places using various methods \cite{AlRa1}, \cite{ShawInt}, \cite{Br17}. Here we will recall the definition presented in \cite{Br17} for the sake of completeness. 
Recall from Section \ref{sec:fanplanes}, that for a fan tropical plane $P \subset \R^N$ and a $\Z^N$-basis $\Delta$ such that $\trp_{\Delta}(M) = P$ we can take the  compactification $\overline{P} \subset \TP^N_{\Delta}$,  and obtain a tropical line arrangement  with the same rank function as $M$, we denote this line arrangement $\A_M$. For two fan curves $C_1, C_2 \subset P$ we will define intersection multiplicities of the curves in $\overline{P}$ at points $p_I$ of the arrangement $\A_{M}$. That is to say at flats of rank $2$ of the matroid $M$.

\begin{defn}\cite[Definition 3.1]{Br17}\label{def:cornerInt}
Let $P \subset \R^N$ be a fan tropical plane and $\Delta = \{u_1, \dots, u_N\}$ be a 
$\Z^N$-basis such that $\trp_{\Delta}(M) = P$ for some matroid $M$. 
Given two fan tropical curves 
$C_1, C_2 \subset P$,
let $\overline{C}_i$ denote their compactifications in 
$\overline{P}
\subset \TP^N_{\Delta}$. Let $p_I \in \overline{P}$ be a point  of $\A_{M}$ and 
suppose that $\overline{C}_1$ and $\overline{C}_2$ 
both have exactly one edge containing the point  $p_I$.
The intersection multiplicity 
of $\overline{C}_1$ and $\overline{C}_2$
at the
point  $p_I$ is defined as follows: 
\begin{enumerate}
\item If  $I = \{i,j\}$ choose an affine chart $U_m$ of $\TP^N_{\Delta}$ where $m \not \in I$ described in Section \ref{subsec:affproj}. Let $ \pi_{i,j}: U_m \to \TT^2 $ be the map induced by extending the linear  projection $\R^N \to \R^2$ with kernel $u_{m'}$ for $m' \neq i, j$.
Suppose the ray of $\pi_{i,j} (\overline{C}_1 \cap U_m) \subset \TT^2$
has weight 
$w_1$
 and primitive integer direction 
$(k_1,k_2)$, and
similarly the ray of $\pi_{i,j}(\overline{C}_2 \cap U_m) \subset
\TT^2$ has weight 
$w_2$ 
and primitive integer direction 
$(l_1,l_2)$
 then, 
$$(\overline{C}_1\cdot  \overline{C}_2)_{p_I} = w_1w_2\min \{k_1l_2, k_2l_1\}.$$
 
\item If $|I| > 2$ choose an affine chart, $U_m \ni p_I$ for $m \not
  \in I$, and a  projection $\pi_{i, j}: U_m \to \T^2$
  where $i, j \in I$ such that the rays of $\overline{C}_1$ and 
$ \overline{C}_2$ are contained in the union 
of the closed faces generated by $u_i, u_j$. Then 
$$(\overline{C}_1 \cdot  \overline{C}_2)_{p_I} = (\pi_{i, j}( \overline{C}_1 \cap U_m). \pi_{i, j}( \overline{C}_2 \cap U_m))_{(-\infty, -\infty)}.$$ 
\end{enumerate}

The intersection multiplicity is extended  by distributivity in the case
when $\overline{C}_1$ and $\overline{C}_2$ have more than one  ray containing the point  $p_I$.
\end{defn}

\begin{definition}\label{def:localInt}
Let $ P\subset \R^N$ be a non-degenerate  plane
and $\Delta$ a $\Z^N$-basis such that $P = \trp_{\Delta}(M)$ for some matroid $M$.
The  tropical intersection multiplicity of fan tropical $1$-cycles $C_1, C_2$ in $P$ at the 
vertex of the fan $p \in P$ is
$$(C_1\cdot C_2)_p= \deg_{\Delta}(C_1)  \deg_{\Delta}(C_1) - \sum_{p_I \in  p(\A)}(\overline{C}_1 \cdot  \overline{C}_2)_{p_I}.$$ 
\end{definition}

Although a choice of  $\Z^N$-basis $\Delta$ is  used in the above definition, the multiplicity of $2$ fan curves  at the vertex of the fan plane $P$ is independent of the choice as long as $P = \trp_{\Delta}(M)$ since  the above definition is equivalent to the intersection multiplicities of tropical cycles in matroidal fans given in \cite{ShawInt} and \cite{FrancoisRau}.

Notice that by definition, fan $1$-cycles  $\overline{C}_1, \overline{C}_2$  satisfy B\'ezout's theorem in the compactification of $P$ to  $\overline{P} \subset \TP_{\Delta}^N$. That is if we define the total intersection multiplicity of  fan $1$-cycles in  $\overline{P}$ to be, 
$$\overline{C}_1\cdot \overline{C}_2 = \sum_{x \in (\overline{C}_1 \cap \overline{C}_2)^{(0)}} (\overline{C}_1\cdot \overline{C}_2)_x,$$ 
then we immediately have the following proposition.

\begin{proposition}
Let $C_1, C_2$ be fan $1$-cycles in a  fan tropical plane $P = \trp_{\Delta}(M) \subset \R^N$, then
$$\overline{C}_1\cdot  \overline{C}_2 = \deg_{\Delta}(\overline{C}_1)\deg_{\Delta}(\overline{C}_2),$$
where $\overline{C}_i$ is the closure of $C_i$ in the compactification of $\R^N$ to $\TP^N_{\Delta}$.
\end{proposition}

Moreover, the above definition of local intersection multiplicity of $2$ fan tropical  curves   at the vertex of the fan plane $P$ reflects the complex  intersection multiplicities in the case when the  curves arise as tropicalizations of complex curves $\C_1, \C_2 \subset \P \subset (\CC^*)^N$, see {\cite[Theorem 3.8]{Br17}}.

\subsection{Fan modifications}\label{sec:mod}

General tropical modifications were introduced by Mikhalkin in \cite{MIk3}. 
 Here we recall the definitions of this  operation for fans in $\T^N$.
Eventually, the restriction will be to so-called degree $1$ modifications, which produce a new  $k$-dimensional fan tropical linear space $\tilde{L} \subset \T^{N+1}$ from a pair of fan tropical linear spaces $D \subset L \subset \T^N$, of dimensions $k-1$ and $k$ respectively. A more thorough treatment of degree $1$ modifications can be found in   \cite[Section 2.4]{ShawInt}.   
A general introduction to tropical modifications, and tropical divisors of regular and rational functions can be found in  \cite{BIMS}, \cite{MikRau}. 

Given a tropical variety  $V \subset \T^N$ of dimension $k$  and a tropical regular function $f: \T^N \to \T$ its graph $\Gamma_f(V) \subset \T^{N+1}$ is a rational polyhedral complex of dimension $k$ and it inherits weights from the top dimensional facets of $V$. However, since $f$ is only piecewise linear, $\Gamma_f(V)$ may not be balanced. 
There is a canonical way to add  weighted facets to $\Gamma_f(V)$ to produce a tropical cycle $\tilde{V}$. At each codimension $1$ face $E$ of $\Gamma_f(V)$ which fails the balancing condition attach the facet, 
$$F_{E} = \{(x- t e_{N+1} \ | \ x \in E \text{ and } t \in [0, \infty]\}.$$ 
Assign to $F_E$ the unique positive integer weight so that the union $\Gamma_f(V) \cup F_E$ is balanced at $E$. After carrying out this procedure for all unbalanced faces $E$ of $\Gamma_f(E)$ call the resulting polyhedral complex $\tilde{V}$. 

Let $\delta: \T^{N+1} \to \T^N$ denote the  linear projection with kernel $e_{N+1}$. Then the restriction to $\tilde{V}$, is  $\delta: \tilde{V} \to V$ is the tropical modification of $V$ along $f$. If $f(x) \neq -\infty$ for all $x \in \T^N$ then the  divisor of $f$,  $\text{div}_f(V)$ is a tropical cycle, with support 
$\{ x \in V \ | \ |\delta^{-1}(x)| >1\}$. The weight of a top dimensional face $E \subset \text{div}_f(V)$ is assigned the same weight as $F_E$  in $\tilde{V}$ where $\delta^{-1}(E) = F_E$. If $f(x) = -\infty$ for some $x \in V$ the divisor of $f$ may have additional components in the boundary strata of $\T^N$, here we ignore this case for simplicity, it is treated in \cite{ShawInt} and \cite{BIMS}. 

A tropical rational function is  $f= ``g/h"$ where $g, h : \T^N \to \T$ are both tropical polynomial functions. For simplicity, again assume that $g(x) \neq -\infty$ for all points $x \in V \subset \T^N$ and similarly for $h$. Define the divisor of $f$ restricted to $V$ by  
$$\text{div}_f(V) = \text{div}_g(V) - \text{div}_h(V),$$
If the divisor of $f$ is  an effective cycle,   we can once again take the graph of $V$ along $f$ and complete it to a tropical variety $\tilde{V}$ as above.

\begin{definition}
Let $V \subset \T^N$ be a tropical variety and $f = ``g/h"$ a tropical rational function such that $g$, $h$ do not both attain $-\infty $ at any point $x \in V$ and that $\text{div}_f(V)$ is effective. 
The tropical modification of $V$ along $f$  is $\delta:\tilde{V} \to V$ where $\tilde{V}$ is described above. 
\end{definition}

\begin{definition}\label{def:modT}
Suppose $L \subset \T^N$ is a fan tropical linear space and let $f$ be a tropical rational function on $\T^N$ such that $\div_L(f)$ is also a  fan tropical linear space in $\T^N$. Then the tropical modification $\delta: \tilde{L} \to L$ along $f$ is said to be a degree $1$ modification of $L \subset \T^N$. 
\end{definition}

Given a degree $1$ fan tropical modification, $\delta: \tilde{L} \to L$,  the tropical cycle $\tilde{L} \subset \T^{N+1}$ is also a fan tropical plane.  
 In particular,
a degree $1$ modification of a fan linear space $L \subset \T^N$ corresponds to a so-called matroidal extension on the underlying matroids {\cite[Proposition 2.25]{ShawInt}}.

We can also define open fan tropical modifications of a fan tropical linear space $L \subset \R^N$.   Given an open modification $\delta^o: \tilde{L} \to L$ along an effective divisor $D \subset L$, the space $\tilde{L} \subset \R^{N+1}$ should be thought of as the complement of $D$ in $L$, see Section 2.4 of \cite{ShawInt}  for more details.

One difference between open fan modifications of degree $1$ and the tropical modifications described in Definition \ref{def:modT}
 is that the $\Z^N$-basis $\Delta$ is no longer fixed in the open case; for a fan tropical linear space $L \subset \R^N$ an open degree $1$ modification may be along any effective divisor 
 $D \subset L$ as long as there exist a $\Z^{n}$-basis $\Delta$ and matroids $M, N$  such that $\trp_{\Delta}(M) = L$ and $\trp_{\Delta}(M) = D$.

\begin{definition}
Let  $L = \trp_{\Delta}(M)   \subset \R^N$ be a fan tropical linear space and let $f$ be a tropical rational function on $\R^N$ such that 
$\div_L(f) = \trp_{\Delta}(N)$ for a matroid $N$, denote $\tilde{L} \subset \R^{N+1}$ the tropical cycle obtained by completing the graph of 
$L$ along $f$. 
Then $\delta^o: \tilde{L} \to L$ is an  open degree $1$ tropical modification along $f$, where $\delta^o$ is induced by the linear projection $\R^{N+1} \to \R^N$ with kernel generated by $e_{N+1}$. 

\end{definition}

Once again the tropical cycle $\tilde{L}$ in the above definition is a fan tropical linear space in $\R^{N+1}$. The $\Z^N$-basis $\tilde{\Delta}$ of $\R^{N+1}$, for which $\deg_{\Delta}(\tilde{L}) = 1$ is obtained  by adding $-e_{N+1}$ to the $\Z^N$-basis $\Delta$ of $\R^N$ for which $\deg_{\Delta}(D) = \deg_{\Delta}(L)=1$.

The next example shows that the local degree of a $1$-cycle in a fan tropical plane from Definition \ref{def:degreeLoc}  is not invariant under open degree $1$  tropical  modifications.

\begin{example}\label{exam:curveDegDecrease}
Let  $P = \trp_{\Delta}(U_{3, 4}) \subset \R^3$ where $\Delta = \{-e_1, -e_2, -e_3\}$ see Figure \ref{fig:planeConic}. Let $C$ be the fan tropical curve with rays of weight $1$ in directions 
$$(-2, -1, 0), \qquad (1, 0, 1), \qquad (1, 1 ,   -1),$$
shown in red in Figure \ref{fig:planeConic}. Then $C \subset P$. 

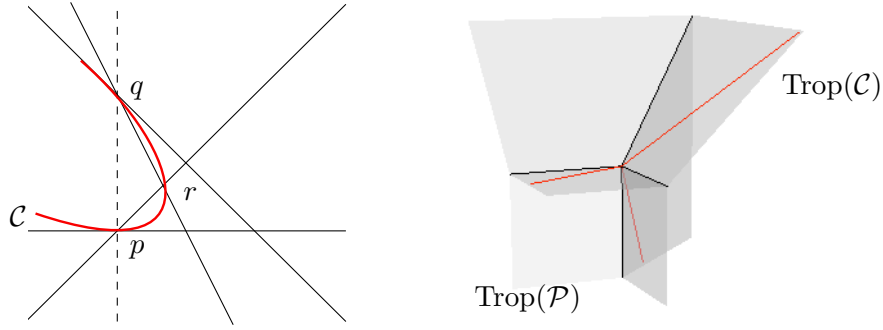
\begin{figure}
\includegraphics[scale=1]{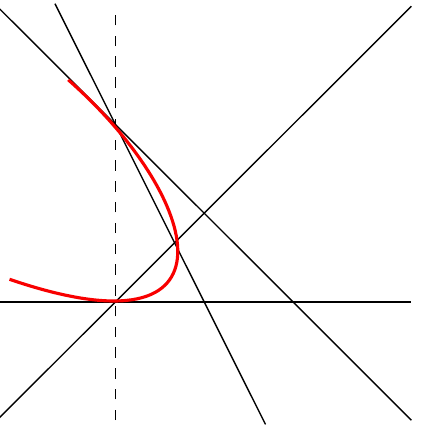}
\put(-85, 30){$p$}
\put(-85, 90){$q$}
\put(-65, 50){$r$}
\put(-130, 40){$\mathcal{C}$}
\hspace{1cm}
\includegraphics[scale=1]{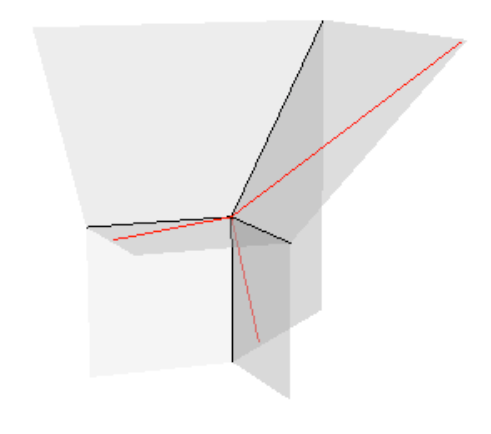}
\put(-15, 90){$\trp(\mathcal{C})$}
\put(-130, 10){$\trp(\mathcal{P})$}

\caption{On the right is the arrangement of lines defined by $\P \subset (\CC^*)^3$ in bold and the additional line defined by the arrangement of $\tilde{P} \subset (\CC^*)^4$ from Example \ref{exam:curveDegDecrease}. The conic $\C$ with respect to the arrangements is drawn in red. On the left is the tropicalization of the pair $\C \subset \P$ in $\R^3$.}
\label{fig:planeConic}
\end{figure}
There is a unique $\Z^3$-basis $\Delta$ for which $P = \trp_{\Delta}(M)$ for some matroid, and $\deg_P(C)  =\deg_{\Delta}(C) = 2$. The tropical $1$-cycle $L$ with support the affine line in direction $(1, 1, 0)$ equipped with weight $1$, is of degree $1$ in $P$. There is a tropical rational function $f$, such that $L = \text{div}_P(f)$. Performing the   modification of $P$ along 
$f$ yields a fan tropical   plane $\tilde{P} \subset \R^4$.
Applying the same modification to the above fan tropical curve $C$ yields the fan tropical curve $\tilde{C} \subset \tilde{P} \subset \R^4$ with directions, 
$$(-2, -1, 0, -1), \qquad (1, 0, 1, 1), \qquad (1, 1 ,   -1, 0),$$
 
The fan $\tilde{P}$ is a matroidal fan with respect to the $\Z^4$-bases $\Delta_1$, $\Delta_2$. The degrees of $\tilde{C}$ with respect to these bases are 
 $\deg_{\Delta_1}(\tilde{C}) = 2$ and $\deg_{\Delta_2}(\tilde{C}) = 1$, therefore  $\deg_{\tilde{P}}(C) =1$.  See  \cite[Examples 2.3 and 3.4]{Br17} for details.

\end{example}

\begin{rem}\label{rem:rectifiable}
Let $\P \subset (\CC^*)^3$ be the plane defined by the equation $z_1 + z_2 + z_3 +1 = 0$, then $\trp(\P) = P$. 
This plane defines the arrangement of $4$ bold lines on the left side of Figure \ref{fig:planeConic} and
The $\Z^4$-bases in the above example give two compactifications of the complex plane $\tilde{\P} \subset (\CC^*)^4$ to $\CP^2$ from  \cite[Example 2.3]{Br17}. The $2$ compactifications are related by performing the Cremona transformation at the $3$ points $p, q, r$ in Figure \ref{fig:planeConic}. The fan tropical curve $C$ is the tropicalization of the conic $\C$  drawn with respect to a line arrangement on the left. The image of $\C$ under the Cremona transformation is a line. 

Given a plane curve $\C \subset \CP^2$, consider a linear embedding  $\phi: \CP^2 \to \CP^N$. Let $\P^o = \phi(\CP^2) \cap (\CC^*)^N$ and analogously for $\C^o$.  There may be another compactification of $\P^o$ to a linear space in $\CP^N$ such that the closure of $\C^o$ in this compactificaiton may have smaller degree than the original plane curve $\trp(\C)$, as is the case in the example above.  
We may ask what is the  minimal  degree of a curve $\C \subset \CP^2$ which can be obtained by such a procedure. 
This minimal degree is  bounded below by the Cremona degree of a curve. Rational curves in $\CP^2$ of Cremona degree $1$ are called rectifiable. There  are known examples of rational curves which are not rectifiable 
see \cite{CalCrem}.  
\end{rem}

\section{Tropical surfaces}\label{sec:SURF}

\subsection{Tropical manifolds}\label{sec:defsuf}

An integer affine map $f: \R^N \to \R^M$ is a composition of an integer linear map and a translation in $\R^M$. 
Such a map can be given by $M$ tropical monomials i.e.~$$(x_1, \dots , x_N) \mapsto (``a_1\cdot x^{\alpha_1}", \cdots , ``a_M\cdot x^{\alpha_M}"),$$ where $(a_1 , \cdots a_M) \in \R^M$ encodes the translation and together the   $\alpha_i \in \Z^N$ form an integer $N \times M$ matrix. An integer affine map $f: \T^N \to \T^M$, is defined to be  the extension of an integer affine map $\R^N \to \R^M$. 

Tropical manifolds are instances of abstract tropical varieties from \cite{MIk3},  or \cite{MikZhaEig} (also called tropical spaces)
which are locally modeled on fan tropical linear spaces.  Just as for tropical spaces, the
coordinate changes are restrictions of, possibly partially defined, integer affine maps $\T^N \to \T^M$.

\begin{definition}\label{Man}

A tropical manifold $X$ of dimension $n$ is a Hausdorff topological space equipped with an atlas of charts $\{ U_\alpha , \phi_\alpha \}$, 
$\phi_\alpha : U_\alpha \rightarrow V_{\alpha} \subset \T^{N_{\alpha}}$, such that the following hold,

\begin{enumerate}

\item {for every $\alpha$ there is an open embedding  $\phi_\alpha : U_\alpha \rightarrow V_{\alpha} \subset \T^{N_\alpha}$, where $V_\alpha$ is a non-degenerate fan tropical  linear space of  dimension $n$;}

\item coordinate changes on overlaps  $$\phi_{\alpha_1} \circ \phi_{\alpha_2}^{-1}:
\phi_{\alpha_2}(U_{\alpha_1} \cap U_{\alpha_2}) \to \phi_{\alpha_1}(U_{\alpha_1} \cap U_{\alpha_2})
$$
are restrictions of (possibly partially defined) integer affine linear maps  $\Phi_{\alpha_1 \alpha_2}: \T^{N_{\alpha_2}} \to \T^{N_{\alpha_1}}$;

\item $X$ is of finite type:  there is a finite collection of open sets $\{W_i\}_{i=1}^s$ such that $\bigcup_{i=1}^s W_i = X$ and  $W_i \subset U_{\alpha}$ for some $\alpha$ and $\overline{\phi_{\alpha}(W_i)} \subset \phi_{\alpha}(U_{\alpha}) \subset \T^{N_{\alpha}}$. 

\end{enumerate}

\end{definition}

\begin{figure}
\includegraphics[scale=0.75]{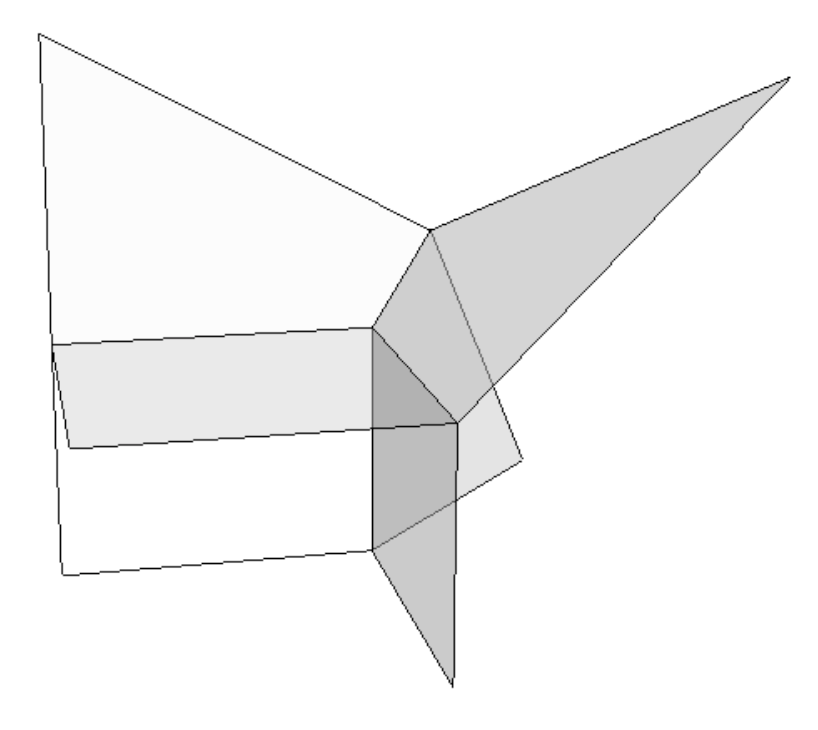}
\includegraphics[scale=1]{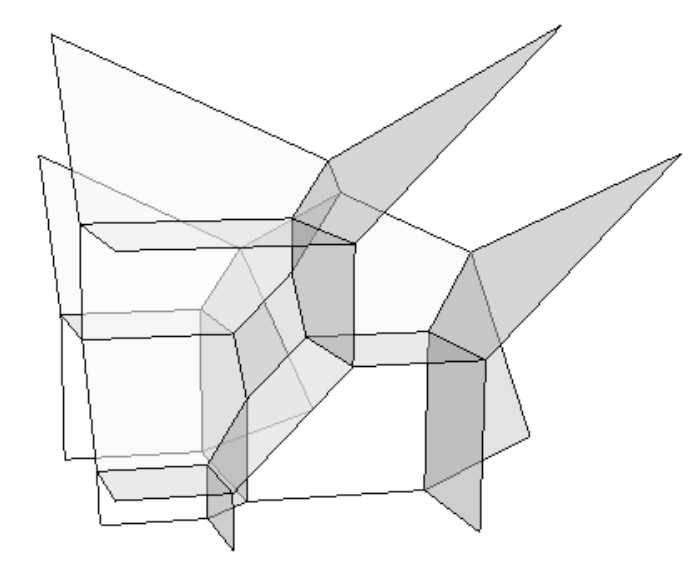}
\caption{A tropical plane in $\TP^3$ on the left and a quadric hypersurface in $\TP^3$ on the right.}
\label{fig:hypersurf}
\end{figure}

Just as with smooth manifolds, we say $2$ atlases $\{U_{\alpha}, \phi_{\alpha}\}$, $\{U^{\prime}_{\beta}, \phi^{\prime}_{\beta}\}$ on $X$ are equivalent if their union is also an atlas.  An equivalence class of atlases on $X$ has a unique maximal atlas. We  often work with this atlas even if a tropical manifold is defined with a more manageable collection of charts. 
See \cite[Example 7.2]{BIMS} for a discussion of the finite type condition.

Call a tropical manifold of dimension  $2$ simply a tropical surface. In this case the charts are  $\phi_{\alpha} : U_{\alpha} \rightarrow P_{\alpha} \subset \T^{N_\alpha}$ where $P_{\alpha}$ is a non-degenerate  
tropical fan plane.

\begin{example}[Integral affine manifolds]
An integer affine manifold satisfying the finite type condition is also a tropical manifold. In dimension $2$, the diffeomorphism   type of a  compact integer  affine manifold is either $S^1 \times S^1$ or the  Klein bottle (see Example \ref{ex:Klein}). 
For the orientable case, a tropical structure on $S^1 \times S^1$ can be given by the quotient  $\R^2 /  \Lambda$ where $\Lambda \subset \R^2$ is a lattice of full rank \cite{MikZha:Jac}. 
\end{example}

\begin{example}[Tropical toric surfaces]\label{ex:toricman}
A tropical toric  variety  $X$ of dimension $n$  is locally modeled on $\T^N$ and so it has an atlas $\{U_{\alpha}, \phi_{\alpha} \}$ where $\phi_{\alpha}: U_{\alpha} \to \T^N$. 
Tropical projective space appeared in the beginning of Section \ref{sec:prel}.
In general, copies of  affine space are glued together along tropical monomial maps, which classically are maps in $\text{GL}_N(\Z)$. Just as in the classical case a tropical toric variety can be  encoded by a 
 simplicial fan $\Sigma \subset \R^N$ and the resulting space is a tropical manifold if and only if  $\Sigma \subset \R^2$ is unimodular fan, see Section 3.2 \cite{MikRau}.  The tropical variety is compact if and only if $\Sigma$ is complete.   
\end{example}

\begin{example}[Non-singular tropical hypersurfaces in toric varieties]\label{ex:hyper}
A tropical hypersurface $X_f \subset \R^N$ is the divisor of a tropical polynomial function  $f:\R^3  \to \R$, see \cite{St2}, \cite{MIk3}. It is a   weighted  polyhedral complex dual to a regular subdivision of the Newton polytope of the defining polynomial.  If the dual subdivision is \emph{primitive}, meaning each polytope in the subdivision has normalized volume equal to $1$, the hypersurface is called non-singular and 
produces a tropical manifold.

 Examples of tropical surfaces  in $\R^3$ defined by  tropical polynomial functions of degrees $1$ and $2$ respectively  are shown in Figure \ref{fig:hypersurf}.
When $X_f$ is $2$ dimensional and non-singular, then locally, up to the action of $\text{GL}_3(\Z)$,   $X_f$  is the standard tropical hyperplane in $\R^3$ shown in Figure \ref{fig:hypersurf}. 

\end{example}

\begin{example}[Products of curves]
A non-singular abstract tropical curve is equivalent to a graph equipped with a complete inner metric,  \cite{MikZha:Jac}, \cite{BIMS}.

Given $2$ tropical curves $C_1$ and $C_2$, their product $X = C_1 \times C_2$  is a tropical surface in the sense of Definition \ref{Man} above. A point of the $0$-skeleton of 
$X$ which is not on the boundary, arises as the product of $2$ non-boundary  vertices of $C_1$ and $C_2$. 
The link of every such  vertex of $X$  is a complete bipartite graph $K_{l_1, l_2}$ where 
$l_i$ is valency of the corresponding vertex of $C_i$. The local model of $X$ at such a vertex is a fan tropical plane from part $(3)$ of Corollary \ref{cor:missinglines}.

Using  examples of  superabundant tropical curves in $\R^N$ for $N \geq 3$ \cite{Spe2}, \cite{Mik08}, it is possible construct examples of tropical surfaces in $\R^M$ for $M\geq 4$ which are products of tropical curves and are not approximable in the sense of Definition 5.16 of \cite{BIMS}.  
\end{example}

\subsection{Boundary arrangements}\label{sec:bdyArr}
The  sedentarity of a point in $\T^N$ from Definition \ref{def:sed} is coordinate dependent and does not translate directly to tropical manifolds. However, the order of sedentarity is still well defined since the  local models of tropical manifolds are non-degenerate fan tropical linear space and the coordinate changes come from extensions of integer linear maps over $\R$. 
For a point $x$  in a surface $X$ choose a chart $\phi_x: U_x \to \T^{N_{\alpha}}$ such that $x \in U_x$ and define the order of sedentarity of the point $x$ by $s(x) := s(\phi_x(x))$.

\begin{definition}\label{def:boundman}
The \textit{boundary} of a tropical manifold $X$ is $$\partial X = \{ x \in X \ | \ s(x) > 0 \}. $$ The \textit{points of sedentarity order $0$}, or interior points,  of $X$ are denoted 
$$X^o = \{ x \in X \ | \ s(x) = 0 \} = X \backslash \partial X.$$ 
An \textit{irreducible boundary divisor} $D$ of a tropical manifold $X$ is $\overline{D}^o$ where $D^o$ is a connected component of the set $\{ x \in X \ | \ s(x) = 1\}$. 
\end{definition}

\begin{example}
The points of a tropical surface can be classified into $3$ types based on their order of sedentarity. There are interior points ($s(x) = 0$), boundary edge points ($s(x) = 1$), and corner points ($s(x) \geq 2$). 
\end{example}

Every irreducible boundary divisor is of codimension $1$ in $X$. The irreducible  boundary divisors of a tropical manifold $X$ form an arrangement  $\A_X$. When $X$ is a surface call them boundary curves. There is a chart independent notion of sedentarity for
points in a tropical manifold in terms of this boundary arrangement.

\begin{definition}\label{def:sedman}
The sedentarity of a point $x$ in a tropical manifold $X$ is 
$$S(x) = \{ D \ | \ x \in D \in \A_X \} \subset \A_X.$$
\end{definition}

Notice that the set of points of  sedentarity $I \subset \A(X)$ of a tropical manifold $X$ may not be connected. 

\begin{definition}
A tropical manifold $X$ of dimension $n$ has simple normal crossing boundary divisors $D_1, \dots , D_k$ if every  connected component of the intersection $\cup_{i \in I} D_i$ is of dimension $n- |I|$. 
\end{definition}

\begin{example}[Boundary arrangements of linear spaces of dimension $2$ in  $\TP^N$]  
In \cite{TropPlanes}, a tropical linear space $L$ of dimension $2$ in $\R^N$ is described  by a so-called tree arrangement. The tropical linear space $L$ can be compactified to  $\bar{L} \subset \TP^N$, just as done for fan tropical linear spaces. The boundary arrangement of $\bar{L}$ is exactly the tree arrangement from \cite{TropPlanes}.

The tree arrangements of  tropicalizations of del Pezzo surfaces from \cite{ReShSt} can also be seen as boundary divisor arrangements on compactifications of these tropical surfaces.   Each such tree is in correspondence with a $(-1)$-curve on the  del Pezzo surface. 

\end{example}

\begin{example} 
A tropical hypersurface $X_f \subset \R^N$ from Example \ref{ex:hyper} can be naturally  compactified in the tropical toric variety $X(\Sigma)$ from Example \ref{ex:toricman},  where $\Sigma$ is the dual fan of the  Newton polytope of $f$. Under the same assumptions for non-singularity as in Example \ref{ex:hyper} the compactification  $\overline{X_f} \subset X(\Delta)$ is a tropical manifold and the boundary divisors of  $\overline{X}_f$ are in correspondence with the facets of $\Delta$ and have simple normal crossings. A collection of boundary divisors intersect if and only if the corresponding facets of the polytope $\Delta$ intersect. 
\end{example}

\subsection{Cycles in tropical surfaces}\label{sec:cycles}

\begin{definition}
A $0$-cycle $Z$ in a tropical surface $X$ is a finite formal sum of points in $X$ with integer coefficients, 
$$Z = \sum_{i = 1}^N m_ix_i.$$
The degree of a $0$-cycle is $deg(Z) = \sum_{i = 1}^N m_i$. 
\end{definition}

\begin{definition}\label{def:cycman}
 A tropical $1$-cycle of sedentarity $\emptyset$ in a tropical surface $X$ is a subset  $C \subset X$ such that in every chart $\phi_{\alpha}: U_{\alpha}  \to P_{\alpha} \subset \T^{N_{\alpha}}$  there exists a $1$-cycle $C_{\alpha} \subset P_{\alpha}$ of sedentarity $\emptyset$ with $$\phi_{\alpha}(C \cap U_{\alpha}) = C_{\alpha} \cap \phi_ {\alpha}(U_{\alpha})$$ and the weights on the edges of $C_{\alpha}$ are consistent on the intersections $U_{\alpha} \cap U_{\beta}$
\end{definition}

\begin{definition}\label{def:boundarycyc}
 A boundary $1$-cycle $C$ in a tropical surface $X$ is a finite linear combination of boundary curves of $X$ with integer coefficients. 
\end{definition}

Given $k$-cycles $A, B$ in a tropical surface  $X$ their sum $A +B$ is  the tropical $k$-cycle supported  on union of the supports $A \cup B$ (with refinements  if necessary) along with addition of weight functions when facets coincide. Two cycles are equivalent if they differ by a tropical cycle of weight $0$. 
Denote the set of $k$-cycles in a surface $X$ up to the above equivalence by $Z_k(X)$. Then $Z_k(X)$ forms a group, see \cite{AlRa1}.  
The group of tropical $1$-cycles in a surface $X$ splits as a direct sum of  sedentarity $\emptyset$ cycles and  $\Z$-multiples of  $D_i$ for every irreducible boundary curve $D_i \in \A_X$, 
\begin{equation}\label{eqn:cycleDec}
Z_1(X) = Z_1(X^o) \oplus_{i=1}^s \ZZ D_i
\end{equation}

As before an effective tropical $1$-cycle in $X$ is also called a tropical curve.  
A tropical curve is  irreducible if it cannot be expressed as a sum of effective tropical cycles. 

\subsection{Chern cycles}\label{sec:Cherncycles}

The $k^{th}$ Chern cycle of a tropical variety $X$ is a cycle supported on its codimension $k$-skeleton \cite{MIk3}. The weights of the top dimensional faces of $c_k$ were defined in the case of the canonical class $K_X = -c_1(X)$.

\begin{definition}\cite{MIk3}
Given a tropical manifold $X$ of dimension $n$, its  canonical cycle $K_X$ is supported on the codimension $1$ skeleton $X^{(n-1)}$ of $X$. The weight of  a top dimensional face $E \subset X^{(n-1)}$  is given by $w_E = val(E) -2$, where $val(E)$ is the number of facets in $X$ adjacent to $E$. 
\end{definition}

For $1$ dimensional tropical manifolds (tropical curves) this is the   canonical class used in \cite{BakerNorine1}, \cite{GatKer}, \cite{MikZha:Jac} in relation to the tropical Riemann-Roch theorem.
In the case of a tropical surface $X$, the  canonical cycle is a $1$-cycle supported on the $1$-skeleton of $X$, the balancing condition is proved here in Proposition \ref{prop:canonicalBal}. 
By the above definition and the direct sum decomposition of the cycle group in Definition \ref{eqn:cycleDec}, the canonical class of a surface $X$ splits into a cycle supported on the boundary of $X$ and a cycle $K^{0}_X$ supported on the closure of the $1$-skeleton of the points of sedentarity $0$ of $X$. Therefore,   
$$K_X = K^o_X - \partial X,$$ since an edge $E$ of the surface located at the boundary has valency $1$, and is equipped with weight $-1$ in $K_X$.

\begin{proposition}\label{prop:canonicalBal}
The canonical cycle $K_X$ of  a tropical surface satisfies the balancing condition. 
\end{proposition}

\begin{proof}
The condition is non-trivial to check only at points $x$ in the $0$-skeleton $X^{(0)}$ of $X$ of sedentarity $0$. A neighborhood of such a point has a chart to some tropical plane $P_x \subset \R^{N_x}$ where $P_x = \trp_{\Delta}(M)$ for a matroid $M$ and  $\Z^{N_x}$-basis $\Delta = \{ u_1, \dots , u_{N_x}\}$. Let $w_u$ denote the weight in $K_X$ of the edge in $X^{(0)}$ in the direction $u$.   
 Construction \ref{cons:matfan}  determines the number of neighboring facets of an edge in direction $u_i$, so that 
$ w_{u_i} =  |\{p_I \ | \ i \in I \}| - 2 $ and $w_{u_I}  = |I| -2$. Therefore, 
 
$$
w_{u_i} + \sum_{i \in I}w_{u_I} =  (|\{p_I \ | \ i \in I \}| - 2 )+ \sum_{p_I \ | \ i \in I}(|I| -2)  = - 2 + \sum_{p_I \ | \ i \in I} (|I| - 1).
$$

Since $P_x$ is non-degenerate, $M$ has no double points, and by the covering axiom of flats of a matroid,  $\sum_{p_I \ | \ i \in I} (|I| - 1) = N_x$, since $N_x + 1$ is the number of elements of the ground set of $M$. Let $\Lambda(M)$ be the lattice of flats of $M$. 
Then the  balancing of $K_{P_x}$ at $x$ follows since 
$$\sum_{ I \in \Lambda(M) } w_{u_I} u_I = (N_x - 2) \sum_{i = 0}^{N_x} u_i = 0.$$
This proves the proposition. 
\end{proof}

It is enough to check that $K_X$ is balanced  for a  tropical surface $X$ to deduce that $K_Y$ is balanced for any tropical manifold $Y$.

\begin{corollary}
The canonical cycle $K_X$ of an $n$-dimensional tropical  manifold  $X$ is balanced. 
\end{corollary}

\begin{proof}The balancing condition for $K_Y$ is a local condition on codimension $2$ faces of $Y$ of sedentarity $\emptyset$. Let $x$ be a point in a codimension $2$ face $E$ and consider a chart $\phi_x: U_x \to V_x$ such that  $V_x \cong \R^{n-2} \times P_x$ where $P_x$ is a fan tropical plane. Then $K_X$ is balanced along $E$ if and only if 
$K_{P_x}$ is balanced at $x$. 
\end{proof}

The  proof of Proposition \ref{prop:canonicalBal} also shows that for a non-degenerate fan plane $P \subset \R^N$ the degree of $K_P$ in $P$ is $N-2$.  
If $P \subset \R^N$ is a fan tropical plane, the self-intersection of  $K_P^2$ in $P$ by Definition \ref{def:localInt} is, 
\begin{equation}\label{eqn:KP2}
K_P^2 = (N-2)^2 - \sum_{p_I}(|I|-2)^2, 
\end{equation}
where $p_I$ are the points of the arrangement of the matroid associated to $P$. 
The next  Proposition gives another expression for $K_P^2$. 

\begin{prop}\label{prop:localdeg}

Let $P \subset \R^N$ be a fan tropical plane and let $\text{Edge}(P)$ and $\text{Edge}(P)$ denote the set of $1$ dimensional and $2$ dimensional  faces of $P$ respectively, 

\begin{enumerate} 
\item if $P$ contains a lineality space then $K_P^2 = 0$;
\item if $P$ is a product of tropical lines from part $(3)$ of Corollary \ref{cor:missinglines} then 
$$K_P^2 
= 8 -  4|\text{Edge}(P)|  + 2|\text{Face}(P)|  ;$$ 
\item otherwise 
$$K_P^2 =  10 +N - 5|\text{Edge}(P)| + 2|\text{Face}(P)| - \sum_{E \in \text{Edge}(P)} \sigma(E),$$
where  $\sigma(E)$ is defined by $$\sigma(E)v_E = - \sum_{E, E^{\prime} \text{ span a face}} v_{E'},$$
and $v_E$ is the primitive integer vector in the direction of an edge $E \in \text{Edge}(P)$.   
\end{enumerate}
\end{prop}

\begin{rem}\label{remark:Lap}
Suppose $P = \trp_{\Delta}(M)$,  if a ray $E \in \text{Edge}(P)$ corresponds to a point of the associated tropical line arrangement $\A_M$ then  $ \sigma(E) = -1$, and if it corresponds to a line, $L_i$ then $\sigma(E)= 1 - ( n - |\{p_I \ | \ i \in I\}). $  
\end{rem}

\begin{proof}
Statement $(1)$ is clear. For statement $(2)$, if $P \subset \R^N$ is a product of tropical lines, then its link is a complete bipartite graph $K_{k  ,l }$ where $N = k + l- 2$ and $k, l \geq 3$. Therefore,  $\text{Faces}(P) = kl$
and  $\text{Edges}(P) = k+l$.  Moreover,  there are exactly $2$ points  $p_I$ of the corresponding tropical line arrangement for which $|I|>2$. 
Now  Equation \ref{eqn:KP2} 
becomes,  
\begin{align*}
K_P^2 & = (N-2)^2 - \sum_{p_I}(|I|-2)^2 \\
& = (k +l -4 )^2 - (k-2)^2 - (l-2)^2  \\
& = 2|\text{Face}(P)| - 4|\text{Edge}(P)| + 8.
\end{align*}

Part $(3)$ uses that $N(N+1) = \sum_{p_I}|I|(|I|-1)$ for the arrangements of non-degenerate tropical planes, this follows again from the covering axiom of
flats of a matroid.  Furthermore, from Remark \ref{remark:Lap} preceding the proof we have $$\sum_{E \in \text{Edge}(P)} \sigma(E) = N+1 -|\{p_I \ | \ |I| \geq 3\}| - \sum_{p_I \ s.t. \ |I| \geq 3} |I|.$$
Applying this equality, we arrive at the statement in Part $(3)$ above. 
\end{proof}

 We will now give the weights for the points in the top Chern cycle $c_n(X)$ of an $n$ dimensional tropical manifold, using the  local  matroidal structure. 
Firstly, for  a matroid $M$, let $\chi_M(t)$ denote its characteristic polynomial. The  \textit{reduced characteristic polynomial} of $M$ is $$\overline{\chi}_M(t) = \frac{\chi_M(t)}{1-t}.$$
If $\A$ is a complex hyperplane arrangement in $\CP^k$ then the Euler characteristic of the complement is given by $\overline{\chi}_M(1)$. 
The absolute value of $\overline{\chi}_M(1)$ is also known as the $\beta$-invariant of a matroid, introduced by Crapo \cite{Crapo}. For real  hyperplane arrangements,  the $\beta$-invariant gives the number of bounded components of the complement \cite{Zaslavsky:Arrangements}.

For a tropical manifold let $X^{(0)}$ denote the $0$-skeleton of $X$; that is points of $X$ which are in the $0$-skeleton of $V_{\alpha}$ for some chart $\phi_{\alpha}:U_{\alpha} \to V_{\alpha} \subset \T^{N_{\alpha}}$. 

\begin{definition}
Let $X$ be an $n$ dimensional tropical manifold with simple normal crossing boundary divisors. Then its top Chern cycle $c_n(X)$ is a $0$-cycle supported on $X^{(0)}$.  For  $x \in X^{(0)}$ let there be a neighborhood $U_x$ and chart $\phi_x: U_x \to V_x \subset \T^{N_x}$ with  sedentarity $S(\phi_x(x)) = I$.
Then, the multiplicity of $x$ in $c_n(X)$ is 
$$m_{c_n(X)}(x) = \overline{\chi}_{M_x}(1)$$ where  $M_x$ is a matroid on $N_x - |I| +1$ elements such that $ \trp_{\Delta}(M_x) = V_x \cap \R^{N_x}_I $.
\end{definition}

For a matroid $M$, $\overline{\chi}_M(1)$ can also be given in terms of the Orlik-Solomon algebra of $M$. From \cite{ZhaOrl} this algebra can be constructed from just the support of the fan $\trp_{\Delta}(M)$, 
see also {\cite[Theorem 2.2.6]{ShawTh}} for an alternative proof. 
Therefore the  multiplicity of a point $x \in X^{(0)}$  in $c_n(X)$ is independent of matroid chosen to represent $V_x$ 
For a tropical surface $X$, the following lemma expresses the weight of a point $x$ in $c_2(X)$  without recalling the underlying matroid.

\begin{lemma}\label{lem:Chern2mult}
Let $x$ be a  point in the $0$-skeleton of a tropical surface $X$, 
 
\begin{enumerate}
\item if $x \in X^{(0)}$ is a point of sedentarity order  $2$ then, $$m_{c_2(X)}(x) = 1;$$

\item if $x \in X^{(0)}$ is a point of sedentarity order $1$ then, 
$$m_{c_2(X)}(x) = 2 - \text{val}(x),$$ where $\text{val}(x)$ is the number edges of sedentarity order $1$ adjacent to $x$ in a chart $\phi_x: U_x \to P_x \subset \T^{N_x}$;

\item if $x \in X^{(0)}$ is a point of sedentarity order $0$, then 
\begin{equation}\label{chernmult}
m_x(c_2(X)) = 2 - N_x + |\text{Face}(x)| -  |\text{Edge}(x)|,
\end{equation}
where $\phi_x : U_x \to P_x \subset \T^{N_x}$ for a non-degenerate fan tropical plane $P_x \subset \T^{N_x}$ and $\phi_x(x)$ is the vertex of $P_x$. The sets $\text{Edge}(x)$,  $\text{Face}(x)$ are the sets of $1$ and  $2$ dimensional faces of $P_x$. 
 
\end{enumerate}

\end{lemma}

\begin{proof}
In the case of points of sedentarity order $1$ or $2$ the statement concerns rank $1$ or $2$ matroids and can be checked directly. 

For part $(3)$, the coefficients of the reduced characteristic polynomial of a matroid $M$ are given by the M\"obius function on $\Lambda_{d \A_M}$, the lattice of flats of the decone $d\A_M$ of the  arrangement  $\A_M$  associated to 
$\trp_{\Delta}(M_x)$, see \cite{Katz:Matroid}. The decone, $d\A_M$, is the  affine arrangement obtained by declaring $1$ of the hyperplanes of $\A_M$ to be the hyperplane at infinity and then removing it. 
Let $\textbf{p}(d\A)$  denote the set of points (flats of rank $2$) of $d\A_M$, and $\textbf{p}(\A)$ denote the set of points of $\A$.  Then, 
\begin{eqnarray}\label{eqn:reducedEulerChar}
\bar{\chi}_{M_x}(t) & =  &\sum_{F \in \Lambda_{d\A_M} } \mu(F) t^{rk(F)}\\ 
\bar{\chi}_{M_x}(1)&  = &1 - N_x + \sum_{\tilde{p}_I \in \textbf{p}(d\A)} (|I| -1)\\
& = &1 - 2N_x + \sum_{p_I \in \textbf{p}(\A)} (|I| -1).
\end{eqnarray}
If all lines of $\A_{M_x}$ correspond to $1$-dimensional rays in the support of the fan $\trp(M_x)$, then we have 
$$|\text{Edge}(x)| = N_x+1 + |\{p_I \in \textbf{p}(\A) \ | \  |I| \geq 3\}|$$
$$|\text{Face}(x)| =  |\{p_I \in \textbf{p}(\A) \ | \  |I| = 2\}| + \sum_{p_I \in \textbf{p}(\A) \text{ s.t. } |I| \geq 3 } |I|.$$
These expressions show that   (\ref{chernmult}) is equal to (3.6) above. 

If there is a line of $\A_{M_x}$ not corresponding to a $1$ dimensional ray of the fan then  the possible  arrangements are described explicitly in Corollary \ref{cor:missinglines} and the statement of this lemma can be checked directly in these cases. This completes the proof. 
\end{proof}

\subsection{Intersections of $1$-cycles in surfaces}\label{sec:int1cyc}
A convenient  feature of tropical intersection theory  is the ability to calculate intersection products locally and on the level of cycles in many cases. 
A first example of this is the stable intersection of tropical cycles in $\R^N$ \cite{St2}, \cite{MIk3}.
It is also the case that there is an intersection product defined on the level of tropical cycles in matroidal fans \cite{ShawInt}, \cite{FrancoisRau}, which extends to tropical manifolds without boundary. 

Here we describe the intersections of $1$-cycles in  tropical surfaces. Such intersections are defined on the cycle level with the exception of self-intersections of boundary divisors.  
By abuse of notation we will  use $A \cdot B$ to sometimes  denote the $0$-cycle of the intersection and  also the total degree of the intersection,  which is the integer
$$\sum_{x \in X}(A \cdot B)_x.$$

\subsubsection{Intersections of $1$-cycles of  sedentarity $0$}

Given  $1$-cycles $A$ and $B$ in a tropical surface $X$ and a point $x \in A \cap B$ we may  choose an open set $U_x \ni x$ and  chart $\phi_x: U_x \to P_x \subset \T^{N_x}$ such that $A_x = \phi_x(A \cap U_x )$
and $B_x = \phi_x(B \cap U_x)$ are fan $1$-cycles. 

\begin{definition}\label{intbddy}
Let $A, B$ be tropical $1$-cycles of sedentarity $\emptyset$ in a tropical surface $X$, then their intersection is the  $0$-cycle
$$A \cdot B = \sum_{x \in (A \cap B)^{(0)}}(A \cdot  B)_xx,$$
where $(A \cdot B)_x$ is the intersection multiplicity from  Definition \ref{def:localInt} of  $A_x$ and $B_x$ in $P_x$ at $\phi_x(x)$ in the chart $\phi_x: U_x  \to P_x \subset \T^{N_x}$. 
\end{definition}

\subsubsection{Intersection of a boundary divisor and $1$-cycle of sedentarity $0$}
 An irreducible  boundary curve $C$ and a non-boundary cycle $A$ in a tropical surface always intersect in a finite collection of points of sedentarity order $1$ or $2$. The intersection product of $A$ and $C$ is a well defined $0$-cycle supported on $A \cap C$. 

\begin{definition}\label{2bddy}
Let $A$ be a cycle of sedentarity $\emptyset$ in a tropical surface $X$  and $C \subset X$ an irreducible boundary curve. We define,  $$A \cdot  C= \sum_{x \in A \cap C} (A \cdot C)_x x,$$ where,

\begin{list}{\labelitemi}{\leftmargin=0em}
\item[] $(1)$  if $x$ is a point of sedentarity order $1$ in $A \cap C$ adjacent to an edge $E$ of $A$ then $(A \cdot C)_x = w_E$, where $w_E$ is the weight of $E$;
\item[] $(2)$ if $x$ is a point of sedentarity order $2$ of $X$, choose a neighborhood $U_x$ of $x$ and take a chart $\phi_x: U_x\to P_x \subset \T^{N_x}$, suppose  $\phi_x (C \cap U_x) \subset \T^{N_x}$ is contained $\{x_i = -\infty\}$. Let $v_1, \dots v_s$ denote the primitive integer directions of all edges $E_1, \dots , E_s$ of $\phi_x(A \cap U_x)$ which contain $\phi_x(x) \in \T^{n_x}$ and $w_1, \dots w_s$ their respective weights.
Then $$(A \cdot C)_x = \sum_{j=1}^s w_j \langle e_i, v_j \rangle,$$
where $e_i$ is $i^{th}$   standard basis vector  of $\R^{N_x}$.
\end{list}
\end{definition}

When $D$ is a boundary divisor, which is not a irreducible boundary curve  then extend the above definition by distributivity.

\subsubsection{Intersection of boundary divisors}\label{sec:bdybdy}

\begin{definition}
Given distinct irreducible boundary curves $C_1, C_2$ in a tropical manifold $X$, then
$$C_1 \cdot C_2 = \sum_{x \in C_1 \cap C_2} x.$$  
\end{definition}

Again, the above definition can be extended by distributivity to reducible boundary divisors $D_1, D_2$ as long as the  components of $D_1$ are distinct from the components of $D_2$. 

\subsubsection{Self-intersections of boundary divisors}\label{sec:selfbdy}
The self-intersection of a boundary divisor  in a tropical surface is  not defined on the cycle level. 
To define self-intersections in this case, we define  normal bundles of irreducible boundary curves and take sections. 
Tropical line bundles on curves were introduced in \cite{MikZha:Jac}, and more general vector bundles  in \cite{AllBun}.

Given an irreducible boundary curve $C$ of a tropical surface $X$ having simple normal crossings with the other boundary divisors of $X$, then a neighborhood of $C$ in $X$ defines a tropical line bundle. When $C$ does not have simple normal crossings it is possible to alter a neighborhood of $C$ in $X$ to produce its normal bundle, for simplicity we do not describe the construction in this case. 

Let $\{U_{\beta}\}$ be a covering of $C$ in $X$, such that $U^{\prime}_{\beta} = C \cap U_{\beta}$ is  simply  connected for all $\beta$. Assume also that the atlas is fine enough so that  $\phi_{\alpha}(U_{\alpha} \cap U_{\beta}) \subset \R \times \T$ for all $U_{\alpha} \cap U_{\beta} \neq \emptyset$. 
Then for each pair $U_{\alpha} \cap U_{\beta}$, the map  $\phi_{\alpha} \circ \phi^{-1}_{\beta}$ is induced by an integer affine map  $\Phi_{\alpha \beta} : \R \times \T \to \R \times \T$.

\begin{definition}\label{def:normalbundle}
Let $C \subset X$ be an irreducible boundary curve of a tropical surface $X$  having simple normal crossings with the other boundary curves of $X$, let $\{U_{\beta}\}$ be a covering of $C$ in $X$, satisfying the conditions above. The normal bundle $\pi: N_X(C) \to C$ is the tropical surface given by
the quotient 
${\bigsqcup (U_{\beta}^{\prime} \times \T )}\backslash { \sim}$
where the equivalence relation is given by identifying points in $U_{\alpha} \times \T$ and $U_{\beta} \times \T$ via the maps $\Phi_{\alpha \beta}: \R \times \T \to \R \times \T$ from above.

\end{definition}

A  section of the normal bundle, $\pi: N_X(C) \to C$, is a continuous function $\sigma: C \to N_X(C)$ such that $\pi \circ \sigma = id$, with the additional requirement that   the restriction $\sigma|_{U_{\beta}}: U_{\beta} \to U_{\beta} \times \T$ is induced by a piecewise integer affine function  $\T^{N_{\beta}} \to \T$ in each chart $\phi_{\beta}: U_{\beta} \to \T^{N_{\beta}}$. In other words, $\sigma \circ \phi^{-1}_{\beta}$ is given by a tropical rational function. 
By Proposition 4.6 of \cite{MikZha:Jac} a section  of $N_X(C)$ always exists. We can also make the requirement that the section is bounded from below, so that there exists an $M >>0$ such that $\sigma \circ \phi^{-1}_{\beta}  > - M$. 

The graph of a section $\sigma: C \to N_X(C)$ can be made into a balanced tropical $1$-cycle in $N_C(X)$, similar to the  process of tropical  modification from Section \ref{sec:mod}. At each point $x$ of the graph $\sigma(C) \subset N_X(C)$  which is not balanced, in each chart  add an edge in the direction towards $C$,  equipped with an integer weight.  That is, add  the half-line $(x, a)$ for $-\infty \leq a \leq \sigma(x)$ in some trivialisation of the bundle. This edge  can be equipped with a unique $\Z$-weight  $w_x$ so that the resulting $1$ dimensional complex satisfies the balancing condition in each chart. 
Define the  degree of a section to be $$\deg(\sigma) = \sum_{x \in \sigma(C)^{(0)}} w_x.$$

Given two sections $\sigma_1, \sigma_2: N_X(C) \to C$  of the same bundle their degrees are the same by \cite[Lemma 1.20]{AllBun}.  

\begin{definition}\label{def:boundaryintersection}
The self-intersection of an irreducible  boundary curve  $C$ in a  tropical surface $X$  is  $\deg(\sigma)$, where  $\sigma: C \to N_X(C)$ is a section of $N_X(C)$. 
\end{definition}

\subsubsection{Tropical Cartier divisors}\label{sec:Cartierdivisors}

A Cartier divisor on a tropical space is a collection of tropical rational functions $\{f_{\alpha}\}$ defined on $\T^{N_{\alpha}}$ where $\phi:U_{\alpha} \to V_{\alpha} \subset \T^{N_{\alpha}}$, such that on the overlaps the functions agree up to an integer affine function, see \cite{AlRa1} or \cite{MIk3}. 
Every tropical Cartier divisor produces a codimension $1$ tropical cycle. 
The next proposition  says that on a tropical manifold $X$ the converse is also true. 

\begin{proposition}\label{Cartier}
Every codimension $1$ tropical cycle  $D$ in a tropical manifold $X$  is a tropical Cartier divisor. 
\end{proposition}

\begin{proof}
Choose a collection of charts $U_{\alpha}$ of $X$ so that $D_{\alpha} = D \cap U_{\alpha}$ is a fan cycle  in a fan linear space $V_{\alpha}$. 
In every chart the fan cycle is the divisor of a tropical rational  function $f_{\alpha}$ restricted to the fan linear space  \cite[Lemma 2.23]{ShawInt}. 
On the overlap $U_{\alpha} \cap U_{\beta}$, the cycles $D_{\alpha}$ and $D_{\beta}$ agree, therefore $``f_{\alpha}/f_{\beta}"$ is an invertible tropical function on the overlap. So $f = \{f_{\alpha}\}$ is a Cartier divisor on $X$ and $\div_X(f) = D$. 
\end{proof}

The intersections of  $1$-cycles in a surface can also be given in terms of Cartier divisors \cite{AlRa1} at least when the divisors are of sedentarity $\emptyset$.  For a Cartier divisor $f = \{f_{\alpha}\}$ on $X$ such that $\div_X(f) = D$ and a $1$-cycle $C$, then
$$C \cdot D = \deg(f|_C).$$
 The intersection product in terms of Cartier divisors is compatible with the product  described above. 
However, this approach requires first expressing a $1$-cycle $D$ as a Cartier divisor.

\subsection{Rational equivalence}

Tropical rational equivalence was introduced in   \cite{MIk3} by way of families. 
There is another  version of tropical rational equivalence coming from bounded tropical rational functions in \cite{AlRa1}, \cite{AllHamRau:RatEq}.
That equivalence relation is finer than the one defined here, see Remark \ref{rem:bddRatEq} for a comparison.

\begin{definition}
Let $A, B$ be tropical cycles in a tropical manifold $X$, then $A$ and $B$ are rationally equivalent if there exists a tropical cycle $Z \subset X \times \TP^1$ such that $$\pi_{\ast}((X \times \{-\infty\} ) \cdot Z - (X \times \{\infty\}) \cdot Z) ) = A-B, $$ where $\pi$ is the projection $X \times \TP^1 \to X$.  
\end{definition}

For brevity we refer the reader to \cite{MIk3}  and \cite{AlRa1} for the definition of pushforwards $\pi_{\ast}$ of tropical cycles. Also the intersections of sedentarity $0$ cycles  with boundary divisors is locally determined as for  the intersection with boundary divisors  in $\T^N$ from   \cite{ShawInt}.

\begin{rem}\label{rem:bddRatEq}
The above definition of rational equivalence is  equivalent to taking the equivalence relation given by $A \sim B$ if  $$\pi_{\ast}((X \times \{t_1\} ). Z - (X \times \{t_2\}) . Z) ) = A-B, $$ for a cycle $Z$ and any points $t_1, t_2 \in \TP^1$
\cite[Proposition 2.1.23]{ShawTh}.
The tropical rational equivalence from \cite{AllHamRau:RatEq} restricts  $t_1, t_2$ to be finite, i.e.~$t_i \neq \pm \infty$ \cite[Proposition 3.5]{AllHamRau:RatEq}. For example,  Theorem \ref{ChowMod} to be proved in Section \ref{sec:modGlobal} does not hold for this bounded version of rational equivalence. 
\end{rem}

As in classical algebraic geometry we can define the tropical Chow groups as the quotients of the cycle groups by rational equivalence. Denote the set of $k$-cycles rationally 
equivalent to $0$ by $R_k(X)$. Then $R_k(X)$ forms a subgroup of $Z_k(X)$.

\begin{definition} 
The $k^{th}$ Chow group of a tropical manifold $X$ is 
$$\text{CH}_k(X) = \frac{Z_k(X)}{R_k(X)}.$$
\end{definition}

\subsection{Tropical $(p,q)$-homology}
Tropical $(p, q)$-homology from  \cite{IKMZ}  is  homology of tropical varieties with respect to a coefficient system denoted by $\mathcal{F}_p$  defined by the local structure of the variety.  Other references on the subject include \cite{ShawHomo}, 
\cite{MikZhaEig} and the more introductory \cite{BIMS}.

Here we outline the definitions of this homology theory for tropical manifolds which are \textit{polyhedral}. 
A tropical manifold $X$ comes with  a combinatorial stratification, see Section 1.5 of \cite{MikZhaEig}.  Here we assume that $X$ has a stratification (perhaps a refinement of the combinatorial stratification)
which is polyhedral in the following  sense. 
\begin{definition}\cite[Definition 1.10]{MikZhaEig}
A tropical manifold $X$ of dimension $n$ is polyhedral if there are finitely many closed sets $F_j \subset X$, called facets, such that
\begin{itemize}
\item $X = \cup F_j$;
\item for each $F_j$ there is a chart $\phi_{\alpha}:U_{\alpha} \to V_{\alpha} \subset \T^{N_{\alpha}}$ such that $F_j \subset U_{\alpha}$ and $\phi_{\alpha}(F_j)$ is a
polyhedron of dimension $n$ in $\T^{N_{\alpha}}$;

\item for each collection $\{ F_i\}_I$ of facets of $X$, the intersection $\cap_{i \in I} F_i$ must be a face of every $F_j$.  
\end{itemize}

\end{definition}

Tropical $(p, q)$-homology can still be defined when  $X$ is not polyhedral, see for example \cite[Section 7]{BIMS}.

For a face $E$ of $X$ let $\text{int}(E)$ denote its relative interior. Let $x \in \text{int}(E)$, and suppose $\phi_{\alpha}(x)$ has sedentarity $I$ in $\T^{N_{\alpha}}$. 
Suppose $U_{\alpha}$ is a neighborhood of a face $E$ and $\phi_{\alpha}: U_{\alpha} \to V_{\alpha} \subset \T^{N_{\alpha}}$. 
The $p^{th}$ integral multi-tangent module at $E$ is, 
$$\mathcal{F}_{p}^{\Z}(E)  = \{ v_1 \wedge \dots \wedge v_p \ | \ v_1, \dots, v_p \in \Z_I^{N_{\alpha}}  \text{ and } v_1, \dots, v_p \in \sigma \subset T_x(X)\} \subset \Lambda^k (\Z_I^{N_{\alpha}}),$$
and the $p^{th}$ multi-tangent space is 
$\mathcal{F}_p(E) : = \mathcal{F}_p(E)^{\Z} \otimes \R$. 

For a pair of  faces satisfying $E \subset E^{\prime}$ there are maps $\iota_p : \mathcal{F}^{\Z}_p(E^{\prime}) \to \mathcal{F}^{\Z}_p(E)$ which are inclusions when $E$ and $E^{\prime}$ have the same sedentarity, and compositions of quotients and inclusions otherwise. 
 The fact that $\mathcal{F}_{p}^{\Z}(E)$ and $\iota_p$  are uniquely defined follows from the assumption that $X$ is polyhedral. 

A tropical $(p,q)$-cell is a  singular $q$-cell, respecting the polyhedral structure of $X$, i.e.~$\sigma: \Delta_q \to E  \subset X$, such that for each face $\Delta'_q$ of $\Delta_q$, $\sigma(\text{int}(\Delta'_q))$ is contained in the interior of a single  face of  $E$. Throughout we will abuse notation and use $\sigma$ to also denote the image of the map in $X$.  A singular $q$-chain $\sigma$ is equipped with a  coefficient $\beta_{\sigma} \in \mathcal{F}^{\Z}_p(E)$ to produce a $(p, q)$-cell, $\beta_{\sigma}\sigma$. Singular tropical $(p, q)$-chains are finite formal sums of $(p,q)$-cells. Denote the group of $(p,q)$-chains by $C_{p,q}(X; \Z)$. For a $(p,q)$-chain $\gamma$ denote its supporting $q$-chain  by $|\gamma|$.

The 
boundary map is the usual  singular boundary map 
$$\partial: C_{p, q}(X; \Z) \to C_{p, q-1}(X; \Z)$$
composed with given  maps on the coefficients $$\iota_p: \mathcal{F}_p^{\Z}(E) \to \mathcal{F}_p^{\Z}(E^{\prime}) \quad \text{ for } \quad  E^{\prime} \subset E,$$  when the boundary of a supporting $q$-cell $\sigma$ is in a face of $X$ different from  $\text{int}(\sigma)$. 
Just as for singular chains with constant coefficients the boundary operator satisfies $\partial^2 = 0$. The tropical integral $(p,q)$-homology groups are the homology groups of this complex, and are denoted $H_{p,q}(X; \Z)$. 

The $p^{th}$ multi-tangent spaces over $\R$, denoted  $\mathcal{F}_p := \mathcal{F}^{\Z}_p \otimes \R$ can also be used as coefficients for tropical 
homology. Denote these groups by $H_{p, q}(X;\R)$. 

When $X$ is a polyhedral, the tropical homology groups over $\R$ and the cellular cosheaf  homology (see \cite{Curry}) are equal \cite[Proposition 2.2]{MikZhaEig}.  
The cellular $(p,q)$-chain  groups of a tropical manifold have the advantage of being finitely generated since $X$ satisfies the finite type condition.

\begin{figure}
\includegraphics[scale = 0.3]{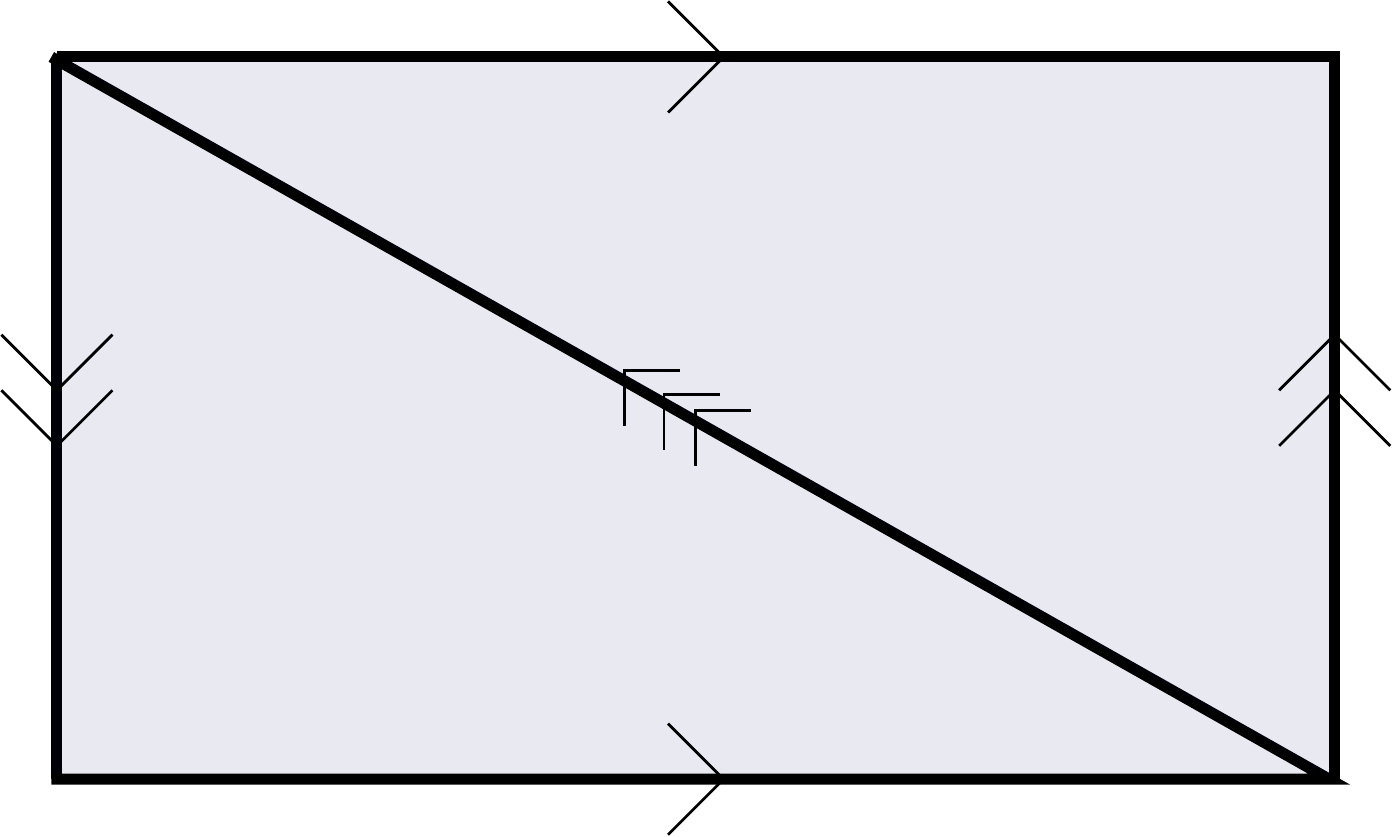}
\put(-100, 20){$F_2$}
\put(-30, 50){$F_1$}
\put(-130, 20){$E_3$}
\put(-30, 73){$E_1$}
\put(-60, 40){$E_2$}
\put(-125, 70){$x$}
\hspace{2cm}
\includegraphics[scale = 0.3]{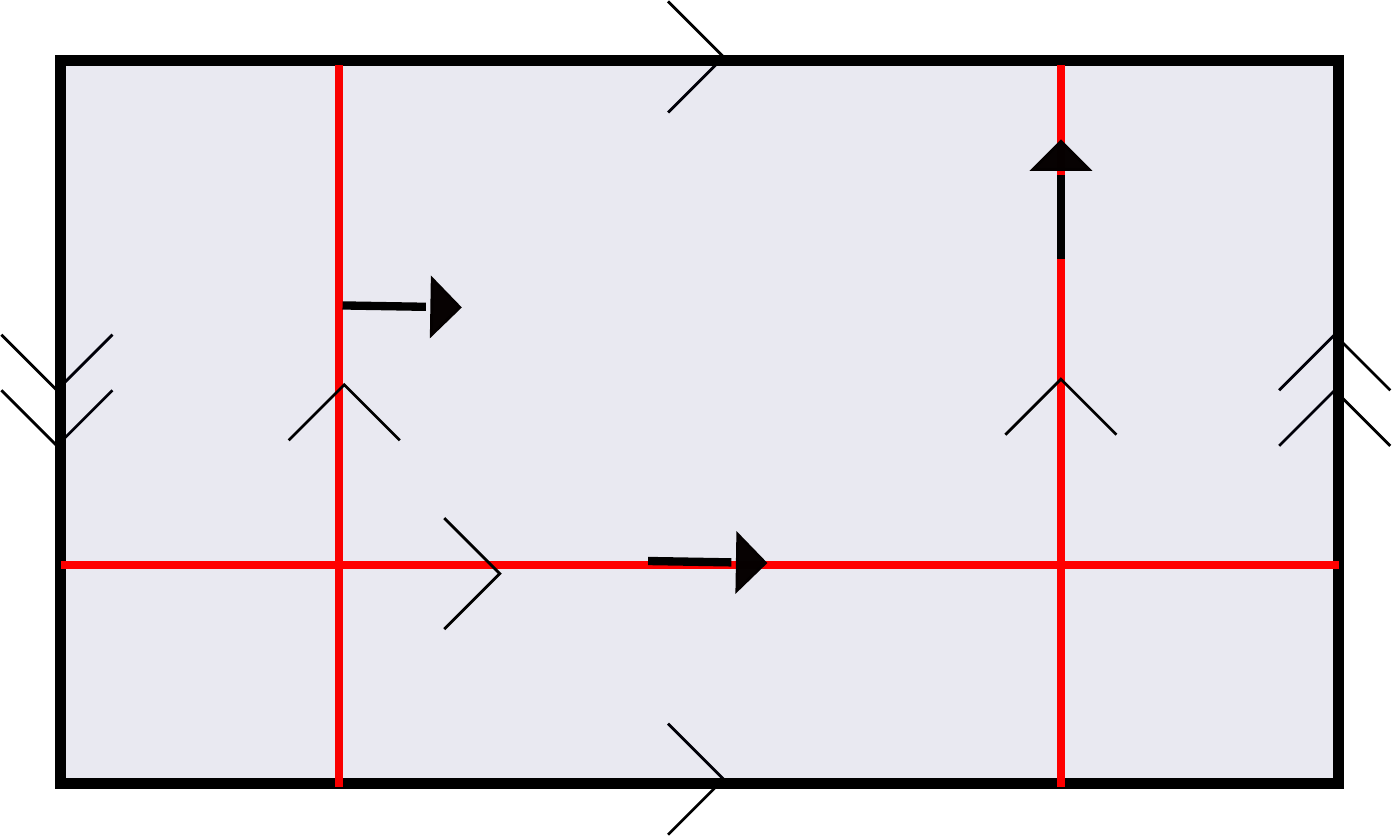}
\put(-63, 30){$\gamma$}
\put(-25, 50){$\gamma_2$}
\put(-103, 50){$\gamma_1$}
\caption{A Klein bottle $K$ with the cellular structure from Example \ref{ex:Klein}. On the right, the tropical $(1, 1)$-cycles in $K$ from Example \ref{ex:KleinInt}}
\label{fig:Klein}
\end{figure}
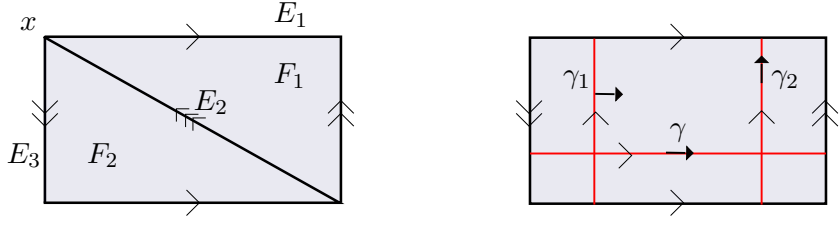

\begin{example}\label{ex:Klein}
Let $K$ be a Klein bottle  obtained from the  product of intervals $[-1, 1] \times [-1, 1]$ by 
 identifying $[-1, 1] \times \{-1\}$ and $[-1, 1] \times \{1\}$ with the same orientation and $\{-1\} \times [-1, 1] $ with $\{1\} \times [-1, 1] $ with the opposite orientation, 
as drawn in Figure \ref{fig:Klein}. Then $K$ comes with the structure of an integral affine manifold, and so it is a compact non-singular tropical manifold with charts to $\R^2$. 
It can also be made polyhedral by choosing an appropriate subdivision of  $K$. However, the computations are reduced and the  same integral  $(p, q)$-homology groups are obtained  by using the simplicial subdivision of $K$ from Figure \ref{fig:Klein} along with the groups $\mathcal{F}^{\Z}_p$ and maps $\iota_p$ specified below. 
In the subdivision of $K$ shown  in Figure \ref{fig:Klein}, the  edges are labeled by $E_i$, the faces by $F_i$, and the single point by $x$.

For every face $F$ of the subdivision of $K$ in Figure \ref{fig:Klein}, the $p^{th}$ multi-tangent space is 
$\mathcal{F}_p^{\Z}(F)  = \Lambda^p (\Z^2)$.  Let $e_1, e_2$ denote the standard basis vectors of $\Z^2$. 
The maps $\iota_p: \mathcal{F}^{\Z}_p(F_i)  \to  \mathcal{F}^{\Z}_p(E_j)$ are the identity for $j = 1, 2$ and $i = 1, 2$ and all $p$. For $E_3$ we have,

$$ \begin{array}{cccc}
\iota_1: & \mathcal{F}^{\Z}_1(F_i) &  \to & \mathcal{F}^{\Z}_1(E_3)
\\& e_1 &\longmapsto &  e_1
\\& e_2 & \longmapsto & (-1)^i e_2 
\end{array}
\quad \quad \quad 
\begin{array}{cccc}
\iota_2: & \mathcal{F}^{\Z}_2(F_i) &  \to & \mathcal{F}^{\Z}_2(E_3)
\\& e_1 \wedge e_2 &\longmapsto & (-1)^i e_1 \wedge e_2
\end{array}.$$
Lastly, $\iota_p: \mathcal{F}^{\Z}_p(E_i)  \to  \mathcal{F}^{\Z}_p(x)$ are the identity maps for all $i$ and $p$. 

The groups $H_{0, q}(K;\Z)$ are just the usual homology groups of $K$ with integral coefficients. 
The complex of cellular $(2, q)$-chains is, 
$$0 \to C^{cell}_{2, 2}(K; \Z) \to C^{cell}_{2, 1}(K; \Z) \to C^{cell}_{2, 0}(K; \Z) \to 0 \quad \text{ is } \quad  0 \to \Z^2 \to \Z^3 \to \Z \to 0.$$
The chain groups coincide with the constant coefficient case since $\mathcal{F}_2  = \Lambda^2(\Z^2) \cong \Z$. However, the homology groups  differ due to the maps $\iota_p$. The map $C^{cell}_{2, 2}(K; \Z)  \to C^{cell}_{2, 1}(K; \Z)$ has kernel 
 $$\langle (e_1 \wedge e_2)F_1 +  (e_1 \wedge e_2) F_2 \rangle,$$ 
so  $H_{2, 2}(K; \Z) \cong \Z$.  
Also, $C^{cell}_{2, 1}(K; \Z) \to C^{cell}_{2, 0}(K; \Z)$ has kernel $$\langle (e_1 \wedge e_2)E_3, 
(e_1 \wedge e_2)E_1 + (e_1 \wedge e_2) E_2 \rangle.$$ Quotienting by the image $$\text{Im}(C^{cell}_{2, 2}(K; \Z)  \to C^{cell}_{2, 1}(K; \Z)) = \langle (e_1 \wedge e_2)E_1 +  
(e_1 \wedge e_2)E_2 - (e_1 \wedge e_2)E_3 \rangle,$$ gives
$H_{2, 1}(K; \Z) = \Z$.
Finally, the image of $C^{cell}_{2, 1}(K; \Z) \to C^{cell}_{2, 0}(K; \Z)$ is $\langle 2(e_1 \wedge e_2) \rangle$. 
Therefore, $H_{2,0}(K; \Z) \cong  \Z_2$. 

To compute the groups $H_{1, q}(K; \Z)$ notice that the cellular chain complex splits as the direct sum of $2$ chain complexes, one consisting of cells with coefficients $e_1$ and the other with coefficients $e_2$. The complex with coefficients $e_1$ behaves exactly as the constant coefficient case, whereas the complex with coefficients $e_2$ behaves exactly as the $(2, q)$-complex. Therefore $H_{1, q}(K; \Z) = H_{0, q}(K; \Z) \oplus H_{2, q}(K; \Z)$. This gives all the integral tropical $(p, q)$-homology groups of $K$.

The tropical $(p,q)$-homology groups over $\Z$ can be arranged  in a Hodge like diamond, where the row at height $n$ has entries $H_{p, q}(X;\Z)$ for $n = p+q$ with $p$ increasing from left to right. The diamond for $K$ is 
\begin{center}
\begin{tabular}{cccccc} 
& & & $\Z$ &  & \\
& &$\Z$& & $\Z$ &  \\
&$ 0 $& & $\Z_2 \oplus \Z^2$ && $\Z_2$   \\
& & $\Z_2 \oplus \Z $ && $\Z_2 \oplus \Z$ & \\
& & & $\Z$ & & 
\end{tabular}
\end{center} 
 
\end{example}

\bigskip 

The $(1,1)$-homology group of a tropical surface carries an intersection form \cite{ShawTh}. More  generally there are intersection pairings between $(p, q)$-homology groups of tropical manifolds  \cite{MikZhaEig}.

A pair of  $(1, 1)$-cycles $\gamma_1, \gamma_2$ in a tropical surface $X$ are said to intersect transversally if $|\gamma_1 | \cap |\gamma_2| $ consists of  a finite number of points in the interior of facets of $X$, and each such point $x \in |\gamma_1| \cap |\gamma_2|$  is in the relative interior of just $2$ supporting $1$-cells which  intersect transversally in the usual sense. 
Suppose $\gamma_1$ and $\gamma_2$ intersect transversally in $X$, then each point $x \in |\gamma_1 | \cap |\gamma_2|$ is contained  in a facet $F_x$ of $X$. 
Suppose that $\beta_1(x)$ and $\beta_2(x)$ are the coefficients of $\gamma_1$ and $\gamma_2$ respectively of the supporting $1$-cells $\sigma_1$ and $\sigma_2$ containing $x$. 
Then consider the volume form $\Omega_{F_x} $ evaluated at $\beta_1(x) \wedge \beta_2(x)$. 
Define, 
$$\gamma_1 \cdot \gamma_2 = \sum_{x \in |\gamma_1| \cap |\gamma_2|} \epsilon  \Omega_{F_x}( \beta_{1}(x) \wedge \beta_{2}(x)) x \in C_{0}(X, \Z).$$

The coefficient $\epsilon$ is $1$ if the orientations of $\sigma_1, \sigma_2$ at $x$ induce the same orientation on $F$ as the ordered vectors $\beta_1(x) , \beta_2(x)$ and $\epsilon = -1$ if the orientations are opposite. 
Alternatively, for $\gamma_1$ and $\gamma_2$ intersecting transversally in $X$, the contribution to $ \gamma_1 \cdot  \gamma_2 $ at $x$ is 
$$\epsilon [\Lambda_F : \Lambda_{\beta_1 (x) } \oplus \Lambda_{\beta_2(x)}],$$
where $\Lambda_F$ is the integer lattice parallel to $\phi_i(F)$ for some chart and $\Lambda_{\beta_1(x)} \oplus \Lambda_{\beta_1(x)}$ is the sublattice generated by the vectors $\beta_1(x), \beta_1(x)\in \F_1(F)$. 
Notice that  $\gamma_1 \cdot \gamma_2 =  \gamma_2 \cdot \gamma_1$ and  also that the sign of the intersection multiplicity does not rely on choosing an orientation of facets of $X$. 

When $X$ is a compact tropical surface the above product on transversally intersecting $(1, 1)$-cycles descends to homology, 
$$\cdot : H_{1, 1}(X; \Z) \times H_{1, 1}(X; \Z) \to H_{0}(X, \Z) \cong \Z.$$
See 
\cite[Section 3.1.4]{ShawTh}, or the more general \cite[Theorem 6.14]{MikZhaEig}.

\begin{example}\label{ex:KleinInt}
Returning to Example \ref{ex:Klein}, let $\gamma$ be the cycle in  $C_{1, 1}(K; \Z)$ supported on $\sigma$ in Figure \ref{fig:Klein} and equipped with coefficient $e_1$. Let $\gamma_1, \gamma_2$ be the $(1, 1)$-cycles supported on the $1$-cycle $\tau$ in Figure \ref{fig:Klein} and equipped with coefficients $e_1$ and $e_2$ respectively. Then $\gamma$ intersects both $\gamma_1$ and $\gamma_2$ transversally in $K$ in a single point, $\gamma_1$ and $\gamma_2$ are disjoint.
The intersection product on $H_{1, 1}(K; \Z)$ is 
$$\gamma \cdot \gamma = 0 \quad \quad \gamma \cdot \gamma_1 = 0 \quad \quad \gamma \cdot \gamma_2 = 1 \quad \quad \gamma_i \cdot \gamma_j = 0.$$
\end{example}

Given a tropical $1$-cycle $A$ in a compact tropical surface $X$, there is a cycle map which produces a  tropical $(1,1)$-cycle  $[A]$, 
$$[  \ ]: Z_1(X) \to Z_{1, 1}(X).$$
Firstly, equip each edge $E$ of $A$ with an orientation to obtain a $1$-cell 
$\sigma_E$. These cells form the underlying $1$-chain of $[A]$. The coefficient of $\sigma_E$ in $\mathcal{F}_1$ is 
the integer vector parallel to $\sigma_E$ multiplied by the weight of $E$ in $A$. The fact that $[A]$ is a closed chain follows immediately from the 
balancing condition on $A$ \cite{MikZhaEig}, \cite{ShawTh}. 
If $X$ is not compact, then we can construct a cycle $[A]$ in the Borel-Moore version of tropical $(p, q)$-homology of $X$. 
As usual the Borel-Moore homology groups have the same definitions as the singular $(p,q)$-tropical homology groups except that chains consist of possibly infinite sums of cells satisfying a locally finite condition \cite{BMhomology}. 

Say that a $(1, 1)$-cycle $\gamma \in \T^N$ is parallel if each $(1,1)$-cell $\beta_{\sigma} \sigma$ has support contained in an affine line of direction $\beta_{\sigma}$.  
A $(1, 1)$-cycle in $X$ is parallel if it is parallel in each chart. For example if $A$ is a tropical cycle then $[A]$ is a parallel integral $1$-cycle. 
 
The next proposition shows that the intersection product on $1$-cycles in a compact surface $X$ from Section \ref{sec:int1cyc} is numerically equivalent to their product considered as tropical $(1, 1)$-cycles.

\begin{proposition}\label{prop:11equivInt}
Let $A, B$ be  $1$-cycles in a compact tropical surface $X$, then the total degree of the intersection of $1$-cycles $A \cdot B$ is equal to the intersection multiplicity $[A] \cdot [B]$ as $(1, 1)$-cycles. 
\end{proposition}

The next lemma considers a local situation for fan tropical cycles in a fan tropical plane $P \subset \T^N$. Since $P$ is not compact we are using Borel-Moore homology. 

\begin{lemma}\label{lem:homol}
Let $A, B$ be  fan $1$-cycles of sedentarity $\emptyset$ in a tropical  fan plane $P \subset \T^N$.  There exist arbitrarily small neighborhoods $U_A$, $U_B$  of $A, B$ respectively,  and $(1, 1)$ Borel-Moore cycles $a, b$ homologous to 
$[A], [B]$ in $P$ respectively, such that
\begin{enumerate}
\item  $|a| \subset U_A$ and $|b| \subset U_B$ and $a$ and $b$ intersect transversally in $P$; 
\item for each $x_i \in (A \cap B)^{(0)}$ there is a neighborhood $U_{x_i}$ of $x_i$ such that $U_{x_i} \subset U_A \cap U_B$,  $$|a| \cap |b| \subset \bigcup_i U_{x_i}, \quad \text{and} \quad 
m_{x_i}(A\cdot B) = \sum_{x_i' \in |a| \cap |b| \cap U_{x_i}} m_{x_i'}(a \cdot b). $$
\item outside of  $\cup_{x_i \in (A \cap B)^{(0)}} U_{x_i}$, the support  $|a|$ is contained in the union of a collection of affine lines which are in bijection with the rays of $A$, and 
 each  line in the collection is parallel to the corresponding ray of $A$. Moreover,  outside of  $\cup_{x_i \in (A \cap B)^{(0)}} U_{x_i}$ the $(1, 1)$-cycle $a$ is parallel. 
\end{enumerate}
\end{lemma}

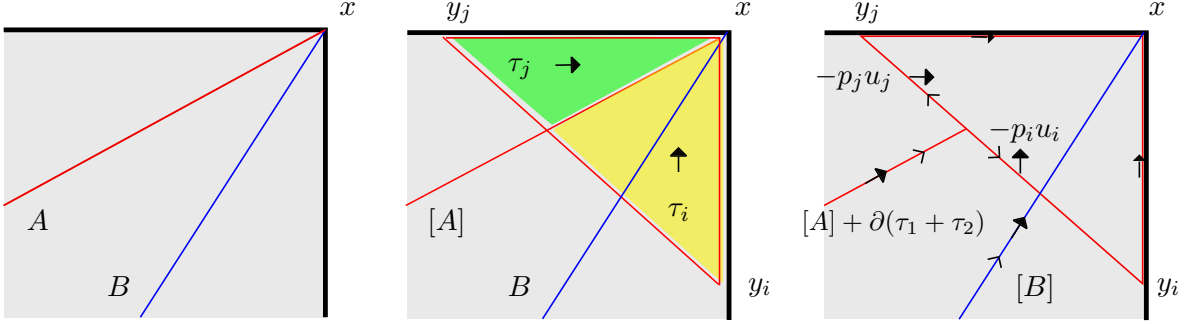
\begin{figure}
\includegraphics[scale = 0.75]{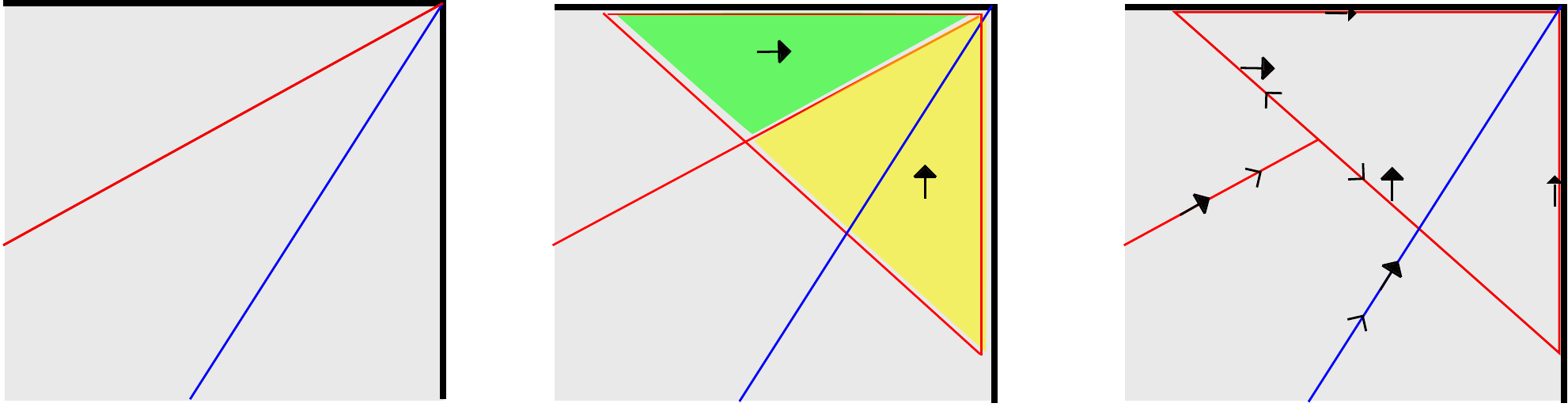}
\put(-420, 35){$A$}
\put(-390, 10){$B$}
\put(-303, 116){$x$}
\put(-150, 12){$y_i$}
\put(-263, 116){$y_j$}
\put(-155, 116){$x$}
\put(-240, 95){$\tau_j$}
\put(-180, 40){$\tau_i$}
\put(-270, 35){$[A]$}
\put(-240, 10){$B$}
\put(-130, 35){\small{$[A] + \partial(\tau_1 + \tau_2)$}}
\put(-50, 10){$[B]$}
\put(-60, 70){$- p_iu_i$}
\put(-125, 90){$- p_ju_j$}
\put(3, 12){$y_i$}
\put(-110, 116){$y_j$}
\put(0, 116){$x$}
\caption{The steps to construct the cycle $a$ from $[A]$ in Lemma \ref{lem:homol}.}
\label{fig:perturb}
\end{figure}

\begin{proof}
Since $A$ and $B$ are fan cycles, the only point of sedentarity $\emptyset $ in  $(A \cap B)^{(0)}$ is the vertex $x$ of the fan $P \subset \T^N$. 
We prove the lemma only in the case of the vertex $x$ of the fan $P$, omitting  the details for points of sedentarity since this case follows in a  similar fashion. 

Enumerate the rays of $P$ from $1, \dots, s$ and let $v_i$ denote the primitive integer vector pointing in the direction of the  $i^{th}$ ray. By Construction  \ref{cons:matfan}, each $v_i$ is in a direction of an indicator vector $u_I$, where $u_i = -e_i$ for $i = 1, \dots N$ and $u_0 = e_0$. 
We assume that in $[A]$ and $[B]$ the supporting $1$-cells are oriented outwards from the vertex $x$ of $P$. First we will find cycles $a, b$ homologous   to $[A], [B]$ respectively, and satisfying conditions $(1)$ and $(3)$ of the lemma, and then  compare the intersection multiplicities. 

The first step is to move $[A]$ to a homologous $(1, 1)$-cycle supported on the $1$-skeleton of $P$ in a neighborhood $U_x$ of the vertex of $P$. 
Let $U_x$ be the intersection of a  convex neighborhood of $x$ with $P$, so that  $U_x \subset U_A \cap U_B$. 
On the $i^{th}$ ray of  $P$ choose a point $y_i \subset U_x$ along the ray and do this for each $i$. 
Denote the face of $P$ spanned by vectors $v_i$ and $v_j$ by $F_{ij}$. Let  $r$ be a ray of $A$ in the interior of $F_{i,j}$ and  suppose that $r$ is equipped with weight $w_r \in \Z$ and is in direction $v_r = k_iv_i + k_jv_j$ for $p_i, p_j$ relatively prime positive integers. 
Let $\gamma_i = w_r k_iv_i \tau_i$ be a   $(1, 2)$-cell where $\tau_i$ is a $2$-simplex contained in $F  \cap U$ which is the convex hull of 
$0$, $y_i$ and the intersection point of $A$ with the segment from $y_i$ to $y_j$, see the middle of Figure \ref{fig:perturb}. 

Orient the simplex $\tau_i$  appropriately, so that  $[A] + \partial(\gamma_1 + \gamma_2)$ no longer has support on $r$ in the neighborhood $U_x$, see right side of Figure \ref{fig:perturb}. 
Apply this procedure to every ray of $A$ contained in the interior of a face of $P$, to obtain a $(1, 1)$-cycle $a^{\prime}$ homologous to $[A]$ which is supported on the $1$-skeleton of $P$ in $U_x$ and is homologous to $[A]$. 

If a   ray of $P$ is in direction $u_I$ for $|I| >1$, then for each $i \in I$, the $2$  dimensional cone spanned by $u_I, u_i$ is also in $P$. Apply the  same procedure as above to construct a cycle $a''$, again homologous to $[A]$, which in $U_x$ is supported on  the rays of $P$ in the directions $u_i$. The union of these rays  is the standard tropical line $L \subset \T^N$, i.e.~$L = \trp(U_{2, N+1})$
 The coefficient in $\mathcal{F}_1$ along the ray spanned by $u_i$ is $c_iu_i$, for some constant $c_i$.   
 The closedness of the cycle $a''$ at the vertex $x$ implies that the $c_i$'s are all equal. Calculating this coefficient for just one $i$ shows that it is $d_1 = \text{deg}_{\Delta}(A)$, where $\Delta = \{-e_1, \dots -e_{N+1}\}$.  
 
Finally, the standard tropical line 
can be translated in $\trp(M)$ for any $M$. Therefore, we can find a homologous cycle $a \sim a'' \sim [A]$ contained in a neighborhood $U_A$  of $A$ and such that in a neighborhood $U'_x \subset U_x$ it coincides with $[d_1L]$.  

The intersection of $a$ and $[B]$ may still be $1$-dimensional or consist of points contained on the $1$-skeleton of $P$ if $B$ has rays in common with $A$ or with the $1$-skeleton of $P$. 
  This can be avoided by choosing another representative $b$ of $[B]$ contained in a neighborhood $U_B$, whose edges are still in bijection with the edges of  $B$, so that corresponding edges are parallel. 
  
The intersection points of $a$ and $b$ of sedentarity $\emptyset$ come in $2$ types, those supported on the cells of $a$ which are segments from the chosen points $y_i$ to $y_j$ and those supported on the segments of $a$ in directions $u_i$. 
The sum of the multiplicities of all  intersection points contained on the rays of $a$ in directions $u_i$ is exactly  $d_1 d_2$. This is because the intersection points of $B$ and the tropical line $L$ coincide with the intersection points of $B$ and $a$ and the total intersection multiplicity of the former is $d_1d_2$, where $d_2 = \text{deg}_{\Delta}(B)$.

The other intersection points of $a$ and $b$ are supported on the segments from $y_i$ and $y_j$. For simplicity, suppose a face $F_{ij}$ of $P$ is generated by rays $u_i$ and $u_j$ and contains exactly $1$ ray from each of $A$ and $B$ in its interior. Let the ray of $A$ be in direction 
$k_iu_i + k_ju_j$ and the ray of $B$ be in direction $l_iu_i + l_ju_j$.
Then the intersection point of $a$ and $b$ contained in the segment 
from $y_i$  to $y_j$ is  $- w_Aw_B \min \{p_1q_2, q_1p_2\}$, where $w_A$, $w_B$ are the weights of the rays of  $A$ and $B$ in the face $F_{ij}$ respectively. 
The right side of Figure \ref{fig:perturb} shows that the orientation induced by the underlying 
$1$-cells of $a$ and $b$ are not consistent with the orientation of the coefficient vectors.
 The above calculated  contribution to the intersection of $a$ and $b$ is equal to the intersection multiplicity,  $- (\bar{A}. \bar{B})_{p_{ij}}$ 
from  Definition \ref{def:cornerInt}. When the rays of $A$ and $B$ are in a face of $P$ spanned by a pair of vectors $u_i$ and $u_I$ the result follows in a similar fashion.  
Extending by distributivity over all the rays of $A$ and $B$ 
and comparing with  the local  intersection multiplicity of $A$ and $B$ at $x$ from Definition \ref{def:localInt} completes the proof for  points in $A \cdot B$ of sedentarity $\emptyset$. 

A similar argument works for points of positive order of sedentarity, which completes the proof.
\end{proof}

\begin{figure}
\includegraphics[scale = 0.7]{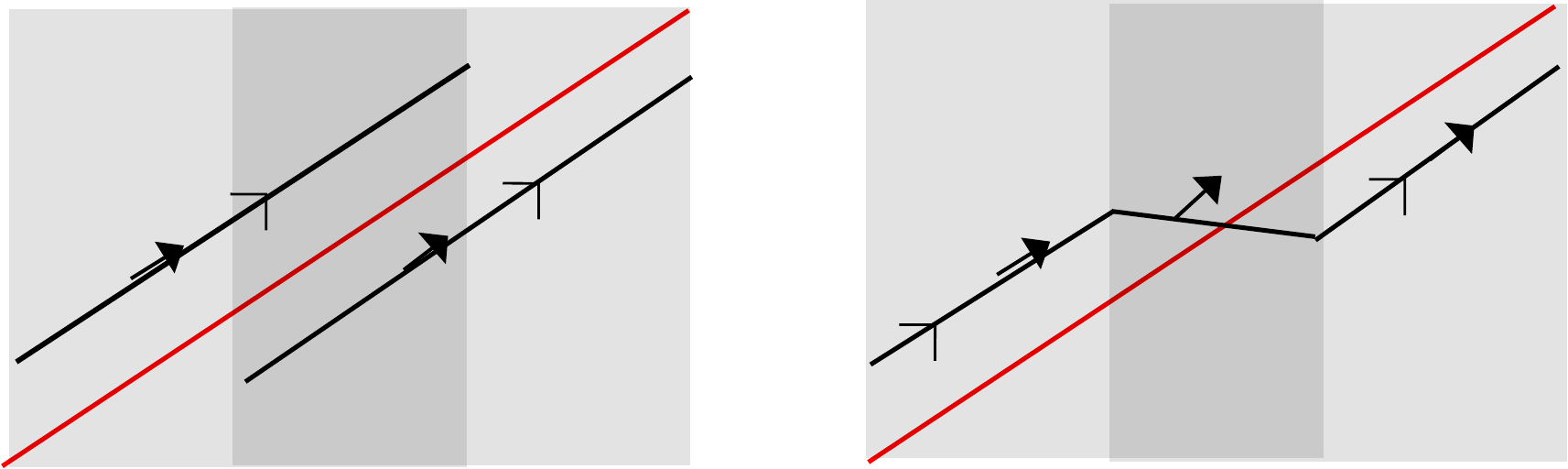}
\put(-360, -4){$A$}
\put(-330, 40){$a_{\alpha}$}
\put(-220, -10){$U_{\beta}$}
\put(-320, -10){$U_{\alpha}$}
\put(-285, 10){$U_{\alpha} \cap U_{\beta}$}
\put(-240, 40){$a_{\beta}$}
\put(-100, 65){$a$}
\put(-100, 10){$U_\alpha \cap U_{\beta}$}
\put(-163, -4){$A$}
\put(-125, -10){$U_{\alpha}$}
\put(-30, -10){$U_{\beta}$}
\caption{The glueing of the cycles $a_{\alpha}, a_{\beta}$ on the overlaps $U_{\alpha} \cap U_{\beta}$ from the proof of Theorem \ref{prop:equivInt}.} 
\label{fig:glueing}
\end{figure}
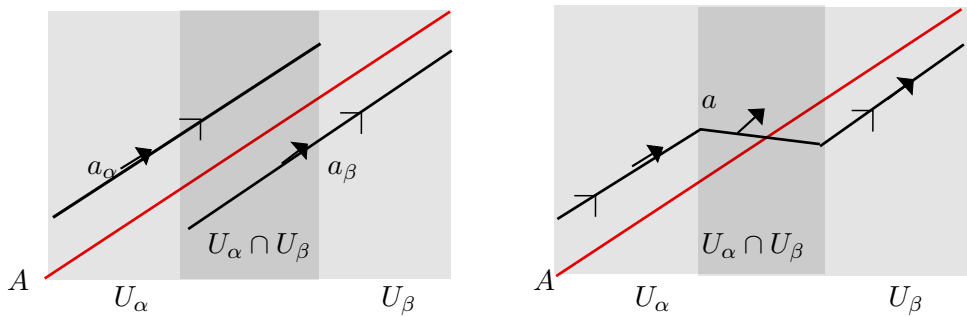

\begin{proof}[Proof of Proposition \ref{prop:11equivInt}]

If at least one of $A$ or $B$ is a boundary cycle find a  rational section  as in \ref{sec:selfbdy}, so that the $1$-cycles intersect transversally in $X$. Then the statement follows immediately by  applying the cycle map. 

We can choose an open covering $\{ U_{\alpha} \}$ of the union   $A \cup  B$  and charts, $\phi_{\alpha}: U_{\alpha} \to P_{\alpha} \subset \T^{N_{\alpha}}$ such that in each chart $A \cap U_{\alpha}$ is contained in a fan cycle  $A_{\alpha}$ and similarly $B \cap U_{\alpha}$ is contained in a fan cycle $B_{\alpha}$.
 Moreover, on the overlaps $U_{\alpha} \cap U_{\beta}$, we can insist that the image of $A \cap U_{\alpha} \cap U_{\beta} $ is a single affine line, and similarly for $B$. 

In each chart $\phi : U_{\alpha} \to P_{\alpha}$ there are $(1, 1)$-cycles $a_{\alpha}$ and $b_{\alpha}$ satisfying the conditions of Lemma \ref{lem:homol}, however, $a_{\alpha}$ and $a_{\beta}$ may not coincide on the intersections $U_{\alpha} \cap U_{\beta}$. On the overlap $U_{\alpha} \cap U_{\beta}$ the cycle $A$ is supported on an affine line.  
By choosing an appropriate covering $\{U_{\alpha}\}$, we can assume that $a_{\alpha}$ and $a_{\beta}$ are supported on $1$-chains parallel to the affine line supporting $A$ in $U_{\alpha} \cap U_{\beta}$. By Lemma \ref{lem:homol} the coefficient of $a_{\alpha}$ in this intersection is also parallel to $A$ and similarly for $a_{\beta}$. 
Now $a_{\alpha}$ and $a_{\beta}$ can be patched together on the overlap $U_{\alpha} \cap U_{\beta}$ by patching together the underlying $1$-chains $a_{\alpha}$ and $a_{\beta}$ with  $1$-chains in the usual way and equipping them with the same coefficient in $\mathcal{F}_1$ as the edge of $A$ to which it is parallel. This can also be done for $B$. Denote the resulting cycles in $X$ by  $a$ and $b$. 

We can assume that $a $ and $b$ intersect transversally in $X$. Moreover, intersection points in the overlaps $U_{\alpha} \cap U_{\beta}$ produced from the patching do not contribute to the intersection  multiplicity of $a$ and $b$. This is because the  coefficients in $\mathcal{F}_1$ of the underlying $(1, 1)$-cells at these new intersection points are parallel, hence the multiplicity of intersection is $0$. 
Therefore, the comparison of the intersection multiplicities $A \cdot B$ with $a \cdot b$ follows from the local comparison in the  previous lemma. 
\end{proof}

The following proposition is a direct extension of \cite[Lemma 4]{ZhaJac}. 

\begin{prop}\label{prop:equivInt}
Suppose $A, A^{\prime}$ are rationally equivalent $1$-cycles in a tropical surface $X$, then $[A], [A^{\prime}]$ are homologous $(1, 1)$-cycles. 
\end{prop}

\begin{corollary}\label{cor:equivInt}
Suppose $A, A^{\prime}, B$ are $1$-cycles in $X$ such that $A $ and $A^{\prime}$ are rationally equivalent, then $A.B = A^{\prime}.B$.
\end{corollary}

\begin{proof}
If $A$ and $A^{\prime}$ are rationally equivalent then $[A], [A^{\prime}]$ are homologous. 
By Proposition \ref{prop:equivInt} $A . B = [A] .  [B] = [A^{\prime}].[B] = A^{\prime}. B$. 
\end{proof}

\section{Constructing tropical surfaces}\label{sec:construction}

In this section we describe $2$ operations for constructing tropical surfaces; modifications and summations. 

\subsection{Tropical Modifications}\label{sec:modGlobal}

\begin{definition}[\cite{MIk3}]\label{def:LinTrop}
Let $X$ and $Y$ be tropical manifolds. A map $f: X \to Y$ is a tropical linear morphism if for every point $x \in X$ there is a neighborhood $U_x$ of $x$ and $U_y$ of $y=f(x)$ with charts $\phi_x:U_x \to V_x \subset \T^{N_x}$, $\phi_y: U_y \to V_y \subset \T^{N_y}$ such that  $\phi_y \circ f \circ \phi_x^{-1}: V_x \subset \T^{N_x} \to V_y \subset \T^{N_y}$  is an integer affine map  $\T^{N_{x}} \to \T^{N_{y}}$. 
\end{definition}

The following is a restriction of modifications of tropical varieties from \cite{MIk3} to tropical manifolds.  

\begin{definition}\label{def:mod}
Let $\tilde{X}$ and $X$ be a pair of tropical manifolds. A tropical modification is a tropical morphism  $\delta: \tilde{X} \to X$ such that, if in borrowing from notation of Definition \ref{def:LinTrop}, at every point $x \in \tilde{X}$ the composition $\phi_x \circ \delta \circ \phi_{\tilde{x}}^{-1}: \tilde{V}_{\tilde{x}}  \subset \T^{N_x+1} \to V_x \subset \T^{N_x}$ is a degree $1$ tropical modification. 
\end{definition}

\begin{definition}
A divisor $D \subset X$ is locally degree $1$ if all of its facets are weight $1$ and there exists an atlas $\{U_{\alpha}, \phi_{\alpha}\}$ with  $\phi_{\alpha}: U_{\alpha} \to V_{\alpha}$ and $\phi_{\alpha}: D \cap U_{\alpha} \to D_{\alpha}$ is an open embedding such that $D_{\alpha} = \trp_{\Delta_{\alpha}}(N_{\alpha})$ and $X_{\alpha} = \trp_{\Delta_{\alpha}}(M_{\alpha})$ for a pair of matroids $N_{\alpha}, M_{\alpha}$. 
\end{definition}

Recall from  Lemma \ref{Cartier}  that every codimension  $1$ tropical cycle in a tropical manifold is a Cartier divisor.

\begin{proposition}\label{prop:modEx}
If $D$ is a locally degree $1$ divisor in a   tropical manifold $X$, then there exists a  tropical manifold $\tilde{X}$ and  a modification $\delta: \tilde{X} \to X$ along  $D$. 
\end{proposition}

\begin{proof}
Let $g = \{g_{\alpha}\}$ be the tropical Cartier divisor such that $\text{div}_X(g) = D$. 
In  each chart $\phi_{\alpha}:U_{\alpha} \to V_{\alpha} \subset \T^{N_{\alpha}}$ there is a degree $1$ modification along $D_{\alpha}$,  given by $\delta_{\alpha}: \tilde{V}_{\alpha} \to V_{\alpha}$ given by the tropical rational  function $g_{\alpha}$. To construct the surface $\tilde{X}$ from the atlas  $\{U_{\alpha}, \phi_{\alpha}\}$ for $X$, first set $\tilde{U}_{\alpha} = \delta^{-1}_{\alpha}(\phi_{\alpha}(U_{\alpha}))$. The sets $\tilde{U}_{\alpha}$ come with charts $\tilde{\phi}_{\alpha}: \tilde{U}_{\alpha} \to \tilde{V}_{\alpha}$, and maps $\tilde{\Phi}_{\alpha \beta} : \phi_{\beta}(\tilde{U}_{\alpha} \cap \tilde{U}_{\beta}) \to \phi_{\alpha}(\tilde{U}_{\alpha} \cap \tilde{U}_{\beta})$.
 The maps $\tilde{\Phi}_{\alpha \beta}$ are restrictions of integer affine functions $\T^{N_{\beta}+1} \to \T^{N_{ \alpha}+1}$ given by

$$(x_1, \dots, x_{N_{\beta}+1}) \mapsto ((\Phi_{\alpha \beta })_1(x_1, \dots, x_{N_{\beta}}), \dots, ( \Phi_{\alpha \beta})_{N_{\beta}}(x_1, \dots , x_{N_{\beta}}), ``x_{N_{\beta}+1} \frac{g_{\alpha}}{g_{\beta}}").$$

The topological space  $\tilde{X}$ is defined by quotienting  the disjoint union $\sqcup_{\alpha} \tilde{U}_{\alpha}$ by the relation $x \sim y$ if $\tilde{\Phi}_{\alpha \beta}  \circ \tilde{\phi}_{\alpha} (x) = \tilde{\phi}_{\beta}(y)$. The above relation is indeed an equivalence relation since the functions $\tilde{\Phi}_{\alpha \beta}$ satisfy:
$$\tilde{\Phi}_{\alpha \beta} \circ \tilde{\Phi}_{\beta \alpha} = Id \quad \text{and} \quad \tilde{\Phi}_{\gamma \alpha} \circ \tilde{\Phi}_{\alpha \beta} = \tilde{\Phi}_{\gamma \beta}.$$
The above relations hold since they hold for $\Phi_{\alpha \beta}$ and the functions $g_{\alpha}$ form a tropical Cartier divisor.
The quotient  $\tilde{X}$ is Hausdorff  and  is equipped with  charts to fan tropical linear spaces. Moreover it satisfies the finite type condition since so does $X$, thus $\tilde{X}$ is a tropical manifold.
\end{proof}

\begin{definition}
Given a tropical modification of manifolds $\delta: \tilde{X} \to X$ there is the pushforward map  $\delta_{\ast}:Z_{\ast}(\tilde{X}) \to Z_{\ast}(X)$ and pullback map $\delta_{\ast}:Z_{\ast}(X) \to Z_{\ast}(\tilde{X})$ on the cycle groups. 
\end{definition}

The definition of these maps in a local chart are given in  \cite[Definition 2.15]{ShawTh}. 
To summarise, the pushforward  of a cycle $\delta_{\ast}A$ is supported on the image $\delta(A)$. The weights of a facet $F$ of $\delta_{\ast}A$ are defined by:
$$w_{\delta_{\ast}A}(F) = \sum_{F_i \subset A,  \delta(F_i) = F} w_A(F_i)[\bar{\Lambda}_{F_i} : \Lambda_{F}],$$ where $\bar{\Lambda}_{F_i}$ is the image under $\delta$ of the integer lattice parallel to  $F_i$ and  $\Lambda_{F}$ is the integer lattice parallel to $F$. 
If the modification $\delta:\tilde{X} \to X$ is along the Cartier divisor $g = \{g_{\alpha}\}$ then the pullback of a cycle $A$  is given by restricting the modification along $g$ to $A$.  
 As in the local case, $\delta_{\ast}\delta^{\ast}$ is the identity map on $Z_{\ast}(X)$. 

\begin{definition}
Let $\delta: \tilde{X} \to X$ be a modification of tropical manifolds, for a cycle $A \subset X$ call $\tilde{A} \subset \tilde{X}$ a lift  of $A$  to $\tilde{X}$ if $\delta_{\ast}\tilde{A} = A$. 
\end{definition}

\begin{thm}\label{ChowMod}
Given a tropical modification $\delta: \tilde{X} \to X$ of tropical manifolds the pushforward and pullback maps induce isomorphisms of the Chow groups. 
In other words, 
$$\delta_{\ast}: {\rm CH}_{k}(\tilde{X}) \to {\rm CH}_{k}(X) \qquad \text{and} \qquad \delta^{\ast}: {\rm CH}_{k}(X) \to {\rm CH}_{k}(\tilde{X}),$$ 
are isomorphisms.  
\end{thm}

\begin{proof}
The modification $\delta: \tilde{X } \to X$ can be  extended to a modification on the products $\delta: \tilde{X} \times \TP^1 \to X \times \TP^1$. For a family $B \subset X \times \TP^1$ we may pullback the cycle $\delta^{\ast}B$ to obtain a family in $\tilde{X} \times \TP^1$. Similarly, the pushforward map can be applied to any family $\tilde{B} \subset  \tilde{X} \times \TP^1$. Therefore, the pushforward and pullback, $\delta_{\ast}$ and $\delta^{\ast}$ descend to homomorphisms on the Chow groups. 

The composition $\delta_{\ast}\delta^{\ast}$ is already the  identity  map on the cycle level.  To establish the isomorphism we must show that for a cycle $A \in Z_k(\tilde{X})$, the cycle $\delta^{\ast}\delta_{\ast}A$ is rationally equivalent to $A$. In other words,  any two \emph{lifts} of a cycle in $X$ to $\tilde{X}$ are rationally equivalent. 
Notice that the difference $ \Delta_{A} = A - \delta^{\ast}\delta_{\ast}A$ is contained in the closure of the faces contracted by the modification $\delta$, so this cycle can be ``translated" in $\tilde{X}$. More precisely, for $t \geq 0 $,  identify $\Delta_A$ with its image in a chart $\phi_{\alpha}: U_{\alpha} \to V_{\alpha} \subset \T^{N_{\alpha}+1}$  and define for $t \geq 0$ a cycle in $\tilde{X}$ given by, 
$$B_t \cap \tilde{V}_{\alpha}  = \{ x  -  t e_{N_{\alpha}+1} \ | \ x \in  \Delta_A \cap \tilde{V}_{\alpha} \}.$$ 

To construct a family  $B$, first over  $[-\infty, \infty]$ take the product $ B_0 \times [-\infty, 0]$, and let the fiber over   $ t \in [0, \infty]$ be $B_t$ from above. Let $\tilde{B}$ be the union of all of these fibers. 
In other words $\tilde{B}$ is the image of a map $\Delta_A \times \TP^1 \to X \times \TP^1$ which is piecewise linear in each chart and the locus of non-linearity is exactly at $\Delta_A \times \{0\} \subset \Delta_A \times \TP^1$. 
Therefore, $\tilde{B}$ violates the balancing condition exactly at its codimension $1$ faces which are facets of $\Delta_A \times \{0\}$.  
For every top dimensional face of $\Delta_A$, add to $\tilde{B}$ the set, 
$$\{ ( x + te_{N_{\alpha}+1} , 0) \ | \ x \in \Delta_A \cap \tilde{V}_{\alpha}
 \text{ and }  t \in [-\infty, 0] \} \subset X \times \{0\}$$  in each chart. Let $B$ denote this union. Equip all facets $B \backslash \tilde{B}$ with weight $-1$ 
then $B$ is  balanced in every chart and is a tropical cycle in $X \times \TP^1$. 
 
By construction,  $B(-\infty) = A  - \delta^{\ast}\delta_{\ast}A$ and since $\delta_{\ast} (A - \delta^{\ast}\delta_{\ast}A) = 0$ we have $B(\infty) = 0$. Therefore, $A$ and $\delta^{\ast}\delta_{\ast}A$ are rationally equivalent and the Chow groups of $\tilde{X}$ and $X$ are isomorphic. 
\end{proof}

\begin{corollary}\label{cor:IntModNum}
Let $\delta:\tilde{X} \to X$ be a tropical modification of  tropical surfaces. Suppose $C_1, C_2$ are  $1$-cycles in $X$, then numerically we have, 
$$ C_1 \cdot C_2 = \delta^{\ast}C_1 \cdot  \delta^{\ast}C_2. $$
Furthermore, if $\tilde{C}_1, \tilde{C}_2$ are $1$-cycles in $\tilde{X}$ then numerically we have, 
$$\tilde{C}_1 \cdot \tilde{C}_2 = \delta_{\ast}\tilde{C}_1 \cdot \delta_{\ast}\tilde{C}_2.$$
\end{corollary}

\begin{proof}
For cycles of sedentarity $0$, the first statement follows immediately from Definition 3.6 of \cite{ShawInt} of the local intersection products in matroidal fans. If one or both of $C_1$, $C_2$ are boundary curves, the statement follows also from the definitions of intersections given in \ref{sec:int1cyc}. 
Now given  lifts $\tilde{C}_1$, $\tilde{C}_2$ of  cycles $C_1$, $C_2$ respectively, then  $\delta^{\ast}C_i$ and $\tilde{C}_i$  are rationally equivalent by the proof of Theorem  \ref{ChowMod}.  Corollary \ref{cor:equivInt}  combined with the previous statement gives
 $\tilde{C}_1\cdot  \tilde{C_2} = \delta^{\ast}C_1 \cdot \delta^{\ast}C_2 = C_1\cdot C_2$. 
\end{proof}

\begin{corollary}\label{cor:canMod}
Let $\delta: \tilde{X} \to X$ be a modification of tropical manifolds, then $\delta_{\ast}K_{\tilde{X}} = K_{X}$, moreover, $K_{\tilde{X}}$ is rationally equivalent to  $\delta^{\ast}K_{X}$. If $\tilde{X}$ and $X$ are surfaces we have equality of intersection numbers  $K_{\tilde{X}}^2 = K_X^2$.
\end{corollary}

\begin{lemma}\label{cor:chern2mod}
Let $\delta: \tilde{X} \to X$ be a modification of tropical surfaces then, 
 $\delta_{\ast}c_2(\tilde{X}) = c_2(X)$,  in particular $\deg({c_2(X)}) = \deg(c_2(\tilde{X}))$. 
\end{lemma}

\begin{proof}
 Notice that for a point $x \in X$ there are at most $2$ points $x_1, x_2 \in \tilde{X}^{(0)}$ such that $\delta(x_i) = x$.  Using Lemma \ref{lem:Chern2mult}, it can  be directly checked that the multiplicity in $c_2(X)$ of a point $x \in X$ is equal to the sum of  the multiplicities 
$m_{x_1}(c_2(\tilde{X}))+ m_{x_2}(c_2(\tilde{X}))$. This proves the lemma. 
\end{proof}

Tropical modifications can be used to  show that any locally degree $1$ curve $C$ in a compact tropical surface satisfies the adjunction formula.

\begin{thm}[Tropical Adjunction Formula]\label{thm:adjunction}
An locally degree $1$ tropical curve $C$ in a compact tropical surface $X$ satisfies
\begin{equation}\label{matAdj}
b_1(C) = \frac{K_X \cdot C + C^2}{2} +1,
\end{equation}
where $b_1(C)$ is the $1^{st}$ Betti number of $C$. 
\end{thm}

\begin{proof}
Suppose $C$ is an irreducible boundary curve of $X$.  Then we may write: $$K_X = K_X^o- C - \sum_{i = 1}^d D_i$$ where the irreducible  boundary curves of $X$ are $C, D_1 , \dots , D_d$. Then (\ref{matAdj}) becomes $$b_1(C)=  \frac{K_X^o \cdot  C - \sum_{i = 1}^d D_i  \cdot  C}{2} +1.$$ This equality is verified by a simple Euler characteristic computation as follows. 
For every leaf $l$ ($1$-valent vertex) of the curve $C$, there is a collection of  boundary curves $D_i$ which meet $C$ at this point. We construct a new graph $G$ from $C$ by adding to each leaf $l$ of $C$ an edge for each boundary curve which meets $C$ at $l$.  Let $V(G)$ and $E(G)$ denote the  edge and vertex sets of the graph $G$, and $\text{val}_G(x)$ denote the valency  of $x \in V(G)$ in $G$.  Definitions \ref{intbddy} and  \ref{2bddy} give, 
\begin{equation}\label{C.K^o}
K_X^o \cdot C = \sum_{x \in V(C)} \text{val}_G(x) - 2 \quad \text{ and }  \quad \sum_{i = 1}^d D_i  \cdot  C =  |L(G)|. 
\end{equation}
 So that, 
  $$K_X^0  \cdot  C - \sum_{i = 1}^d D_i  \cdot  C = \sum_{x \in v(G)} \text{val}_G(x) -2 = 2b_1(G) -2, $$
  and $b_1(C) = b_1(G)$. This proves the statement when $C$ is an irreducible boundary curve. 

When $C$ is not a boundary curve there is a tropical modification $\delta: \tilde{X} \to X$ along $C$ by Proposition \ref{prop:modEx}. The tropical  surface $\tilde{X}$ has a boundary curve  $\tilde{C}$ such that $\delta_{\ast}(\tilde{C}) = C$, and moreover $b_1(C) = b_1(\tilde{C})$.  Since $\tilde{C} \subset \tilde{X}$ is a boundary curve it satisfies the adjunction formula by the above argument. 
 Corollaries \ref{cor:IntModNum} and  \ref{cor:canMod} give $\tilde{C}^2 = C^2$  and  $K_{\tilde{X}} \cdot \tilde{C}  =   K_X \cdot C $. 
Thus $C \subset X$ satisfies the tropical adjunction formula  as well. 
\end{proof}

\begin{rem}\label{rem:adjunction}
The  tropical adjunction formula from Theorem \ref{thm:adjunction} does not hold for tropical curves in general.  In fact, the right hand side of the above formula is neither an upper nor lower bound for the $1^{st}$ Betti number of a tropical curve in a surface. 
Already for fan tropical curves in a fan tropical plane the the right hand side of the equation from Theorem \ref{thm:adjunction} can be greater than or less than $0$,  see  \cite{Br17}. Moreover, the condition that $C$ be locally degree $1$ in $X$ is not a  necessary condition to satisfy the tropical adjunction formula. 
For example the fan tropical curve $C$ from Example \ref{exam:curveDegDecrease} is not locally degree $1$, yet it satisfies the tropical adjunction formula in the compactification of the fan plane $P \subset \R^4$ to $\overline{P} \subset \TP^4$. 
\end{rem}

\vspace{0.5cm}
Lastly, given a tropical modification of manifolds, $\delta: \tilde{X} \to X$,  we show that the tropical $(p, q)$-homology groups of $\tilde{X}$ and $X$ are isomorphic. 
Firstly, there is a map on the chain level  $$\delta_{\ast}: C_{p, q}(\tilde{X}; \Z) \to C_{p, q}(X; \Z).$$
Suppose that a $(p, q)$-cell $\beta_{\sigma}\sigma$ in $C_{p, q}(\tilde{X}; \Z)$ is such that $\sigma$ is contained in a single chart $\tilde{U}_{\alpha} \to \tilde{V}_{\alpha}$ of $\tilde{X}$, and $\delta$ is given 
by $\delta_{\alpha}: \tilde{V}_{\alpha} \to V_{\alpha}$ in that chart. Recall that $\delta_{\alpha}: \T^{N_{\alpha+1}} \to \T^{N_{\alpha}}$ is the extension of the linear map $\R^{N_{\alpha+1}} \to \R^{N_{\alpha}}$ 
which has kernel generated by $e_{N_{\alpha + 1}}$. 
Then $\delta_{\alpha}(\beta_{\sigma}\sigma) =  \beta_{\sigma}^{\prime} \sigma^{\prime}$
 where $\sigma^{\prime} = \delta_{\alpha}(\sigma) \in  X$ and $\beta_{\sigma}^{\prime}$ is the image of the coefficient $\beta_{\sigma}$ under the map $\mathcal{F}^{\Z}_p(\tilde{X}) \to \mathcal{F}^{\Z}_p(X)$ 
 induced by the linear projections $\delta_{\alpha}$. 
In fact, in  each chart the  following diagram commutes:

$$
\xymatrix{
\mathcal{F}^{\Z}_p(E)  \ar[d]^{\delta_{\alpha}} \ar[r]^{\iota_p} &\mathcal{F}^{\Z}_p(E^{\prime}) \ar[d]^{\delta_{\alpha}}\\
\mathcal{F}^{\Z}_p(\delta (E))  \ar[r]^{\iota_p}         &\mathcal{F}^{\Z}_p(\delta (E^{\prime})) }
$$
which  implies 
 $\delta_{\ast} \partial = \partial \delta_{\ast}$ and so the map $\delta_{\ast}$ of chain groups descends to a homomorphism on the tropical $(p, q)$-homology groups.  
 
 For $p= 0$, the map $\delta_{\ast}$ gives an isomorphism on homology since  the $(0, q)$ groups  are just the singular homology  and $X$ is a strong deformation retract of $\tilde{X}$.  
When $p>0$, showing that $\delta_{\ast}$ induces an isomorphism of homology groups, follows the standard argument in basic algebraic topology while keeping track of the coefficients. 

\begin{thm}
Let $\delta: \tilde{X} \to X$ be a tropical modification of tropical spaces, then 
$$H_{p, q}(\tilde{X}; \Z) \cong H_{p, q}(X; \Z)$$ for all $p, q$. 
\end{thm}

\begin{proof}
The tropical space $X$ can be identified as a subspace of $\tilde{X}$ by taking  $f_{\alpha} (V_{\alpha}) \subset \tilde{V_{\alpha}}$ in each chart. We denote this map by $f: X \to \tilde{X}$.  There is a strong deformation retract
$F : \tilde{X} \times I \to \tilde{X}$ such that 
$F: \tilde{X} \times \{0\} \to \tilde{X}$ is the identity and the image of  
$F:\tilde{X} \times \{1\} \to \tilde{X}$ is in $f(X)$. 
Here each  point in $f(X)$ is fixed by $F$ for all $t \in I$. 
Moreover the homotopy respects the stratification of $\tilde{X}$ up until the end when in each chart $U_{\alpha} \to V_{\alpha} \subset \T^{N_{\alpha}+1}$, the points of sedentarity $N_{\alpha}+1$  are sent to $f(D)$.  

For homology with constant coefficients recall the  prism operator 
$P : C_{q}(\tilde{X}; \Z) \to C_{q+1}(\tilde{X}; \Z)$, see Theorem 2.10 \cite{Hatcher} for details. 
For example, in our particular case, if  $\sigma$ is a singular $q$-cell contained in $\tilde{X} \backslash f(X)$  and contained in a chart $\tilde{U}_{\alpha} \to \tilde{V}_{\alpha} \subset \T^{N_{\alpha}+1}$, then the image of the  
prism operator  applied to $\sigma$ is 
$$P(\sigma) = \{x + t e_{N_i +1} \ | \ t \in [0 ,f_{\alpha}(\delta(x))]  \},$$
where $f_{\alpha}$ is the tropical rational function $f_{\alpha}: \T^{N_{\alpha}} \to \T$ defining the modification in the chart $\phi_{\alpha}: U_{\alpha} \to V_{\alpha} \subset \T^{N_{\alpha}}$. 
In homology with constant coefficients, the prism operator satisfies the relation, 
\begin{eqnarray}\label{eqn:prism}
\partial P = id - f_{\ast} - P \partial.
\end{eqnarray}

To extend this to the situation of $(p, q)$-cells, we need only to assign coefficients  in  $\mathcal{F}_p$   to the cells in $P(\sigma)$ 
 in a way that maintains the relation (\ref{eqn:prism}). 
If $\beta_{\sigma} \sigma$ is a $(p,q)$-cell in $\tilde{X} \backslash f(X)$, assign to each face of $P(\sigma)$ the coefficient $\beta_{\sigma}$. 
This is possible since either $\sigma $ and $P(\sigma)$ are contained in the same face of $\tilde{X}$; 
or in a chart $\sigma$ is contained in a face  $E \subset \T^{N_{\alpha}+1}_I$ where $N_{\alpha}+1 \in I$, and $P(\sigma)$ is contained in an adjacent face $E' \subset \T^{N_{\alpha}+1}_{I'}$ where the $I^{\prime} =I  \backslash \{N_i +1\}.$ 
In this second case, there is still an injection  $\mathcal{F}^{\Z}_p(E) \to \mathcal{F}^{\Z}_p(E^{\prime})$, so $\beta_{\sigma} \in \mathcal{F}^{\Z}_p(E^{\prime})$. 

This assignment of coefficients still satisfies the relation in (\ref{eqn:prism}). So by the usual argument in algebraic topology the map  $\delta_{\ast}$ induces an 
isomorphism of the $(p,q)$-homology groups.  
\end{proof}

\subsection{Summation of tropical surfaces}

Given tropical surfaces $X_1, X_2$ each with boundary curves $C_1, C_2$ satisfying certain conditions we define their tropical sum, which is a new tropical surface obtained  by ``glueing" together $X_1$ and $X_2$. 
This construction is analogous to the fibre sum for manifolds or the symplectic sum in the symplectic category \cite{Gompf}.  

The tropical sum of $X_1$ and $X_2$  depends firstly on a chosen identification of the boundary curves $C_1, C_2$ as well an orientation reversing isomorphism of their normal bundles. This information determines a $1$-parameter family of tropical surfaces, similar to the situation for symplectic sums. 

Recall,  that an abstract tropical curve is a graph with a complete inner  metric defined away from the $1$-valent vertices, see \cite[Section 7.2]{BIMS}.

\begin{definition}
An isomorphism of tropical curves is an isometry of metric graphs $f:C_1 \to C_2$.
\end{definition}

\begin{figure}
\includegraphics[scale=1]{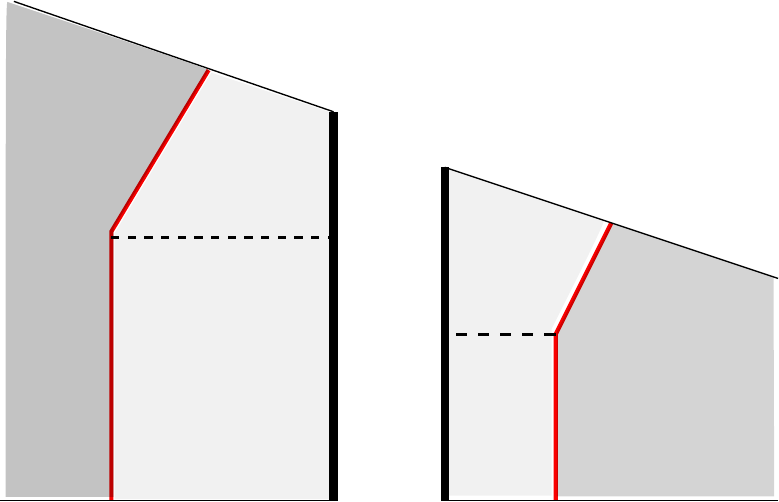}
\put(-170, 60){$V(\sigma_1)$}
\put(-90, 60){$V(\sigma_2)$}
\put(-170, 130){$\sigma_1$}
\put(-55,90 ){$\sigma_2$}
\put(-135,-10 ){$C_1$}
\put(-100,-10 ){$C_2$}
\put(-220,-10 ){$X_1$}
\put(-20,-10 ){$X_2$}
\put(-125,70 ){$a_1$}
\put(-110,47 ){$a_2$}
\hspace{2cm}
\includegraphics[scale=1]{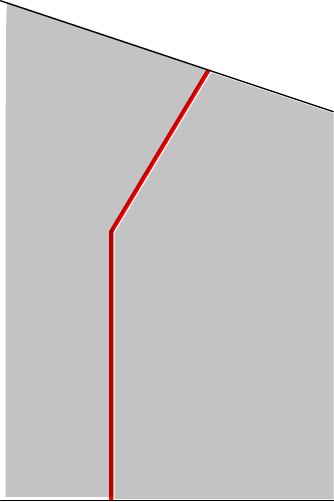}
\put(-80,-10 ){$X_1 \#_{f_1, f_2,  g, \sigma} X_2$}
\caption{A depiction of the sum of $2$ surfaces $X_1, X_2$ along $2$ curves $C_1, C_2$ where $f_i : \TP^1 \to C_i$ sends $0 \mapsto a_i$ and 
$\tilde{g}(\sigma_1)  = \sigma_2$. }
\label{fig:sumNeighborhoods}
\end{figure}
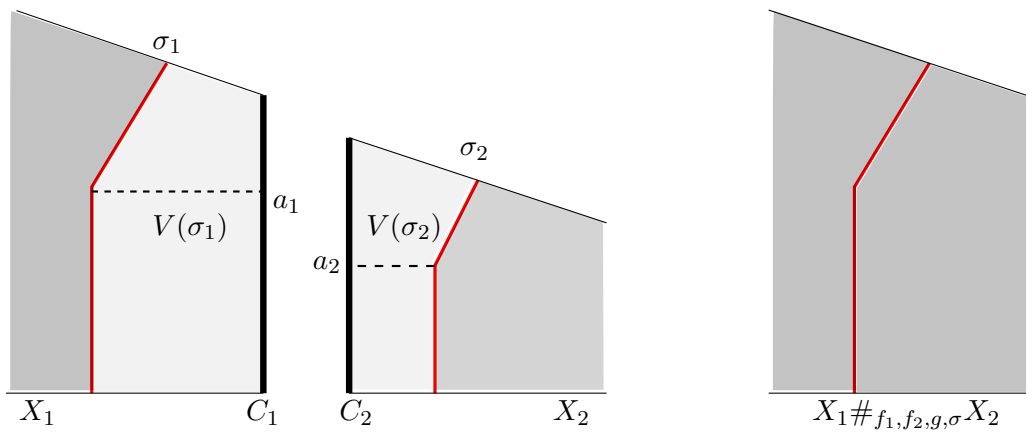

\comment{
\begin{figure}
\begin{center}
\vspace{0.5cm}
\put(-32, -5){$D^o$}
\put(-80, 55){$B(D)$}
\end{center}
\caption{A section $\sigma$ from Lemma \ref{lem:ratsection}  of the normal bundle $N_X(D)$ where  $N_X(D^o)$ is trivial. \label{MoveBddy}} 
\end{figure}}

To define the tropical sum we will use the normal bundle of a tropical boundary curve in a surface from Definition \ref{def:normalbundle}. 
The isomorphism classes of tropical line bundles on a curve form the tropical  Picard group $\text{Pic}(C)$  \cite{MikZha:Jac}. 
Two tropical line bundles on a curve $C$ are inverses to each other if they are inverses in the tropical Picard group. 
In particular, if a tropical line bundle $\pi: L \to C$  is defined by transition functions  $\{ f_{\alpha \alpha'}\}$ on some trivialization  $\{U_{\alpha}\}$ then its inverse bundle $L^{-1}$ can be given by the transition functions $ \{ - f_{\alpha \alpha'}\}$. 
Let $N^o$ denote the complement of the $-\infty$ section in a tropical line bundle $L$. Then $L^o$ is an $\R$ bundle over $C$ and  there is a tropical bundle isomorphism $r: L^o \to (L^{-1})^o$ given in each fiber by $x \mapsto -x$

Recall that a tropical section is  $\sigma: C \to L$, such that in each chart $\sigma$ is given by a  function which is a bounded piecewise affine function. For a section $\sigma$, let $V(\sigma) \subset L$ be the neighborhood of the $-\infty$ section defined by $\{(x, t ) \ | \ t < \sigma(x)\}$ in each trivialization, see the left side of Figure \ref{fig:sumNeighborhoods}. The subset $V(\sigma)$ will play the role of a tubular neighborhood of the $-\infty$ section in what follows.

\begin{definition}[\textbf{Tropical sum}]\label{def:SUM}
Let $X_1$ and $X_2$ be tropical surfaces, and $C$ a  tropical curve. Fix tropical morphisms  $f_i : C \to X_i$ whose images are boundary curves $C_i$ of $X_i$ having simple normal crossings with the other boundary curves for $i = 1, 2$. 
Let $g: N_{X_1}(C) \to N_{X_2}(C)^{-1}$ be an  bundle isomorphism and $\tilde{g}: N^o_{X_1}(C) \to N_{X_2}(C)^o$ be the orientation reversing $\R$ bundle isomorphism given by composing $g$ with $x \mapsto -x$ in each fiber. 

Fix sections $\sigma_i: C\to N_{X_i}(C)$ satisfying $\tilde{g} \circ \sigma_1 = \sigma_2$  and such that there are  neighborhoods $\overline{V(\sigma_i)}$ can be identified with neighborhoods of $C_i$ in $X_i$. 
 The tropical sum is obtained by identifying $X_1 \backslash V(\sigma_1)$ and $X_2 \backslash V(\sigma_2)$ along $\sigma_1$ and $\sigma_2$ by way of the map  $\tilde{g}$, that is  
$$X_1 \#_{ f_1, f_2, g, \sigma_1} X_2 = 
\frac{X_1 \backslash V(\sigma_1) \sqcup X_2 \backslash V(\sigma_2) }{\sigma_1 \sim \tilde{g}\circ \sigma_1}.$$
\end{definition}

\begin{proposition}
The sum $X_1 \#_{f_1, f_2, g, \sigma_1}  X_2$ from Definition \ref{def:SUM} is a tropical surface. 
\end{proposition}

\begin{proof}
It is possible to choose sufficiently small neighborhoods $U_1, U_2$ of $X_1 \backslash V(\sigma_1), X_2 \backslash V(\sigma_2)$ respectively, so that the tropical sum is also the quotient of the disjoint union $U_1 \sqcup U_2$ given by identifying points by way of  $\tilde{g}$. Then in the standard way, we can construct an open cover $\{U_{\alpha}\}$ of the sum from open coverings of $X_1$ and $X_2$ restricted to $U_1$ and $U_2$. 
This cover comes with charts $\phi_{\alpha}$ and coordinate changes on the overlaps $U_{\alpha} \cap U_{\beta}$ either induced by the coordinate changes from $X_1$ or $X_2$ or from $\tilde{g}$. The quotient is a Hausdorff topological space. 

To verify the finite type condition, let $\{W_i \}$ be the finite open cover of $X_1$ and $\{Y_j\}$ be a finite open cover of $X_2$. Simply intersecting $W_i \cap U_1$ and $Y_j \cap U_2$ may not provide a 
finite open cover satisfying the finite type condition. 
For example, if $\overline{\phi_{\alpha}(W_i \cap U_1)} \not \subset \phi_{\alpha}(U_{\alpha} \cap U_1)$ then $W_i \cap \bar{U}_1 \backslash U_1$.  The image of $\phi_{\alpha}(W_i \cap U_1) \subset \T^{N_{\alpha}}$ can be shrunk 
by an arbitrarily small amount $\epsilon$ so that it no longer intersects the image of $  \bar{U}_1 \backslash U_1$. Pull this set back to obtain a new open set $W_i' \subset U_{\alpha} \cap U_1$. This can be done for each $W_i$ and $Y_j$ by a  sufficiently 
small $\epsilon$ so that we still  obtain an open cover of the sum which now satisfies the finite type condition. Therefore the sum operation produces a new tropical surface. 
\end{proof}

\vspace{0.5cm}

\begin{figure}[b]
\begin{center}
\includegraphics[scale= 0.3]{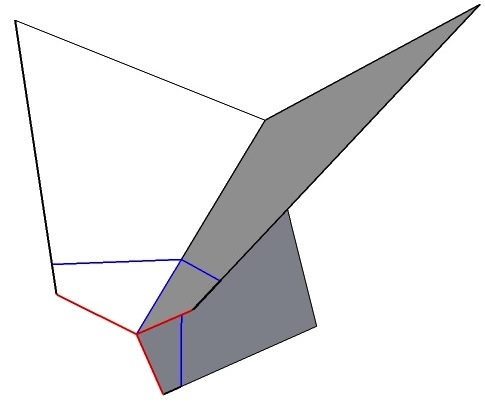}
\hspace{0.5cm}
\includegraphics[scale= 0.28]{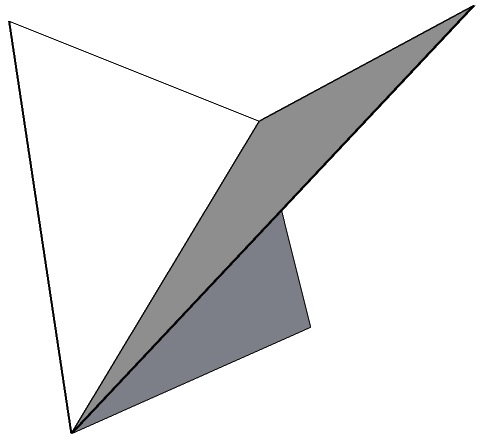}
\put(-30, 0){$X \# \overline{V}$}
\put(-200, 0){$X$}
\put(-280, 10){$E$}
\put(-275, 50){\small{$i(\sigma)$}}
\end{center}
\caption{ A neighborhood of a $-1$-fan curve $E$ in a surface $X$, and the result after summing with $\overline{V}$.  \label{cornerBlowup}} 
\end{figure}

\begin{example}[Contracting rational $(-1)$-curves]\label{ex:-1}

Let $E \subset X$ be an irreducible boundary curve of a surface $X$ with simple normal crossings with the other boundary curves.  Suppose that $E$ has $k$-leaves, $E^2 = -1$ and that $b_1(E) = 0$, so that $E$ is a rational tropical curve.  
Let $E^o$ denote the complement of the points of sedentarity order $2$ in $E$. The normal bundle restricted to $E^o$ is $$N_{X}(E^o)= L^o \times \T \subset \R^{k-1} \times \T.$$ Where $L^o$ is a tropical line  in $\R^{k-1}$. 
The graph of a section $\sigma: E \to N_{X}(E)$ has $k$ unbounded edges %
identifying $N_{X}(E^o) \subset \R^{k-1} \times \T$ then these outgoing rays have primitive integer directions: $-e_1, \dots , -e_{k-1}, \text{ and }  e_1 + \dots + e_{k-1} - e_k$.

To ``contract" $E$, perform a tropical sum with a tropical plane $V \subset \TP^k$ which is the cone over a tropical line $L$ isomorphic to $E$ contained in the hyperplane $\TP^k_{\{0\}}$ of $\TP^k$. Since $V$ is a cone (has no internal vertices) the resulting sum is independent of the choice of section and identification of $E$ with $L$. 
Denote the sum $Y = X \#_{E, L} V$, and by $y$ the point of $Y$ corresponding to the cone point of $V$. There is a tropical morphism $f: X \to Y$ which is an isomorphism restricted to 
$X \backslash E \to Y \backslash y$.

Let $\tilde{\T}^k$ denote the tropical toric variety corresponding to the toric blow up of affine space at the origin. The fan $\Sigma$ producing this toric variety has for example, $1$-dimensional  rays $-e_1, \dots , -e_N$ and $-e_0 = - \sum_{i=1} e_i$.  Let $\tilde{E}$ denote the boundary divisor of $\tilde{\T}^k$ corresponding to the ray in direction $-e_0$. 
There is a tropical linear morphism $\tilde{f}: \tilde{\T}^k \to \T^k$ which is the identity when restricted to $\tilde{\T}^k\backslash \tilde{E}  \to \T^k \backslash -\infty$. A neighborhood $U$ of $E$ in $X$ has an open embedding  $\tilde{L} \subset \tilde{\T}^k$, where $\tilde{L} \cap \tilde{\T}^k\backslash \tilde{E}$ is a tropical linear space in $\T^k \backslash -\infty$. 
The tropical morphism $f: X \to Y$ is just the identity outside of a neighborhood of $E$ and the restriction of $\tilde{f}$ on a neighborhood of $E$. 

If $E \subset X$ is not necessarily a boundary curve but still a locally degree $1$ curve in $X$ with $E^{2} = -1$ and $b_1(E) = 0$, then perform a   non-singular modification,  $\delta: \tilde{X} \to X$  along  $E$. Then $\tilde{X}$ has  a boundary divisor $\tilde{E}$ with $\tilde{E}^2 = -1$ and we may take the  sum  $\tilde{X}$ with a tropical plane $V$ along $\tilde{E}$ and $L$ as above. Therefore, up to tropical modification, we can contract  rational curves of self-intersection $-1$, which are locally degree $1$ in  $X$. 

This gives a  partial tropical version of the Castelnuovo-Enriques criterion for blowing down $(-1)$-curves in classical algebraic geometry. Recall from Example \ref{exam:curveDegDecrease} that the local degree of a tropical curve may be decreased upon performing modifications. However, there are examples of rational $(-1)$-curves in $\CP^2$ which are not rectifiable. The tropicalization of such a $(-1)$-curve could not be contracted in this way. 
\end{example}

\begin{example}\label{Example:SelfSum}
The tropical Hirzebruch surface $X_k$ of degree $k$ can be constructed following Example \ref{ex:toricman}. Then $X_k$ is a compactification of $\R^2$ with four boundary curves all of which are isomorphic to $\TP^1 = [-\infty, \infty]$ and satisfying $C_1^2 = C_3^2 = 0$ and $C_2^2 = -C_4^2 = k$. Moreover,  $C_1 \cap C_3 = C_2 \cap C_4 = \emptyset$. These four divisors correspond to the directions in $\R^2$
$$v_1 = (-1, 0), \qquad v_2 = ( 0, -1), \qquad  v_3 = ( 1, k), \qquad \text{and} \qquad v_4 = (0, 1),$$ 
in $\R^2 \subset X_k$.
The tropical self-sum of $X_k$ can be formed along either pair $C_1, C_3$ or $C_2, C_4$. 

Fix identifications $f_1: \TP^1 \to C_1 \subset X_k$ and $f_3 : \TP^1 \to C_3 \subset X_k$, so the underlying topological space of the sum is a cylinder.  
For any choice of $g$ and $\sigma$, the self-sum $(X_k)\#_{f_1, f_3, g, \sigma}$ produces a tropical projective  line bundle over a tropical elliptic curve $E$, (which is just a circle equipped with a length). The length of $E$ is determined by the choice of $g$ and $\sigma$. The degree of a section of this bundle is equal to $k$.

Precomposing the identification $f_1: \TP^1 \to C_1$ with the map $x \mapsto -x$ to obtain $f_1': \TP^1 \to C_1$, then the sum $(X_k)\#_{f_1', f_3, g, \sigma}$ is diffeomorphic to  a M\"obius band with boundary.  The result is still a tropical projective line bundle over an elliptic curve, $E'$ which is a circle now of length twice the length of $E$. 
The degree of a section of the bundle is $0$. 

We may also form  the self-sum along the tropical divisors $C_2$ and $C_4$ after fixing identifications $f_2: \TP^1 \to C_2 \subset X_k$, $f_4 : \TP^1 \to C_4 \subset X_k$ to obtain a 
tropical surface 
$Y = (X_k)\#_{f_2, f_4, g, \sigma}$. 
Again the underlying topological  space is either a cylinder or M\"obius band 
but depending the choices of $f_2, f_4, g$ and $\sigma$ it may or may not be fibered by tropical elliptic curves. These are examples of tropical models of Hopf surfaces, see \cite{RuggShaw}.
\end{example}

\begin{figure}

\includegraphics[scale=0.01]{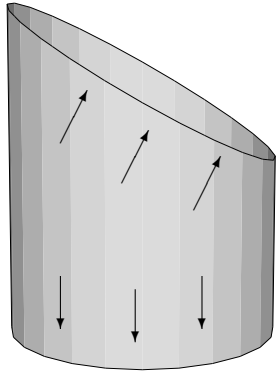}
\put(-35,40){\vector(0,-1){20}}
\put(-60,35){\vector(0,-1){20}}
\put(-88,40){\vector(0,-1){20}}
\put(-38,65){\vector(1,2){10}}
\put(-65,75){\vector(1, 2){10}}
\put(-88,90){\vector(1, 2){10}}
\caption{The self-sum of a tropical Hirzebruch surface along the $-\infty$ and $\infty$ sections.}
\label{fig:sumHirz}
\end{figure}

\begin{example}[Logarithmic transformations]

Consider the tropical surface $Y$ obtained from quotienting  $\T \times \R$ by the extension of a translation  $\R^2 \to \R^2$.  
If the translation is by $(a, b)$, where  $a/b$ is rational then $Y$ is fibered over $\T$ by tropical elliptic curves.
 
The elliptic curves over $\R \subset \TP^1$ are all circles of length $l$ where $(a, b)= l (n, m)$ for $n, m$ relatively prime integers. 
The boundary curve is also a circle but instead it has length $b$, since the restriction of the translation to the boundary of $\T \times \R$ is given by $(-\infty, x) \mapsto (-\infty, x+b)$.  The boundary curve behaves like a fibre of multiplicity $m$, where $b = ml$.

Suppose $X$ is another tropical surface with a  boundary curve $C$, that is again a tropical elliptic curve which is a circle of length $l$. Suppose moreover that the normal bundle $N_X(C)$ is trivial (i.e.~$N_X(C) = C \times \T$). 
Denote by  $X \#_C Y$ the surface obtained by removing a neighborhood of $C$ in $X$ and glueing in a copy of $Y$ by the identification of fibers, as in the tropical sum.  
Then $X \#_C Y$ is an example of a tropical logarithmic transformation of $X$ performed at $C$.

\end{example}

The description  of a tropical surface as a sum of other surfaces provides a means for calculating some invariants. As a first example, there is a Mayer-Vietoris sequence for tropical $(p, q)$-homology groups which can be applied to the sum of tropical surfaces. In the next section Lemma \ref{lem:SumChern} relates the Chern numbers of tropical sums of surfaces.

\begin{proposition}\cite[Proposition 4.2]{ShawHomo}
Let $X$ be a tropical manifold and $A, B \subset X$ open subsets covering $X$, then the following sequence of tropical $(p,q)$-homology groups is exact, 
$$\dots \to H_{p, q}(A\cap B) \to H_{p,q}(A) \oplus H_{p,q}(B) \to H_{p,q}(X) \to H_{p,q-1}(A \cap B) \to \dots $$

\end{proposition}

To apply the above Mayer-Vietoris sequence in the case of a sum of tropical surfaces, take $A = X_1 \backslash D$ and $B = X_1 \backslash D$, so that $A \cap B = N^o_{X_1}(D)$.

\section{Noether's formula}\label{sec:Noether}

Classically, Noether's formula for  a non-singular projective surface $\X$  relates its holomorphic Euler characteristic with its $1^{st}$ and $2^{nd}$ Chern classes \cite{GriffithsHarris}, 
$$\chi(\mathcal{O}_{\X}) = \frac{K_{\X}^2 + c_2(\X)}{ 12 }. $$ 
Here $\chi(\mathcal{O}_{\X})$ is the holomorphic Euler characteristic of $\X$,  $K_{\X}$ is its  canonical class of $\X$ and $c_2(\X)$ is its second Chern class. This formula is a special case of the  Hirzebruch-Riemann-Roch formula for surfaces, which states that for a line bundle $\Ln$ on a compact complex surface $\X$:
$$\chi(\Ln) = \chi(\mathcal{O}_{\X}) + \frac{\Ln.\Ln - \Ln.K_{\X}}{2}.$$

By the relation of tropical $(p, q)$-homology to Hodge numbers from \cite{IKMZ}, (see also \cite[Theorem 7.34]{BIMS}),  we replace the holomorphic Euler characteristic $\chi(\X) = \sum_{p=0}^N h^{0,n}(\X)$ by the  topological Euler characteristic of the tropical surface $X$.

\begin{thm}\label{prop:Noether}
A compact  tropical surface $X$ obtained from tropical toric surfaces via summations and modifications satisfies, 
$$\chi(X) = \frac{K_X^2 + c_2(X)}{12},$$
where $\chi(X)$ is the topological Euler characteristic of $X$. 
\end{thm}

 To prove the above theorem we require a series of lemmas. 
 
\begin{lemma}\label{lem:minusprods}
Suppose $X$ is a tropical sum of $X_1$ and $X_2$ along the irreducible boundary divisors $C_1 \subset X_1$ and $C_2 \subset X_2$ then,
\begin{enumerate}
 \item{$ (K^o_X)^2 = (K^o_{X_1})^2 + (K^o_{X_2})^2$;}
 \item{$ K^o_X \cdot \partial X = K^o_{X_1}\cdot\partial X_1 + K^o_{X_2}\cdot \partial X_2 - 2 \sum_{v \in V(C_1^o)}(\val_{C_1}(v)-2)$;}
\item{$ C_1^2 =  - C_2^2$.}
 \end{enumerate}

\end{lemma}

\begin{proof}
The intersection $(K^o_X)^2$ is supported on points of $X^{(0)}$ that are of  sedentarity $\emptyset$ and also points of sedentarity, which are the intersection of more than $2$ boundary divisors of $X$. Such a  point is in exactly one of the skeletons  $X^{(0)}_1$ or $X^{(0)}_2$ 
by the assumption that $C_1, C_2$ have simple normal crossings with the other boundary divisors of $X_1$ and $X_2$ respectively. Therefore  $(1)$ follows from local intersection multiplicities.

For each of the surfaces there is the equality, 
$$K^o_X \cdot \partial X = \sum_{D \in \A} \sum_{v \in V(D^o)} (\val_D(v)-2) + \sum_{x \ s.t. \ S(x) = 2} m_x(K^o_X \cdot \partial X).$$
When $X$ is a tropical sum of $X_1$ and $X_2$ there is a contribution of $ \val_{C_i}(v)-2$ for an internal vertex  of $C_i$ to  $K^o_{X_i} \cdot  \partial X_i$ for $i = 1, 2$ which does not appear in $K_X^2 \cdot \partial X$, otherwise the expressions are the same. 
Since $C_1$ and $C_2$ are isomorphic the equality in $(2)$ follows.

Part $(3)$ follows simply from Definition \ref{def:boundaryintersection} and that  the sections  $\sigma, \sigma^{\prime}$ of the normal bundles of $C_1$, $C_2$ respectively are related by  $\sigma^{\prime} =  g\circ (\sigma)$ where $g$ is orientation reversing. 
\end{proof}

\begin{lemma}\label{lem:sumsof2}
Suppose $X$ is  a tropical sum of $X_1, X_2$ along boundary curves $C_1 \subset X_1$, $C_2 \subset X_2$. Denote the boundary arrangement of $X_i$ by  $\A_i$ for $i=1, 2$ and the boundary arrangement of $X$ by $\A_X$. Then, 
$$(\sum_{\tilde{D}_i \in \A_X} \tilde{D}_k) ^2 = (\sum_{E_i \in \A_1} E_i)^2 + (\sum_{F_j \in \A_2} F_j)^2 - 4|l (C)| , $$
where $l(C)$ is the set of leaves of $C \cong C_i$. 
\end{lemma}

\begin{proof}

The section $\sigma$ from the construction of the tropical sum is a subset of $X$. 
For every $\tilde{D}_k \subset \A_X$, we can find a cycle $\tilde{D}^{\prime}_k$ rationally equivalent to $\tilde{D}_k$ and contained in a neighborhood of $\tilde{D}_k$, by taking a section in the sense of Definition \ref{def:boundaryintersection} and completing it to a tropical cycle. We can also insist that $\tilde{D}_k$  intersects $\tilde{D}'_k$ in a finite number of 	points 
in the interior of $1$ dimensional faces of both $\tilde{D}'_k$ and $\tilde{D}_k$ and in the interior of a top dimensional stratum  of $X$. We can also suppose that at each point $x \in \tilde{D}_k \cap \tilde{D}^{\prime}_k$ the tropical intersection multiplicity is $1$, and similarly for the intersection with $\tilde{D}^{\prime}_k$ and $\sigma$. 
Denote  $\textbf{D} = \sum \tilde{D}^{\prime}_k$.

The two connected components $V_1, V_2$ of  $X \backslash \sigma$ are identified with open subsets of $X_1$ and $X_2$ respectively. By  intersecting $\textbf{D}$ with  $V_1$ and extending any open  rays of $\textbf{D} \cap V_1$ to the rest of  $X_1$ in the direction of $C_1 $ we produce a cycle $\textbf{E}$  rationally equivalent to $\sum_{E_i \in \A_1} E_i- C_1$. Moreover, $\textbf{E}. C_1 = |l (C_1)|$ since $C_1$ has simple normal crossings with the rest of the divisors of $X_1$ and the cycles $\tilde{D}^{\prime}_k$ are chosen so that they had intersection multiplicity $1$ with $\sigma$.  The analogous statements hold in $X_2$  using $\textbf{F}$ to denote the cycle $\sum_{F_j \in \A_2} F_j- C_2$. 

By the local intersection multiplicites for cycles of sedentarity $\emptyset$,  we obtain $\textbf{D}^2  = \textbf{E}^2 + \textbf{F}^2$. Now,  $ (\sum_{E_i \in \A_1} E_i)^2   = (\textbf{E} + C_1)^2$ 
and similarly $ (\sum_{F_j \in \A_1} F_j)^2   = (\textbf{F} + C_2)^2$. 
Using that $C_1^2 = -C^2_2$ and $\textbf{E} \cdot C_1 = \textbf{F} \cdot C_2 = |l (C_1)|$, the lemma then follows by equivalence of intersection numbers for rationally equivalent cycles from Proposition \ref{cor:equivInt}.
\end{proof}

\begin{lemma}\label{lem:SumChern}
Let $X_1, X_2$ be compact tropical surfaces and $C_1 \subset X_1$, $C_2 \subset X_2$ boundary curves such that the tropical sum along $C_1, C_2$ exists. Suppose $X = X_1 \#_f X_2$, then
\begin{enumerate}
\item $c_2(X) = c_2(X_1) + c_2(X_2) + 2K_C,$;
\item  $K_X^2  =   K_{X_1}^2 + K_{X_2}^2 + 4K_C$,
\end{enumerate}
where $K_C$ is the canonical class of the curve $C \cong C_i$. 
\end{lemma}

\begin{proof}
A comparison of vertices in $X_1$, $X_2$ and $X$ gives the  first formula for $c_2(X)$ 
\begin{eqnarray*}
c_2(X) & = & c_2(X_1) + c_2(X_2) - 2 \sum_{v \in V(C)} (2-\text{val}_C(v)) \\ 
&= & c_2(X_1) + c_2(X_2) + 2K_C,
\end{eqnarray*}
where  $V(C)$ is the vertex set of the curve $C$ and $\text{val}_C(v)$ is the valency of a vertex. 

Recall that  $K_X = K_X^o - \partial X$ and apply  Lemmas \ref{lem:minusprods} and  \ref{lem:sumsof2}  to obtain:
\begin{eqnarray*}
 K_X^2 - (K_{X_1}^2 +K_{X_2}^2 ) & = & -  2[  \sum E_i \cdot C_1 + \sum F_i \cdot C_2 - K^o_{X_1} \cdot C_1 - K^o_{X_2} \cdot C_2 ] \\
& = & 4\sum_{v \in V(C)} (\text{val}_C(v) -2) \\
&=& 4K_C.
\end{eqnarray*}
This proves the lemma. 
\end{proof}

A similar sequence of lemmas can be proved when $X$ is the result of a self-sum of a surface $X'$ as  in Example \ref{Example:SelfSum}. We summarize the statement in a single lemma below. The proof is omitted since it follows the same lines as for the sum of distinct surfaces. 
  
\begin{lemma}\label{lem:selfsumChern}
Let $X$ be the result of the self-sum of a tropical surface $X'$ along disjoint boundary curves $C_1, C_2$, then 
\begin{enumerate}
\item $c_2(X) = c_2(X') + 2K_C,$;
\item  $K_X^2  =   K_{X'}^2 + 4K_C$,
\end{enumerate}
where $K_C$ is the canonical class of the curve $C \cong C_i$. 
\end{lemma}

The next corollary immediately follows the above lemma.

\begin{corollary}\label{cor:Sum}
Let $X_1, X_2$ be compact tropical surfaces and $C_1 \subset X_1$, $C_2 \subset  X_2$  irreducible boundary curves such that the tropical sum $X$ along $C_1, C_2$ exists. Then, 
$$K_X^2 - 2c_2(X) = K_{X_1}^2 + K_{X_2}^2 - 2[ c_2(X_1) + c_2(X_2)].$$
\end{corollary}

Hirzebruch's signature formula states that for a compact complex surface $\X$, 
$$3\textbf{Sign}(\X)  = K^2_{\X} - 2c_2(\X),$$ 
where $\textbf{Sign}(\X)$ is the signature of the intersection form on $H^2(\X)$. Once again, this formula is a special case of the Hirzebruch-Riemann-Roch formula applied to the exterior algebra bundle of the cotangent bundle of a manifold, see Section $2$ of \cite{Hirzebruch:Signature}.  By Novikov additivity, the signature of $4$-manifolds is additive under cobordism \cite{Kirby:4Manifolds}, and the above corollary shows that tropically the  right hand side of Hirzebruch's formula is also additive under taking tropical sums of surfaces. This leads to the following conjecture.

\begin{conj}\label{conj:Hirz}
A compact tropical surface $X$ satisfies 
$$3\textbf{Sign}_{1, 1}(X) = K^2_{X} - 2c_2(X),$$
where $\textbf{Sign}_{1, 1}(X) $ is the signature of the intersection form on $H_{1, 1}(X)$. 
\end{conj}

\begin{rem}
In the above conjecture we could equivalently replace $\textbf{Sign}_{1, 1}(X)$ by $\textbf{Sign}(X)$ which is the signature of the intersection pairing on $H_{2, 0}(X) \oplus H_{1,1}(X) \oplus H_{0, 2}(X)$. The intersection pairing on $(0, 2)$ and $(2, 0)$ classes is defined in \cite{MikZhaEig}.  

\end{rem}

\begin{proof}[Proof of Theorem  \ref{prop:Noether}]

 For a tropical toric variety we have $\chi(X) = 1$.  The translation of  Noether's formula for toric surfaces to properties of polygons is quite simple and can be found in \cite[Section 4.3]{Ful} 

If $f : \tilde{X} \to X$ is a tropical modification we have  $\chi(\tilde{X}) = \chi(X)$ since $\tilde{X}$ is homotopic to $X$. Combining this with  Lemmas  \ref{cor:canMod} and \ref{cor:chern2mod} proves Noether's formula for $\tilde{X}$.

\comment{
$K_X^2 = K_{\tilde{X}}^2$,  first apply \textcolor{red}{Lemma \ref{EqCanClass}}. Then it is simple to check that, we have $\delta^{\ast}D_i.\delta^{\ast}D_j = D_i.D_j$ and $\delta^{\ast}K^o_X.\delta^{\ast}D_i = K_X^o.D_i$. 
To see that $c_2(X) = c_2(\tilde{X})$, suppose $D \subset X$ is the divisor of the modification $\delta$.  If $x \not \in D^{(0)}$, then the contribution of $x$ to $c_2(X)$ is the same as the contribution of $\delta^{\ast}x$ to  $c_2(\tilde{X})$.  Otherwise, we may check that
$$m_x(X) = m_{\delta^{\ast}x}(\tilde{X}) + m_{x^{\prime}}(\tilde{X}), $$
\textcolor{red}{where $x^{\prime}$ is the point corresponding to $x \in D^{(0)} \subset X$ in the corresponding boundary }divisor $\tilde{D} \subset \tilde{X}$. Therefore, $c_2(X) = c_2(\tilde{X})$ and the tropical Noether formula holds for $\tilde{X}$. }

Suppose $X_1$, $X_2$ both satisfy tropical Noether's formula, and let $X$ denote the sum of $X_1$, $X_2$ along $C_1$, $C_2$ for some choices of $f_1, f_2, g, \sigma_1$. Apply both parts of  Lemma \ref{lem:SumChern} to obtain:
\begin{eqnarray*}
K_X^2 + c_2(X) &= &K_{X_1}^2 + c_2(X_1) + K_{X_2}^2 + c_2(X_2) + 6K_C \\
 & = &  12\chi(X_1) + 12 \chi(X_2) + 6K_C \\
 & = & 12\chi(X_1) + 12 \chi(X_2)  - 12\chi(C),
\end{eqnarray*}
where $C \cong C_i$. 
By the formula for Euler characteristics of non-disjoint unions Noether's formula holds. When $X$ is a self-sum of $X'$ along disjoint boundary curves applying Lemma \ref{lem:selfsumChern} establishes the tropical Noether's formula for the self-sum. 
This proves the theorem.  
\end{proof}

\small
\def\rightmark{\em Bibliography}

\bibliographystyle{alpha}
\bibliography{/homes/combi/shaw/Dropbox/Bib/BiblioMain.bib}

\end{document}